%% file: paper.tex
\documentclass[11pt,a4paper]{article}
\usepackage{amsmath,amssymb,amsthm,epsfig,yfonts,stmaryrd}

\usepackage{diagrams}
\usepackage{pinlabel}
\usepackage{rotating}

\include{fontsforthisfile}

\include{layoutforthisfile}

\title{\vspace{-2cm}Introduction to the Basics of Heegaard Floer Homology} 

\author{Bijan Sahamie\vspace{0.3cm}\\
Mathematisches Institut der LMU M\"unchen\\
Theresienstrasse 39\\
80333 M\"unchen} 

\theoremstyle{plain} 
\newtheorem{theorem}{Theorem}[section]   
\newtheorem{lem}[theorem]{Lemma}         
\newtheorem{prop}[theorem]{Proposition}
\newtheorem{cor}[theorem]{Corollary}

\theoremstyle{definition}
\newtheorem{definition}[theorem]{Definition}   
\newtheorem*{rem}{Remark}

\newtheorem{example}{Example}[section]            
\numberwithin{equation}{section}

\newarrow{Mto}{|}{-}{-}{-}{>}

\date{sahamie@math.lmu.de}

\begin{document}
\include{abbreviations}
\include{pictures}
\maketitle
\begin{abstract}    
\noindent This paper provides an introduction to the basics of Heegaard Floer homology
with some emphasis on the hat theory and to the contact geometric invariants in the theory. The 
exposition is designed to be comprehensible to people without any prior knowledge of 
the subject.
\end{abstract}
\tableofcontents
\section{Introduction}
Heegaard Floer homology was introduced by Peter Ozsv\'{a}th and Zoltan Szab\'{o} at the
beginning of the new millennium. Since then it developed very rapidly due to its
various contributions to low-dimensional topology, particularly knot theory
and contact geometry. The present paper is designed to give an introduction to
{\bf the basics} of Heegaard Floer theory with some emphasis on the hat theory. We try
to provide all details necessary to communicate a complete and comprehensible
picture. We would like to remark that there already are introductory 
articles to this subject (see \cite{OsZa08}, \cite{OsZa09} and \cite{OsZa10}).
The difference between the existing articles and the present article is threefold: 
First of all we present a lot more details. We hope that these details will provide
a complete picture of the basics of the theory. Our goal is to focus on those only which 
are relevant for the understanding of Heegaard Floer homology. Secondly, our exposition 
is not designed to present any applications and, in fact, we do not present any. Explaining
applications to the reader would lead us too far away from the basics and would force us to 
make some compromise to the exposition. We felt that going into advanced elements
would be disturbing to the goal of this paper. And  thirdly, we have a slight contact geometric 
focus. 

We think that the reader will profit the most from this paper when reading it 
completely rather than selecting a few elements: We start with a low-paced exposition 
and gain velocity as we move on. In this way we circumvent the creation of too many redundancies 
and it enables us to focus on the important facts at each stage of the paper. We expect the 
reader to have some knowledge about algebraic topology and surgery theory. As standard references we 
suggest \cite{Bredon} and \cite{GoSt}.\vspace{0.3cm}\\
In \S\ref{roadto} and \S\ref{genhomology} we start with Heegaard diagrams and introduce everything
necessary to construct the homology theory. We included a complete discussion of the invariance
of Heegaard Floer theory (cf.~\S\ref{topoinvariance}) for two reasons: Firstly, the isomorphisms defined
for showing invariance appear very frequently in the research literature. Secondly, the proof is 
based on constructions which can be called {\it the standard
constructions of the theory}. Those who are impatient may just read \S\ref{popprod} and
skip the rest of \S\ref{topoinvariance}. However, the remainder of
the article refers to details of \S\ref{topoinvariance} several times. The following paragraph, i.e.~\S\ref{knotfloerhomology}, is devoted to the knot theoretic variant of 
Heegaard Floer theory, called knot Floer homology. In \S\ref{parcobmaps} and \S\ref{parsurextri} we
outline how to assign to a $4$-dimensional cobordism a map between the Floer homologies of
the boundary components and derive the surgery exact triangle. This triangle is one of the
most important tools, particularly for the contact geometric applications. Finally, the article 
focuses on the definition of the contact geometric invariants.\vspace{0.3cm}\\
We are aware of the fact that there is a lot of material missing in this article. However, the 
presented theory provides a solid groundwork for understanding of what we omitted. We would
like to outline at least some of the missing material: First of all the homology groups as well 
as the cobordism maps refine with respect to $\spinc$-structures. We indicate this fact in \S\ref{roadto} but do not outline any details. The standard reference is the article \cite{OsZa01} of Ozsv\'{a}th and
Szab\'{o}. However, we suggest the reader first to familiarize with $\spinc$-structures, especially with 
their interpretation as homology classes of vector fields (cf.~\cite{Tur}). Furthermore, there is an absolute $\Q$-grading on these homology groups (see~\cite{OsZa03})
and in case of knot Floer homologies for homologically trivial knots an additional $\Z$-grading (see~\cite{OsZa04}). Both gradings carry topological information and may appear as a help in explicit calculations, especially in combination with the surgery exact triangles. The knot Floer homologies admit additional exact sequences besides the surgery exact sequence. An example is the skein exact sequence (see~\cite{OsZa04} and \cite{OsZa11}). For 
contact geometric applications the adjunction inequalities play a central 
role as they give a criterion for the vanishing of cobordism maps (see~\cite{OsZa02} or cf.~\cite{OzbSti}). 
Going a bit further, there are other flavors of Heegaard Floer homology: Andr\'{a}s Juhasz defined the so-called Sutured Floer homology of sutured manifolds (see~\cite{AJu}) and Ozsv\'{a}th, Lipshitz and Thurston defined a variant of Heegaard Floer homology for manifolds with parameterized boundary (see~\cite{LiOzTh}).

\section{Introduction to $\hfhat$ as a Model for Heegaard Floer Theory}\label{roadto}

\subsection{Heegaard Diagrams}\label{heegdiag}
One of the major 
results of Morse theory is the development of surgery and handle
decompositions. Morse theory captures the manifold's topology
in terms of a decomposition of it into topologically 
easy-to-understand pieces called {\bf handles} (cf.~\cite{GoSt}).
In the case of closed $3$-manifolds the handle decomposition can
be assumed to be very symmetric. This symmetry allows us to
describe the manifold's diffeomorphism type by a small amount
of data. Heegaard diagrams are omnipresent in low-dimensional topology. 
Unfortunately there is no convention what precisely to
call a Heegaard diagram; the definition of this notion 
underlies slight variations in different sources. Since Heegaard Floer
Homology intentionally uses a non-efficient version of Heegaard
diagrams, i.e.~we fix more information than needed to describe the
manifold's type, we shortly discuss, what is to be understood as
Heegaard diagram throughout this article.\vspace{0.3cm}\\
A short summary of what we will discuss would be that we fix the
data describing a handle decomposition relative to a splitting 
surface. Let $Y$ be a closed oriented $3$-manifold and 
$\Sigma\subset Y$ a {\bf splitting surface}, i.e.~a surface of
genus $g$ such that $Y\backslash\Sigma$ decomposes into two
handlebodies $H_0$ and $H_1$. We fix a handle decomposition 
of $\overline{Y\backslash H_1}$ relative to this splitting 
surface $\Sigma$, i.e.~there are $2$-handles $h^2_{1,i}$, $i=1,\dots,g$,
and a $3$-handle $h^3_1$ such that (cf.~\cite{GoSt})
\begin{equation}
Y\backslash H_1
\cong
(\Sigma\times[0,1])
\cupb
(h^2_{1,1}\cupb \dots\cupb h^2_{1,g}\cupb h^3_1).
\end{equation}
We can rebuild $Y$ from this by gluing 
in $2$-handles $h^2_{0,i}$, $i=1,\dots,g$, and a $3$-handle $h^3_0$.
Hence, $Y$ can be written as
\begin{equation}
  Y\cong 
  (h^3_0\cupb h^2_{0,1}\cupb\dots\cupb h^2_{0,g})
  \cupb 
  (\Sigma\times[0,1])
  \cupb
  (h^2_{1,1}\cupb \dots\cupb h^2_{1,g}\cupb h^3_1).
  \label{handledecomp}
\end{equation}
Collecting the data from this decomposition we obtain a triple
$(\Sigma,\alpha,\beta)$ where $\Sigma$ is the splitting surface
of genus $g$, $\alpha=\{\alpha_1,\dots,\alpha_g\}$ are the
images of the attaching circles of the $h^2_{0,i}$ interpreted as
sitting in $\Sigma$ and $\beta=\{\beta_1,\dots,\beta_g\}$ the
images of the attaching circles of the $2$-handles $h^2_{1,i}$
interpreted as sitting in $\Sigma$. This will be called
a {\bf Heegaard diagram of $Y$}. Observe that these data determine
a Heegaard decomposition {\it in the classical sense} by
dualizing the $h^2_{0,i}$. Dualizing a $k$-handle $D^k\times D^{3-k}$
means to reinterpret this object as $D^{3-k}\times D^k$. Both objects are
diffeomorphic but observe that the former is a $k$-handle and the latter
a $(3-k)$-handle. Observe that the $\alpha$-curves are
the co-cores of the $1$-handles in the dualized picture, and that
sliding $h^1_{0,i}$ over $h^1_{0,j}$ means, in the dual picture,
that $h^2_{0,j}$ is slid over $h^2_{0,i}$.

\subsection{Introduction 
to $\hfhat$ --- Topology and Analysis}\label{symproduct}
Given a closed, oriented $3$-manifold $Y$, we fix a Heegaard diagram 
$(\Sigma,\alpha,\beta)$ of $Y$ as defined in \S\ref{heegdiag}.
We can associate to it the triple $(\symg,\talpha,\tbeta)$ which we 
will explain now:\vspace{0.3cm}\\
By $\symg$ we denote the {\bf $g$-fold symmetric product} of $\Sigma$,
defined by taking the quotient under the canonical action of $S_g$ on 
$\Sigma^{\times g}$, i.e.
\[
  \symg=\Sigma^{\times g}/S_g.
\]
Although the action of $S_g$ has fixed points, the symmetric product
is a manifold. The local model is given by $\symc$ which itself 
can be identified with the set of normalized polynomials of degree
$g$. An isomorphism is given by sending a point
$[(p_1,\ldots,p_g)]$ to the normalized polynomial uniquely determined
by the zero set $\{p_1,\ldots,p_g\}$. Denote by 
\[
  \pi\co\Sigma^{\times g}\lra\symg
\]
the projection map.

The attaching circles $\alpha$ and $\beta$ define submanifolds 
\[
  \talpha
  =
  \alpha_1\times\ldots\times\alpha_g
  \;\;\mbox{\rm and }\;\;
  \tbeta=\beta_1\times\ldots\times\beta_g
\]
in $\Sigma^{\times g}$. Obviously, the projection $\pi$ embeds
these into the symmetric product. In the following we will
denote by $\talpha$ and $\tbeta$ the manifolds embedded into
the symmetric product.
\subsubsection{The chain complex}
Define $\cfhat(\Sigma,\alpha,\beta)$ as the free 
$\Z$-module (or $\ztwo$-module) generated by the intersection 
points $\talpha\cap\tbeta$ inside $\symg$. 
\begin{definition} A map $\phi$ of the $2$-disc $\D^2$ (regarded 
as the unit $2$-disc in $\C$) into the symmetric product $\symg$ is 
said to \textbf{connect} two points $x,y\in\talpha\cap\tbeta$ if 
\begin{eqnarray*}
\phi(i)&=&x,\\
\phi(-i)&=&y,\\  
\phi(\partial\D\cap\{z\in\C\,|\,Re(z)<0\})&\subset&\talpha,\\ 
\phi(\partial\D\cap\{z\in\C\,|\,Re(z)>0\})&\subset&\tbeta.
\end{eqnarray*} 
Continuous mappings of the $2$-disc into the 
symmetric product $\symg$ that connect two intersection points
 $x,y\in\talpha\cap\tbeta$ are called 
\textbf{Whitney discs}. The set of homotopy classes of Whitney 
discs connecting $x$ and $y$ is denoted by $\pitwo(x,y)$ in case $g>2$.
\end{definition}
In case $g\leq2$ we have to define the object $\pitwo(x,y)$ slightly different.
However, we can always assume, without loss of generality, that $g>2$ 
and, thus, we will omit discussing this case at all. We point the interested 
reader to \cite{OsZa01}.\vspace{0.3cm}\\
Fixing a point $z\in\Sigma\backslash(\alpha\cup\beta)$, we can construct a differential
\[
  \parhat_z\co\cfhat(\Sigma,\alpha,\beta)\lra\cfhat(\Sigma,\alpha,\beta)
\] 
by defining it on the generators 
of $\cfhat(\Sigma,\alpha,\beta)$. Given a point $x\in\talpha\cap\tbeta$,
we define $\parhat_z x$ to be a linear combination 
\[
  \parhat_z x
  =
  \sum_{y\in\talpha\cap\tbeta} 
  \left.\parhat_z x\right|_y\cdot y
\]
of all intersection points $y\in\talpha\cap\tbeta$. The definition of
the coefficients will occupy the remainder of this paragraph. The
idea resembles other Floer homology theories. The goal is to
define $\left.\parhat_z x\right|_y$ as a signed count of holomorphic 
Whitney discs connecting $x$ and $y$ which are rigid 
up to reparametrization. First we have to introduce almost complex
structures into this picture. A more detailed discussion of these
will be given in \S\ref{structmoduli}. For the moment it will be
sufficient to say that we choose a generic path 
$(\com_s)_{s\in[0,1]}$ of almost complex structures on the symmetric 
product. Identifying the unit disc, after taking out the points $\pm i$, 
in $\C$ with $[0,1]\times\R$ we define $\phi$ to be {\bf holomorphic} if 
it satisfies for all $(s,t)\in[0,1]\times\R$ the equation
\begin{equation}
  \frac{\partial\phi}{\partial s}(s,t)
  +\com_s\bigl(\frac{\partial\phi}{\partial t}(s,t)\bigr)=0.
  \label{holomeq}
\end{equation}
Looking into $(\ref{holomeq})$ it is easy to see that a
holomorphic Whitney disc $\phi$ can be reparametrized by
a constant shift in $\R$-direction without violating 
$(\ref{holomeq})$.
\begin{definition} Given two points $x,y\in\talpha\cap\tbeta$, we
denote by $\moduli$ the set of holomorphic Whitney discs connecting
$x$ and $y$. We call this set {\bf moduli space of holomorphic Whitney
discs} connecting $x$ and $y$. Given a homotopy class 
$[\phi]\in\pitwo(x,y)$, denote by
$\mathcal{M}_{\mathcal{J}_s,[\phi]}$ the space of holomorphic 
representatives in the
homotopy class of $\phi$.
\end{definition}
In the following the generic path of almost complex structures will
not be important and thus we will suppress it from the notation. 
Since the path is chosen generically (cf.~\S\ref{structmoduli} or
see \cite{OsZa01}) the moduli spaces are manifolds. The constant 
shift in $\R$-direction induces a free $\R$-action on the moduli 
spaces. Thus, if $\mathcal{M}_{[\phi]}$ is non-empty its dimension
is greater than zero. We take the quotient of $\mathcal{M}_{[\phi]}$ 
under the $\R$-action and denote the resulting spaces by
\[
  \modhatphib=\mathcal{M}_{[\phi]}/\R
  \;\;\mbox{\rm and}\;\;
  \modhatxy=\mathcal{M}(x,y)/\R.
\]
The so-called {\bf signed count} of $0$-dimensional components of
$\modhatxy$ means in case of $\ztwo$-coefficients simply to 
count mod $2$. In case of $\Z$-coefficients we have to introduce 
{\bf coherent orientations} on the moduli spaces. We will roughly sketch
this process in the following.\vspace{0.3cm}\\
Obviously, in case of $\Z$-coefficients we cannot simply count
the $0$-dimensional components of $\modhatxy$. The defined
morphism would not be a differential. To circumvent this problem we have
to introduce signs appropriately attached to each component. 
The $0$-dimensional components of $\modhatxy$
correspond to the $1$-dimensional components of $\mathcal{M}(x,y)$.
Each of these components carries a canonical orientation induced
by the free $\R$-action given by constant shifts. We introduce 
orientations
on these components. Comparing the artificial orientations with the 
canonical shifting orientation we can associate to each component, i.e.~each 
element in $\modhatxy$, a sign. The signed count will respect the
signs attached. There is a technical condition 
called {\bf coherence} (see \cite{OsZa01} 
or cf.~\S\ref{structmoduli}) one has to impose on the orientations. 
This technical condition ensures that the morphism $\parhat_z$ is a 
differential.\vspace{0.3cm}\\
The chosen point $z\in\Sigma\backslash(\alpha\cup\beta)$ will be part
of the definition. The path $(\com_s)_{s\in[0,1]}$ is chosen in such a
way that
\[
  V_z=\{z\}\times\symgmo\hookrightarrow\symg
\]
is a complex submanifold. For a Whitney disc (or its homotopy class)
$\phi$ define $n_z(\phi)$ as the intersection number of $\phi$ with
the submanifold $V_z$. We define
\[
  \left.\parhat_z x\right|_y=\#\modhatxy^0_{n_z=0},
\]
i.e.~the signed count of the $0$-dimensional components of
the unparametrized moduli spaces of holomorphic Whitney discs
connecting $x$ and $y$ with the property that their intersection
number $n_z$ is trivial. 
\begin{theorem}[see \cite{OsZa01}]\label{wdefined} The 
assignment $\parhat_z$ is well-defined.
\end{theorem}
%
%
%
%
\begin{theorem}[see \cite{OsZa01}]\label{differential} 
The morphism $\parhat_z$ is a differential.
\end{theorem}
We will give sketches of the proofs of the last two theorems later
in \S\ref{structmoduli}. At the moment we do not know enough about 
Whitney discs and the symmetric product to prove it.
\begin{definition} We denote by $\cfhat(\Sigma,\alpha,\beta,z)$
the chain complex given by the data 
$(\cfhat(\Sigma,\alpha,\beta),\partial_z)$. Denote by 
$\hfhat(Y)$ the induced homology theory $H_*(\cfhat(\Sigma,\alpha,\beta),\partial_z)$.
\end{definition}
The notation should indicate that the homology theory does not depend
on the data chosen. It is a topological invariant of the manifold $Y$,
although this is not the whole story. The theory depends on the choice
of coherent system of orientations. For a manifold $Y$ there are
$2^{b_1(Y)}$ numbers of non-equivalent systems of coherent orientations.
The resulting homologies can differ (see Example \ref{example01}). Nevertheless 
the orientations are not written down. We guess there are two reasons: The first 
would be that most of the time it is not really important which system
is chosen. All reasonable constructions will work for every coherent
orientation system, and in case there is a specific choice needed this
will be explicitly stated. The second reason would be that it is possible to 
give a convention for the choice of coherent orientation systems. 
Since we have not developed the mathematics to state the convention 
precisely we point the reader to Theorem \ref{cohblah}.
\subsubsection{On Holomorphic Discs in the Symmetric Product}
In order to be able to discuss a first example we briefly introduce some
properties of the symmetric product.
\begin{definition}\label{maslovindex} For a Whitney disc $\phi$ we denote by $\mu(\phi)$
the {\bf formal dimension} of $\mathcal{M}_\phi$. We also call $\mu(\phi)$ the {\bf Maslov index}
of $\phi$.
\end{definition}
For the readers that have not heard anything about Floer homology at
all, just think of $\mu(\phi)$ as the dimension of the space 
$\mathcal{M}_\phi$, although even in case $\mathcal{M}_\phi$ is 
not a manifold the number $\mu(\phi)$ is 
defined (cf.~\S\ref{structmoduli}). Just to give some intuition, note that the
moduli spaces are the zero-set of a section, $S$ say, in a Banach bundle 
one associates to the given setup. The linearization of this section
at the zero set is a Fredholm operator. Those operators carry a property
called Fredholm index. The number $\mu$ is the Fredholm index of 
that operator. Even if the moduli spaces are no manifolds this 
number is defined. It is called formal dimension or 
{\bf expected dimension} since in case the section $S$ intersects the
zero-section of the Banach-bundle transversely (and hence the moduli spaces
are manifolds) the Fredholm index $\mu$ equals the dimension of the
moduli spaces. So, negative indices are possible and make sense in some
situations. One can think of negative indices as the number
of missing degrees of freedom to give a manifold.
\begin{lem}\label{howsphere} In case $g(\Sigma)>2$ the 2nd homotopy group
$\pitwo(\symg)$ is isomorphic to $\Z$. It is generated 
by an element $S$ with $\mu(S)=2$ and $n_z(S)=1$, where $n_z$ is defined
the same way as it was defined for Whitney discs.
\end{lem}
Let $\eta\co\Sigma\lra\Sigma$ be an involution such that $\Sigma/\eta$
is a sphere. The map
\[
  \stwo\lra\symg,\,\;y\lmt\{(y,\eta(y),y,\dots,y)\}
\]
is a representative of $S$. Using this representative it is easy to
see that $n_z(S)=1$. It is a property of $\mu$ as an index that it 
behaves additive under concatenation. Indeed the intersection 
number $n_z$ behaves additive, too. 
To develop some intuition for the holomorphic spheres in the 
symmetric product we state the following result from \cite{OsZa01}.
\begin{lem}[see \cite{OsZa01}]\label{holspheres} There is 
an exact sequence 
\[
  0\lra
  \pitwo(\symg)
  \lra\pitwo(x,x)
  \lra
  \mbox{\rm ker}(n_z)
  \lra 0.
\]
The map $n_z$ provides a splitting for the sequence.
\end{lem}
Observe that we can interpret a Withney disc in $\pitwo(x,x)$ as 
a family of paths in $\symg$ based at the constant path $x$. We can also interpret an 
element in $\pitwo(\symg)$ as a family of paths in $\symg$ based at the
constand path $x$. Interpreted in this way there is a natural map 
from $\pitwo(\symg)$ into $\pitwo(x,x)$. The map $n_z$ provides a splitting for the 
sequence as it may be used to define the map 
\[
  \pitwo(x,x)\lra\pitwo(\symg)
\]
sending a Whitney disc $\phi$ to $n_z(\phi)\cdot S$. This obviously defines a splitting 
for the sequence.
\begin{lem}\label{kernz} The Kernel of $n_z$ interpreted as 
a map on $\pitwo(x,x)$ is isomorphic to $H^1(Y;\Z)$.
\end{lem}
With the help of concatenation we are able to define an
action 
\[
  *\co\pitwo(x,x)\times\pitwo(x,y)\lra\pitwo(x,y),
\]
which is obviously free and transitive. Thus, we have an identification
\begin{equation}
\begin{diagram}[size=1.5em,labelstyle=\scriptstyle]
\pitwo(x,y) &&
\rTo^\cong  &&
\pitwo(x,x) &
\cong\Z\oplus H^1(Y;\Z)\\
 &
\rdTo& &
\ldTo& & \\
&&
\{*\}&& &
\end{diagram}\label{pbident}
\end{equation}
as principal bundles over a one-point space, which is another way of
saying that the concatenation action endows $\pitwo(x,y)$ with a group
structure after fixing a unit element in $\pitwo(x,y)$. To address the 
well-definedness of $\parhat_z$ we have to show that the sum 
in the definition of $\parhat_z$ is finite. For the moment 
let us assume that for a generic choice of path 
$(\com_s)_{s\in[0,1]}$ the moduli spaces $\modhatphi$ with
$\mu(\phi)=1$ are compact manifolds (cf.~Theorem \ref{modmani}), 
hence their signed count is finite. Assuming this property we 
are able to show well-definedness of $\parhat_z$ in case $Y$ 
is a homology sphere.
\begin{proof}[Proof of Theorem \ref{wdefined} for $b_1(Y)=0$]
Observe that 
\begin{equation}
  \modhatxy^0_{n_z=0}=\bigsqcup_{\phi\in H(x,y,1)}\modhatphi,
  \label{modsplit}
\end{equation}
where $H(x,y,1)\subset\pitwo(x,y)$ is the subset of homotopy classes
admitting holomorphic representatives with $\mu(\phi)=1$ and $n_z=0$. 
We have to show that $H(x,y,1)$ is a finite set. Since $b_1(Y)=0$
the cohomology $H^1(Y;\Z)$ vanishes. By our preliminary discussion, given 
a reference disc $\phi_0\in\pitwo(x,y)$, any $\phi_{xy}\in\pitwo(x,y)$ can
be written as a concatenation
$\phi_{xy}=\phi*\phi_0$, where $\phi$ is an element in $\pitwo(x,x)$. 
Since we are looking for
discs with index one we have to find all $\phi\in\pitwo(x,x)$ satisfying
the property $\mu(\phi)=1-\mu(\phi_0)$.
Recall that $Y$ is a homology sphere and thus 
$\pitwo(x,x)\cong\Z\otimes\{S\}$. Hence, the
disc $\phi$ is described by an integer $k\in\Z$, i.e.~$\phi=k\cdot S$. The
property $\mu(S)=2$ tells us that
\[
  1-\mu(\phi_0)=\mu(\phi)=\mu(k\cdot S)=k\cdot\mu(S)=2k.
\]
There is at most one $k\in\Z$ satisfying this equation, so there is at most
one homotopy class of Whitney discs satisfying the property $\mu=1$ and
$n_z=0$.
\end{proof}
In case $Y$ has non-trivial first cohomology we need an additional
 condition to make the proof work. The given argument obviously 
breaks down in this case. To fix this we impose a topological/algebraic
condition on the Heegaard diagram. Before we can define these
{\it admissibility} properties we have to go into the theory a bit
more.\vspace{0.3cm}\\
There is an obstruction to finding Whitney discs connecting two
given intersection points $x,y$. The two points $x$ and $y$ can
certainly be connected via paths inside $\talpha$ and $\tbeta$.
Fix two paths $a\co I\lra\talpha$ and $b\co I\lra\tbeta$
such that $-\partial b=\partial a=y-x$. This is the same as saying
we fix a closed curve $\gamma_{xy}$ based at $x$, going to $y$ along
$\talpha$, and moving back to $x$ along $\tbeta$. Obviously
$\gamma_{xy}=b+a$. Is it possible to extend the curve $\gamma_{xy}$, after
possibly homotoping it a bit, to a disc? If so, this would be
a Whitney disc. Thus, finding an obstruction can be reformulated
as: Is $[\gamma_{xy}]=0\in\pi_1(\symg)$?
\begin{lem}[see \cite{OsZa01}]\label{fundone} The 
group $\pi_1(\symg)$ is abelian.
\end{lem}
Given a closed curve $\gamma\subset\symg$ in general
position (i.e.~not meeting the diagonal of $\symg$), we can 
lift this curve to 
\[
  (\gamma_1,\dots,\gamma_g)\co \sone\lra\Sigma^{\times g}.
\]
Projection onto each factor $\Sigma$ defines a $1$-cycle. We define
\[
  \Phi(\gamma)=\gamma_1+\dots+\gamma_g.
\]
\begin{lem}[see \cite{OsZa01}]\label{fundtwo} 
The map $\Phi$ induces an isomorphism 
\[\Phi_*\co H_1(\symg)\lra H_1(\Sigma;\Z).\]
\end{lem}
By surgery theory (see \cite{GoSt}, p.\! 111) we know that
\begin{equation}
\frac{H_1(\Sigma;\Z)}
{[\alpha_1],\ldots,[\alpha_g],
[\beta_1],\ldots,[\beta_g]}
\cong H_1(Y;\Z)\label{hom01}
\end{equation}
The curve $\gamma_{xy}$
is homotopically trivial in the symmetric product if and only
if $\Phi_*([\gamma_{xy}])$ is trivial. If we pick different curves 
$a$ and $b$ to define another curve $\eta_{xy}$, the difference
\[
  \Phi(\gamma_{xy})-\Phi(\eta_{xy})
\]
is a sum of $\alpha$-and $\beta$-curves. Thus, interpreted as a
cycle in $H_1(Y;\Z)$, the class
\[
  [\Phi(\gamma_{xy})]\in H_1(Y;\Z)
\]
does not depend on the choices made in its definition. 
We get a map
\begin{equation*}
\begin{array}{rccl}
  \epsilon\co&(\talpha\cap\tbeta)^{\times 2}&\lra& H_1(Y;\Z)\\
             &(x,y)&\lmt & [\Phi(\gamma_{xy})]_{H_1(Y;\Z)}
\end{array}
\end{equation*}
with the following property.
\begin{lem} If $\epsilon(x,y)$ is non-zero the set $\pitwo(x,y)$ is
empty.
\end{lem}
\begin{proof} Suppose there is a connecting disc $\phi$ then
with $\gamma_{xy}=\partial(\phi(\D^2))$ we have
\[
  \epsilon(x,y)=
  [\Phi(\gamma_{xy})]_{H_1(Y;\Z)}
  =
  \frac{
  \Phi_*([\gamma_{xy}]_{H_1({\tiny \symg})})
  }
  {[\alpha_1],\ldots,[\alpha_g],
[\beta_1],\ldots,\beta_g]}
  =0
\]
since $[\gamma_{xy}]_{\pi_1({\tiny \symg})}=0$.
\end{proof}
As a consequence we can split up the chain complex 
$\cfhat(\Sigma,\alpha,\beta,z)$ into subcomplexes. It is
important to notice that there is a map
\begin{equation}
  s_z\co
  \talpha\cap\tbeta
  \lra
  \mbox{\rm Spin}^c_3(Y)
  \cong H^2(Y;\Z),\label{szmap}
\end{equation}
such that $\mbox{\rm PD}(\epsilon(x,y))=s_z(x)-s_z(y)$. We point the reader 
interested in the definition of $s_z$ to \cite{OsZa01}. Thus, fixing 
a $\spinc$-structure $s$, the $\Z$-module (or $\ztwo$-module) 
$\cfhat(\Sigma,\alpha,\beta,z;s)$ generated by $(s_z)^{-1}(s)$
defines a subcomplex of $\cfhat(\Sigma,\alpha,\beta,z)$. The
associated homology is denoted by $\hfhat(Y,s)$, and it is a submodule
of $\hfhat(Y)$. Especially note that
\[
  \hfhat(Y)=\bigoplus_{s\in{\tiny\spiny}(Y)}\hfhat(Y,s).
\]
Since $\talpha\cap\tbeta$ consists of finitely many
points, there are just finitely many groups in this splitting
which are non-zero. In general this splitting will depend on
the choice of base-point. If $z$ is chosen in a different
component of $\Sigma\backslash\{\alpha\cup\beta\}$ there will
be a difference between the $\spinc$-structure associated to
an intersection point. For details we point 
to \cite{OsZa01}.
\begin{example} The Heegaard diagram given by the 
data $(T^2,\{\mu\},\{\lambda\})$ (cf.~\S\ref{heegdiag}) is 
the $3$-sphere. 
To make use of Lemma \ref{howsphere} we add two stabilizations
to get a Heegaard surface of genus $3$, i.e.
\[
  D=(T^2\#T^2\#T^2,\{\mu_1,\mu_2,\mu_3\},\{\lambda_1,\lambda_2,\lambda_3\}),
\]
where $\mu_i$ are meridians of the tori, and $\lambda_i$ are longitudes. 
The complement of the attaching curves is connected. Thus, we can arbitrarily
choose the base point $z$. The chain complex
$\cfhat(D,z)$ equals one copy of $\Z$ since
it is generated by one single intersection point which we denote by
$x$. We claim that $\parhat_z x=0$. Denote by $[\phi]$ a
homotopy class of Whitney discs connecting $x$ with itself. This 
is a holomorphic sphere which can be seen with Lemma \ref{holspheres},
Lemma \ref{kernz} and the fact that $H^1(\sthree)=0$. By Lemma \ref{howsphere} the set $\pitwo(\symg)$ 
is generated by $S$ with the property $n_z(S)=1$. The additivity
of $n_z$ under concatenation shows that $[\phi]$ is a trivial holomorphic
sphere and $\mu([\phi])=0$. Thus, the space 
$\mathcal{M}(x,x)^1_{n_z=0}$, i.e.~the space of holomorphic Whitney discs connecting $x$ with itself, with
$\mu=1$ and $n_z=0$, is empty.
Hence
\[
  \hfhat(\sthree)\cong\Z.
\]
\end{example}
\subsubsection{A Low-Dimensional Model for Whitney Discs}
The exact sequence in Lemma \ref{holspheres} combined with
Lemma \ref{kernz} and $(\ref{pbident})$ gives an interpretation 
of Whitney discs as homology classes. Given a disc $\phi$, we 
define its associated homology class by $\homology(\phi)$, i.e.
\begin{equation}
  0\lra
  \pitwo(\symg)
  \lra\pitwo(x,x)
  \overset{\mathcal{H}}{\lra}
  H_2(Y;\Z)
  \lra 0.\label{nice}
\end{equation}
In the following we intend to give a description of the map
$\mathcal{H}$. Given a Whitney disc $\phi$, we can lift this
disc to a map $\phitilde$ by pulling back the branched covering
$\pi$ (cf.~diagram (\ref{wdisc})).
\begin{equation}
\begin{diagram}[size=1.5em,labelstyle=\scriptstyle]
F/S_{g-1}=& \Dhat& & \rTo^{\phibar} & &\Sigma\times\symgmo &\rTo &\Sigma\\
  & \uTo & & & &\uTo & &\\
\phi^*\Sigma^{\times g}=&F & & \rTo^\phitilde & & \Sigma^{\times g} & &\\
&\dTo  & &      & & \dTo^\pi & &\\
&\D^2  & & \rTo^\phi & & \symg & &
\end{diagram}\label{wdisc}
\end{equation}
Let $S_{g-1}\subset S_g$ be the subgroup of permutations fixing the
first component. Modding out $S_{g-1}$ we obtain the map $\phibar$
pictured in (\ref{wdisc}). Composing it with the projection onto
the surface $\Sigma$ we define a map
\[
  \phihat\co\Dhat\lra\Sigma.
\]
The image of this map $\phihat$ defines what is called a domain.
\begin{definition}
Denote by $\dom_1,\ldots, \dom_m$ the closures of 
the components of the complement of the attaching circles
$\Sigma\backslash\{\alpha\cup\beta\}$. 
Fix one point $z_i$ in each component. A \textbf{domain} 
is a linear combination
\[
  \mathcal{A}=\sum_{i=1}^m\lambda_i\cdot\dom_i
\] 
with $\lambda_1,\ldots,\lambda_m\in\Z$.
\end{definition}
For a Whitney disc $\phi$ we define its {\bf associated domain} by
\[
  \dom(\phi)=\sum_{i=1}^m n_{z_i}(\phi)\cdot\dom_i.
\]
The map $\phihat$ and $\dom(\phi)$ are related by the equation
\[
  \phihat(\Dhat)=\dom(\phi)
\]
as chains in $\Sigma$ relative to the set $\alpha\cup\beta$. We 
define $\homology(\phi)$ as the associated homology
class of $\phihat_*[\Dhat]$ in $H_2(Y;\Z)$. The correspondence is
given by closing up the boundary components by using the core discs of
the $2$-handles represented by the $\alpha$-curves and the
$\beta$-curves.
\begin{lem} Two Whitney discs $\phi_1,\phi_2\in\pitwo(x,x)$ are 
homotopic if and only if their domains are equal.
\end{lem}
\begin{proof} Given two discs $\phi_1$, $\phi_2$ whose domains are equal, by definition
$\homology(\phi_1)=\homology(\phi_2)$. By $(\ref{nice})$ they
can only differ by a holomorphic sphere, i.e.~$\phi_1=\phi_2+k\cdot S$.
The equality $\dom(\phi_1)=\dom(\phi_2)$ implies 
that $n_z(\phi_1)=n_z(\phi_2)$. The equation 
\[
  0=n_z(\phi_2)-n_z(\phi_1)=n_z(\phi_2)-n_z(\phi_2+k\cdot S)=2k
\]
forces $k$ to vanish.
\end{proof}
The interpretation of Whitney discs as domains is very useful in
computations, as it provides a low-dimensional model. The symmetric
product is $2g$-dimensional, thus an investigation of holomorphic
discs is very inconvenient. However, not all domains are carried
by holomorphic discs. Obviously, the equality 
$[\dom(\phi)	]=\phihat_*[\Dhat]$ connects the boundary conditions
imposed on Whitney discs to boundary conditions of the domains.
It is not hard to observe that the definition of $\phihat$ follows
the same lines as the construction of the isomorphism $\Phi_*$ of
homology groups discussed earlier (cf.~Lemma \ref{fundtwo}). Suppose
we have fixed two intersections $x=\{x_1,\dots,x_g\}$ and 
$y=\{y_1,\dots,y_g\}$ connected by a Whitney disc $\phi$. The boundary
$\partial(\phi(\D^2))$ defines a connecting curve $\gamma_{xy}$. It is easy
to see that
\[
  \mbox{\rm im}(\left.\phihat\right|_{\partial\Dhat})
  =
  \Phi(\gamma_{xy})=\gamma_1+\dots+\gamma_g.
\]
Restricting the $\gamma_i$ to the $\alpha$-curves we get a chain
connecting the set $x_1,\dots,x_g$ with $y_1,\dots,y_g$, and restricting
the $\gamma_i$ to the $\beta$-curves we get a chain connecting
the set $y_1,\dots,y_g$ with $x_1,\dots,x_g$. This means each boundary
component of $\Dhat$ consists of a set of arcs alternating
 through $\alpha$-curves and $\beta$-curves.
\begin{definition} A domain is called {\bf periodic} if its boundary
is a sum of $\alpha$-and $\beta$-curves and $n_z(\dom)=0$, i.e.~the
multiplicity of $\dom$ at the domain $\dom_z$ containing $z$ vanishes. 
\end{definition}
Of course a Whitney disc is called {\bf periodic} if its associated
domain is a periodic domain. The subgroup of periodic classes in
 $\pitwo(x,x)$ is denoted by $\Pi_x$.
\begin{theorem}[see \cite{OsZa01}]\label{calcdim} 
For a $\spinc$-structure $s$
and a periodic class $\phi\in\Pi_x$ we have the equality
\[
  \mu(\phi)=\left<c_1(s),\homology(\phi)\right>.
\]
\end{theorem}
This is a deep result connecting the expected dimension of a periodic disc
with a topological property. Note that, because of the additivity of the
expected dimension $\mu$, the homology groups $\hfhat(Y,s)$ can be endowed with
a relative grading defined by
\[
  \grading(x,y)=\mu(\phi)-2\cdot n_z(\phi),
\]
where $\phi$ is an arbitrary element of $\pitwo(x,y)$. In
the case of homology spheres this defines a relative 
$\Z$-grading because by Theorem \ref{calcdim} the expected dimension
vanishes for all periodic discs. In case of non-trivial homology 
they just vanish modulo $\delta(s)$, where 
\[
  \delta(s)
  =
  \underset{A\in H_2(Y;\Z)}{\mbox{\rm gcd}}
  \left<c_1(s),A\right>,
\]
i.e.~it defines a relative $\Z_{\delta(s)}$-grading. 
\begin{definition}\label{admissintro} A pointed Heegaard 
diagram $(\Sigma,\alpha,\beta,z)$ is called 
{\bf weakly admissible} for the $\spinc$-structure $s$ if 
for every non-trivial periodic domain $\dom$ such that
$\left<c_1(s),\homology(\dom)\right>=0$ the domain has positive 
and negative coefficients.
\end{definition}
With this technical condition imposed the $\parhat_z$ is a 
well-defined map on the subcomplex $\cfhat(\Sigma,\alpha,\beta,s)$. 
From admissibility it follows that for every 
$x,y\in (s_z)^{-1}(s)$ and $j,k\in\Z$ there exists just a 
finite number of $\phi\in\pitwo(x,y)$ with
$\mu(\phi)=j$, $n_z(\phi)=k$ and $\dom(\phi)\geq 0$. The last
condition means that all coefficients in the associated domain
are greater or equal to zero. 
\begin{proof}[Proof of Theorem \ref{wdefined} for $b_1(Y)\not=0$]
Recall that holomorphic discs are either contained in a complex 
submanifold $C$ or they intersect $C$ always transversely and 
always positive. The definition of the 
path $(\com_s)_{s\in[0,1]}$ (cf.~\S\ref{structmoduli}) includes that
all the $\{z_i\}\times\symgmo$ are complex submanifolds. Thus, holomorphic 
Whitney discs always satisfy $\dom(\phi)\geq 0$.
\end{proof}
We close this paragraph with a statement that appears to be useful
for developing intuition for Whitney discs. It helps imagining the
strong connection between the discs and their associated domains.
\begin{theorem}[see \cite{OsZa01}] Consider a domain 
$\dom$ whose coefficients are all greater than or equal to zero. There 
exists an oriented $2$-manifold $S$ with boundary and a map 
$\phi\co S\lra\Sigma$ with $\phi(S)=\dom$ with the property 
that $\phi$ is nowhere orientation-reversing and the restriction 
of $\phi$ to each boundary component of $S$ is a diffeomorphism 
onto its image.
\end{theorem}

\subsection{The Structure of the Moduli Spaces}\label{structmoduli}
The material in this paragraph is presented without any details.
The exposition pictures the bird's eye view of the material. Recall 
from the last paragraphs that we have to choose a path of 
almost complex structures appropriately to define Heegaard 
Floer theory. So, a discussion of these structures is inevitable.
However, a lot of improvements have been made the last years and
we intend to mention some of them.\vspace{0.3cm}\\
Let $(j,\eta)$ be a K\"{a}hler structure on the Heegaard surface 
$\Sigma$, i.e.~$\eta$ is a symplectic form and $j$ an 
almost-complex structure that tames $\eta$. Let $z_1,\dots,z_m$ be points, one in each component
of $\Sigma\backslash\{\alpha\cup\beta\}$. Denote by $V$ an open
neighborhood in $\symg$ of
\[
  D\cup\bigl(\bigcup_{i=1}^m\{z_i\}\times\symgmo\bigr),
\]
where $D$ is the diagonal in $\symg$.
\begin{definition}\label{defnearlysym} An almost complex structure 
$J$ on $\symg$ is called $(j,\eta,V)$-{\bf nearly symmetric} if $J$ 
agrees with $sym^g(j)$ over $V$ and if $J$ tames 
$\pi_*(\eta^{\times g})$ over $\overline{V}^c$. The set of
 $(j,\eta,V)$-nearly symmetric almost-complex structures will be 
denoted by $\mathcal{J}(j,\eta,V)$.
\end{definition}
The almost complex structure $sym^g(j)$ on $\symg$ is the natural 
almost complex structure induced by the structure $j$. Important for 
us is that the structure $J$ agrees with $sym^g(j)$ on $V$. This makes the
$\{z_i\}\times\symgmo$ complex submanifolds with respect 
to $J$. This is necessary to guarantee positive intersections 
with Whitney discs. Without this property the proof of 
Theorem \ref{wdefined} would break down in the case the manifold 
has non-trivial topology. \vspace{0.3cm}\\
We are interested in holomorphic Whitney discs, i.e.~discs in the 
symmetric product which are solutions of (\ref{holomeq}). Denote by
 the $\riemop$ the Cauchy-Riemann type operator  defined by 
 equation (\ref{holomeq}). Define
$\banbund(x,y)$ as the space of Whitney discs connecting $x$ and $y$
such that the discs converge to $x$ and $y$ exponentially with respect
to some Sobolev space norm in a neighborhood of $i$ and $-i$ 
(see \cite{OsZa01}). With these assumptions the solution 
$\riemop\phi$ lies in a space of $L^p$-sections 
\[
  L^p([0,1]\times\R,\phi^*(T\symg)).
\]
These fit together to form a bundle $\mathcal{L}$ over the base
$\banbund(x,y)$.
\begin{theorem} The bundle $\mathcal{L}\lra\banbund(x,y)$ is a
Banach bundle.
\end{theorem}
By construction the operator $\riemop$ is a section of that Banach
bundle. Let us define $\banbund_0\hookrightarrow\banbund(x,y)$ as the
zero section, then obviously
\[
  \moduli=(\riemop)^{-1}(\banbund_0).
\]
Recall from the Differential Topology of finite-dimensional manifolds
that if a smooth map intersects a submanifold transversely then its
preimage is a manifold. There is an analogous result in the 
infinite-dimensional theory. The generalization to infinite dimensions
requires an additional property to be imposed on the map. We will now define
this property.
\begin{definition} A map $f$ between Banach manifolds is called 
{\bf Fredholm} if for every point $p$ the differential $T_pf$ 
is a Fredholm operator, i.e.~has finite-dimensional kernel and 
cokernel. The difference $\dim\ker T_pf-\dim\mbox{\rm coker } T_pf$ is 
called the {\bf Fredholm index} of $f$ at $p$.
\end{definition}
Fortunately the operator $\riemop$ is an elliptic operator, and 
hence it is Fredholm for a generic choice of path 
$(\com_s)_{s\in[0,1]}$ of almost complex structures.
\begin{theorem}(see \cite{OsZa01})\label{modmani} For a dense set of 
paths $(\com_s)_{s\in[0,1]}$ of $(j,\eta,V)$-nearly symmetric  
almost complex structures the moduli spaces $\moduli$ are smooth 
manifolds for all $x,y\in\talpha\cap\tbeta$.
\end{theorem}
The idea is similar to the standard Floer homological
proof. One realizes these paths as regular values of 
the Fredholm projection 
\[
  \pi\co\mathcal{M}\lra\Omega(\mathcal{J}(j,\eta,V)),
\]
where $\Omega(\mathcal{J}(j,\eta,V))$ denotes the space of paths in
$\mathcal{J}(j,\eta,V)$ and $\mathcal{M}$ is the unparametrized moduli space consisting 
of pairs $(\com_s,\phi)$, where $\com_s$ is a path of 
$(j,\eta,V)$-nearly symmetric almost complex structures and 
$\phi$ a Whitney disc. By the Sard-Smale theorem the 
set of regular values is an open and dense set of 
$\mathcal{J}(j,\eta,V)$.\vspace{0.3cm}\\
Besides the smoothness of the moduli spaces we need the number
of one-dimensional components to be finite. This means we
require the spaces $\modhatxy^0_{n_z=0}$ to be compact. One 
ingredient of the compactness is the admissibility property 
introduced in Definition \ref{admissintro}. In (\ref{modsplit}) 
we observed that
\[
  \modhatxy^0_{n_z=0}=\bigsqcup_{\phi\in H(x,y,1)}\modhatphi,
\]
where $H(x,y,1)$ is the set of homotopy classes of Whitney discs
with $n_z=0$ and expected dimension $\mu=1$. Admissibility
guarantees that $H(x,y,1)$ is a finite set. Thus, compactness
follows from the compactness of the $\modhatphi$. The compactness
proof follows similar lines as the Floer homological approach.
It follows from the existence of an {\it energy bound} independent 
of the homotopy class
of Whitney discs. The existence of this energy bound shows that
the moduli spaces $\modhatxy$ admit a compactification by adding
solutions to the space in a controlled way.\vspace{0.3cm}\\
Without giving the precise definition we would like to give some
intuition of what happens at the boundaries. First of all there
is an operation called {\bf gluing} making it possible to
concatenate Whitney discs holomorphically. Given two Whitney
discs $\phi_1\in\pitwo(x,y)$ and $\phi_2\in\pitwo(y,w)$, gluing
describes an operation to generate a family of 
holomorphic solutions $\phi_2\#_t\phi_1$ in the homotopy 
class $\phi_2*\phi_1$. 
\begin{definition} We call the pair $(\phi_2,\phi_1)$ a
{\bf broken} holomorphic Whitney disc.\footnote{This might be a
sloppy and informal definition but appropriate for our intuitive
approach.}
\end{definition}
Moreover, one can think of this solution $\phi_2\#_t\phi_1$ 
as sitting in a small neighborhood of the boundary of the 
moduli space of the homotopy class $\phi_2*\phi_1$, i.e.~the 
family of holomorphic solutions as $t\to\infty$ converges 
to the broken disc $(\phi_2,\phi_1)$. There is a special notion
of convergence used here. The limiting objects can be described 
intuitively in the following way: Think of the disc, after removing 
the points $\pm i$, as a strip
 $\R\times[0,1]$. Choose a properly embedded arc or 
 an embedded $\sone$ in $\R\times[0,1]$. Collapse the curve or 
 the $\sone$ to a point. The resulting object is a potential limiting
 object. The objects at the limits of sequences can be derived by
 applying several knot shrinkings and arc shrinkings simultaneously 
 where we have to keep in mind that the arcs and knots have to be chosen 
 such that they do not intersect (for a detailed treatment 
 see \cite{DuffSal}).\vspace{0.3cm}\\
We see that every broken disc corresponds to a boundary component 
of the compactified moduli space, i.e.~there is an injection
\[
  \fglue\co\mathcal{M}_{\phi_2}\times\mathcal{M}_{\phi_1}
   \hookrightarrow\partial\mathcal{M}_{\phi_2*\phi_1}.
\]
But are these the only boundary components? If this is the case,
by adding broken discs to the space we would compactify it. This
would result in the finiteness of the $0$-dimensional spaces
$\modhatphi$. A compactification by adding broken flow lines 
means that the $0$-dimensional components are compact in the 
usual sense. A simple dimension count contradicts the existence 
of a family of discs in a $0$-dimensional moduli space
converging to a broken disc. But despite that there is a second
reason for us to wish broken flow lines to compactify the moduli
spaces. The map $\parhat_z$ should be a boundary operator. Calculating
$\parhat_z\circ\parhat_z$ we see that the coefficients in the
resulting equation equal the number of boundary components corresponding 
to broken discs 
at the ends of the $1$-dimensional moduli spaces. If the gluing map
is a bijection the broken ends generate all boundary components.
Hence, the coefficients vanish mod $2$.\vspace{0.3cm}\\
There are two further phenomena we have to notice. Besides breaking 
there might be {\bf spheres bubbling off}. This description can 
be taken literally to some point. Figure \ref{Fig:figThree}
\begin{figure}[ht!]
\centerline{\psfig{file=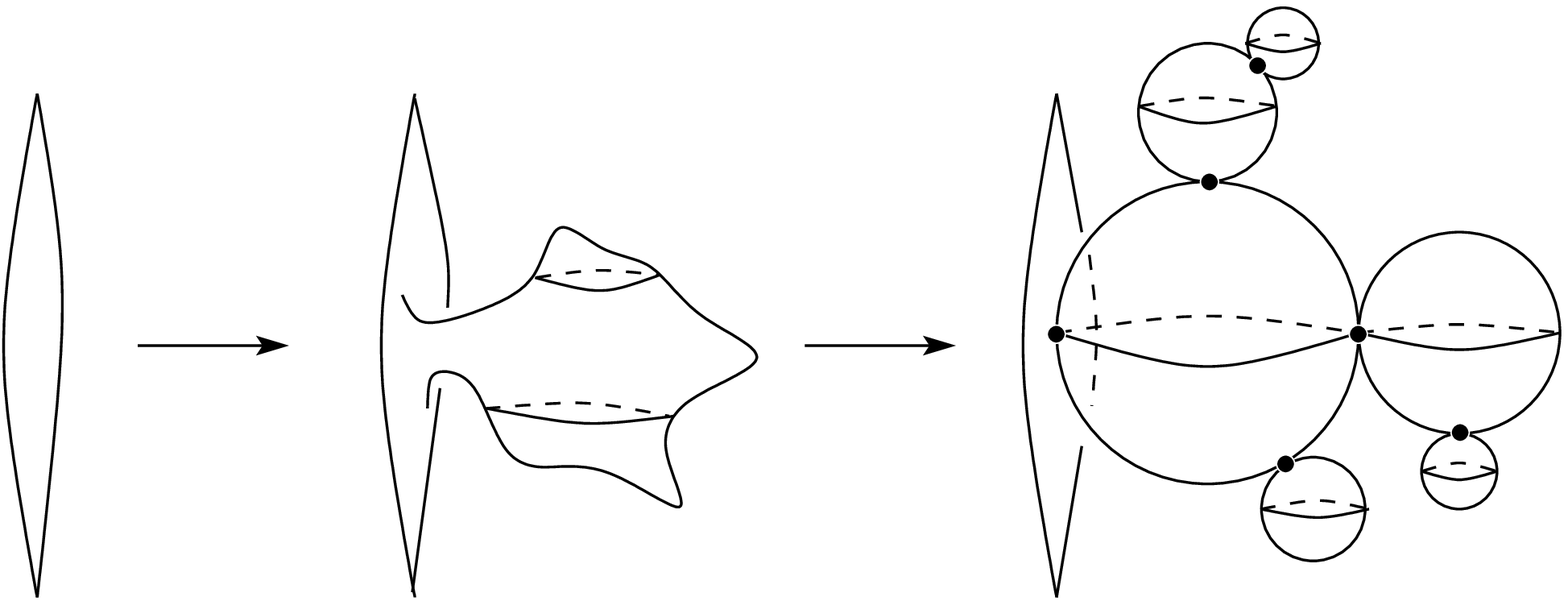,height=3cm}}
\caption{Bubbling of spheres.}
\label{Fig:figThree}
\end{figure}
 illustrates the geometric
picture behind that phenomenon. Bubbling is some kind of breaking
phenomenon but the components here are discs and spheres. We do not 
need to take care of spheres bubbling off at all. Suppose that the 
boundary of the moduli space associated to the homotopy class 
$\phi$ we have breaking into a disc $\phi_1$ and a sphere $S_1$, 
i.e.~$\phi=\phi_1*S_1$. Recall that the spheres in the symmetric 
product are generated by $S$, described in \S\ref{symproduct}. 
Thus, $\phi=\phi_1*k\cdot S$ where $n_z(S)=1$. In consequence 
$n_z(\phi)$ is non-zero, contradicting the assumptions.
\newpage
\begin{definition} For a point $x\in\talpha\cap\tbeta$ an
 {\bf $\alpha$-degenerate} disc is a holomorphic disc
$\phi\co[0,\infty)\times\R\lra\symg$ with the following boundary
conditions $\phi(\{0\}\times\R)\subset\talpha$ and $\phi(p)\to x$
as $x\to\infty$.
\end{definition}
Given a degenerate disc $\psi$, the associated domain $\dom(\psi)$
equals a sphere with holes, i.e.~$\dom(\psi)$ equals a surface
in $\Sigma$ with boundary the $\alpha$-curves. Since the 
$\alpha$-curves do not disconnect $\Sigma$, the domain covers 
the whole surface. Thus, $n_z(\psi)$ is non-zero, showing that 
degenerations are ruled out by assuming that $n_z=0$.

\begin{proof}[Proof of Theorem \ref{differential} 
with $\ztwo$-coefficients] 
Fix an intersection $x\in\talpha\cap\tbeta$. We compute
\begin{eqnarray*}
  \parhat_z x 
  &=&
  \parhat_z
  \bigl(
  \sum_{y\in\talpha\cap\tbeta}\#\modhatxy_{n_z=0}^0\cdot y
  \bigr)\\
  &=&\sum_{y,w\in\talpha\cap\tbeta}
  \#\modhatxy_{n_z=0}^0
  \#\modhatyw_{n_z=0}^0\cdot w.
\end{eqnarray*}
We have to show that the coefficient in front of $w$, denoted by 
$c(x,w)$ vanishes. Observe that the coefficient precisely 
equals the number of components (mod $2$) in
\[
  \modhatxy_{n_z=0}^0\times\modhatyw_{n_z=0}^0
\]
Gluing gives an injection
\[
  \modhatxy_{n_z=0}^0
  \times
  \widehat{\mathcal{M}}(y,w)_{n_z=0}^0
  \hookrightarrow
  \partial\widehat{\mathcal{M}}(x,w)_{n_z=0}^1.
\] 
By the compactification theorem the gluing map is a bijection,  
since bubbling and degenerations do not appear due to the 
condition $n_z=0$. Thus, (mod $2$) we have
\begin{eqnarray*}
  c(x,w)
  &=&\#(\modhatxy_{n_z=0}^0\times\modhatyw_{n_z=0}^0)\\
  &=&\partial\widehat{\mathcal{M}}(x,w)_{n_z=0}^1\\
  &=&0,
\end{eqnarray*}
which shows the theorem.
\end{proof}
Obviously, the proof breaks down in $\Z$-coefficients. We need the
mod $2$ count of ends. There is a way to fix the proof. The goal
is to make the map
\[
  \fglue\co\mathcal{M}_{\phi_2}\times\mathcal{M}_{\phi_1}
   \hookrightarrow\partial\mathcal{M}_{\phi_2*\phi_1}
\]
orientation preserving. For this to make sense we need the moduli 
spaces to be oriented. An orientation is given by choosing a
section of the {\bf determinant line bundle} over the moduli
spaces. The determinant line bundle is defined as the bundle
$\det([\phi])\lra\mathcal{M}_\phi$ given by putting together the spaces
\[
  \det(\psi)
  =
  \bigwedge\,\!\!\!^{\mbox{\rm \tiny max}}
  \ker(D_\psi\riemop)
  \otimes
  \bigwedge\,\!\!\!^{\mbox{\rm \tiny max}}
  \ker((D_\psi\riemop)^*),
\]
where $\psi$ is an element of $\mathcal{M}_{\phi}$. If we achieve
transversality for $\riemop$, i.e.~it has transverse intersection with
the zero section $\banbund_0\hookrightarrow\mathcal{L}$ then
\[
  \begin{array}{rcccc}
  \det(\psi)
  &=&
  \bigwedge\,\!\!\!^{\mbox{\rm \tiny max}}
  \ker(D_\psi\riemop)
  &\otimes&
  \R^*\\
  &=&
  \bigwedge\,\!\!\!^{\mbox{\rm \tiny max}}
  T_\psi\mathcal{M}_{\phi}
  &\otimes&
  \R^*.
  \end{array}
\]
Thus, a section of the determinant line bundle defines an orientation
of $\mathcal{M}_{\phi}$. These have to be chosen in a coherent
fashion to make $\fglue$ orientation preserving. The gluing construction
gives a natural identification
\[
  \det(\phi_1)
  \wedge
  \det(\phi_2)
  \overset{\cong}{\lra}
  \det(\phi_2\#_t\phi_1).
\]
Since these are all line bundles, this identification makes it possible
to identify sections of $\det([\phi_1])\wedge\det([\phi_2])$ with sections
of $\det([\phi_2*\phi_1])$. With this isomorphism at hand we are 
able to define a coherence condition. Namely, let $\orient(\phi_1)$ 
and $\orient(\phi_2)$ be sections of the determinant line bundles 
of the associated moduli spaces, then we need that under 
the identification given above we have
\begin{equation}
  \orient(\phi_1)\wedge\orient(\phi_2)=\orient(\phi_2*\phi_1).
  \label{orientation}
\end{equation}
In consequence, a {\bf coherent system of orientations} is a 
section $\orient(\phi)$ of the determinant line bundle $\det(\phi)$ 
for each homotopy class of Whitney discs $\phi$ connecting two 
intersection points such that equation (\ref{orientation}) holds 
for each pair for which concatenation makes sense. It is not clear 
if these systems exist in general. By construction with respect
to these coherent systems of orientations the map $\fglue$ is
orientation preserving.\vspace{0.3cm}\\
In the case of Heegaard Floer theory there is an easy way giving a
construction for coherent systems of orientations. Namely, fix a
$\spinc$-structure $s$ and let $\{x_0,\dots,x_l\}$ be the points
representing $s$, i.e.~$(s_z)^{-1}(s)=\{x_0,\dots,x_l\}$.  Let $\phi_1,\dots,\phi_q$ be a set of periodic
classes in $\pitwo(x_0,x_0)$ representing a basis for $H^1(Y;\Z)$,
denote by $\theta_i$ an element of $\pitwo(x_0,x_i)$. A
coherent system of orientations is constructed by choosing sections
over all chosen discs, i.e.~$\orient(\phi_i)$, $i=1,\dots,q$ and
$\orient(\theta_j)$, $j=1,\dots,l$. Namely, for each homotopy 
class $\phi\in\pitwo(x_i,x_j)$ we have a presentation 
(cf.~Lemma \ref{holspheres}, Lemma \ref{kernz} and (\ref{pbident}))
\[
  \phi=a_1\phi_1+\dots+a_q\phi_q+\theta_j-\theta_i
\]
inducing an orientation $\orient(\phi)$. This definition clearly
defines a coherent system.\vspace{0.3cm}\\
To give a proof of Theorem \ref{differential} in case of $\Z$-coefficients
we have to translate orientations on the $0$-dimensional components of 
the moduli spaces $\modhat$ of 
connecting Whitney discs into signs. For $\phi$ with $\mu(\phi)=1$
the translation action naturally induces an orientation on
$\mathcal{M}_\phi$. Comparing this orientation with the coherent
orientation induces a sign. We define the {\bf signed count} as the
count of the elements by taking into account the signs induced
by the comparison of the action orientation with the coherent
orientation.
\begin{proof}[Proof of Theorem \ref{differential} for $\Z$-coefficients]
We stay in the notation of the earlier proof. With the coherent system
of orientations introduced we made the map
\[ \fglue\co
  \modhatxy_{n_z=0}^0
  \times
  \widehat{\mathcal{M}}(y,w)_{n_z=0}^0
  \hookrightarrow
  \partial\widehat{\mathcal{M}}(x,z)_{n_z=0}^1
\]
orientation preserving. Hence, we see that $c(x,w)$ equals 
\[
\#(\modhatxy_{n_z=0}^0
  \times
  \widehat{\mathcal{M}}(y,w)_{n_z=0}^0)
\]
which in turn equals the oriented count of boundary components of
$\partial\widehat{\mathcal{M}}(x,z)_{n_z=0}^1$. Since the space
is $1$-dimensional, this count vanishes.
\end{proof}
\subsubsection{More General Theories}
There are variants of Heegaard Floer homology which do not force 
the condition $n_z=0$. To make the compactification work in that
case we have to take care of boundary degenerations and spheres
bubbling off. Both can be shown to be controlled in the sense
that the proof of Theorem \ref{differential} for the general theories
works the same way with some slight additions due to bubbling and
degenerations. This article mainly focuses on the $\hfhat$-theory,
so we exclude these matters from our exposition. Note just
that we get rid of bubbling by a proper choice of almost complex
structure. By choosing $j$ on $\Sigma$ appropriately there is
a contractible open neighborhood of $sym^g(j)$ in 
$\mathcal{J}(j,\eta,V)$ for which all spheres miss the intersections
$\talpha\cap\tbeta$. Moreover, for a generic choice of
path $(\com_s)_{s\in[0,1]}$ inside this neighborhood the signed count
of degenerate discs is zero. With this information it is easy
to modify the given proof for the general theories. We leave this
to the interested reader or point him to \cite{OsZa01}.

\subsection{Choice of Almost Complex Structure}\label{compstruct}
Let $\Sigma$ be endowed with a complex structure $j$ and let 
$U\subset\Sigma$ be a subset diffeomorphic to a disc. 
\begin{theorem}[Riemann mapping theorem] There is a $3$-dimensional connected
family of holomorphic identifications of $U$ with the unit disc
$\D\subset\C$.
\end{theorem} 
Consequently, suppose that all moduli spaces are compact manifolds 
for the path $(\com_s)_{s\in[0,1]}=sym^g(j)$. In this case we conclude 
from the Riemann mapping theorem the following corollary.
\begin{cor} Let $\phi\co\D^2\lra\symg$ be a holomorphic disc with
$\dom(\phi)$ isomorphic to a disc. Then the moduli space $\modhatphi$
contains a unique element.
\end{cor}
There are several ways to achieve this special situation. We call a 
domain $\dom(\phi)$ {\bf $\alpha$-injective} if all its multiplicities are $0$ or $1$ and
its interior is disjoint from the $\alpha$-circles. We then say that the homotopy
class $\phi$ is {\bf $\alpha$-injective}.
\begin{theorem} Let $\phi\in\pitwo(x,y)$ be an $\alpha$-injective homotopy
class and $j$ a complex structure on $\Sigma$. For generic perturbations
of the $\alpha$-curves the moduli space $\mathcal{M}_{sym^g(j),\phi}$ is
a smooth manifold.
\end{theorem}
In explicit calculations it will be nice to have all homotopy classes
carrying holomorphic representatives to be $\alpha$-injective. In 
this case we can choose the path of almost complex structures in such 
a way that homotopy classes of Whitney discs with disc-shaped domains 
just admits a unique element. This is exactly what can be achieved 
in general to make the $\hfhat$-theory combinatorial. For a class 
of Heegaard diagrams called {\bf nice diagrams} all moduli spaces 
with $\mu=1$ just admit one single element. In addition we have 
a precise description of how these domains look like. In $\Z_2$-coefficients 
with nice diagrams this results in a method of calculating the 
differential $\parhat_z$ by counting the number of domains that 
fit into the scheme. This is successfully done for instance for the
$\hfhat$-theory in \cite{sarwang}.
\begin{definition}[see \cite{sarwang}]\label{nicehd} A pointed 
Heegaard diagram $(\Sigma,\alpha,\beta,z)$ is called {\bf nice} if 
any region not containing $z$ is either a bigon or a square.
\end{definition}
\begin{definition}[see \cite{sarwang}] A homotopy class is called an
empty embedded {\bf $2n$-gon} if it is topologically an embedded disc
with $2n$ vertices at its boundary, it does not contain any $x_i$ or
$y_i$ in its interior, and for each vertex $v$ the average of the
coefficients of the four regions around $v$ is $1/4$.
\end{definition}
For a nice Heegaard diagram one can show that all homotopy classes
$\phi\in H(x,y,1)$ with $\mu(\phi)=1$ that admit holomorphic 
representatives are empty embedded bigons or empty embedded squares.
Furthermore, for a generic choice of $j$ on $\Sigma$ the moduli spaces
are regular under a generic perturbation of the $\alpha$-curves and
$\beta$-curves. The moduli space $\modhatphi$ contains one single element.
Thus, the theory can be computed combinatorially. We note the following
property.
\begin{theorem}[see \cite{sarwang}] Every $3$-manifold admits a 
nice Heegaard diagram.
\end{theorem}

\subsection{Dependence on the Choice of Orientation Systems}
From their definition it is easy to reorder the  orientation systems 
into equivalence classes. The elements in these 
classes give rise to isomorphic homologies. Let $\orient$ and $\orient'$
be two orientation systems. We measure their difference
\[
  \delta\co H^1(Y;\Z)\lra\ztwo
\]
by saying that, given a periodic class $\phi\in\pitwo(x,x)$, we
define $\delta(\phi)=0$ if $\orient(\phi)$ and $\orient'(\phi)$ coincide,
i.e.~define equivalent sections, and $\delta(\phi)=1$, if $\orient(\phi)$
and $\orient'(\phi)$ define non-equivalent sections. Thus, two systems
are equivalent if $\delta=0$. Obviously, there are $2^{b_1(Y)}$ different
equivalence classes of orientation systems. In general the Heegaard Floer
homologies will depend on choices of equivalence classes of orientation
systems. As an illustration we will discuss an example.
\begin{example}\label{example01} 
The manifold $\stwo\times\sone$ admits a Heegaard 
splitting of genus one, namely $(T^2,\alpha,\beta,z)$ where $\alpha$ 
and $\beta$ are two distinct meridians of $T^2$.\vspace{0.3cm}\\
Unfortunately this is not an admissible diagram. By the 
universal coefficient theorem 
\[
  H^2(\stwo\times\sone;\Z)
  \cong 
  Hom(H_2(\stwo\times\sone;\Z),\Z)
  \cong
  Hom(\Z,\Z).
\]
Hence we can interpret $\spinc$-structures as homomorphisms $\Z\lra\Z$.
For a number $q\in\Z$ define $s_q$ to be the $\spinc$-structure whose
associated characteristic class, which we also call $s_q$, is given 
by $s_q(1)=q$. The two curves $\alpha$ and $\beta$ cut the torus into 
two components, where $z$ is placed in one of them. Denote the other
 component with $\dom$. It is easy to see that the homology class
 $\homology(\dom)$ is a generator of $H_2(\stwo\times\sone;\Z)$.
Thus, we have
\[
  \left<c_1(s_q),\homology(\lambda\cdot\dom)\right>
  =
  \left<2\cdot s_q,\homology(\lambda\cdot\dom)\right>
  =2\cdot s_q(\lambda\cdot 1)
  =2\lambda q.
\]
This clearly contradicts the weak admissibility condition. We fix this
problem by perturbing the $\beta$-curve slightly to give a Heegaard diagram
as illustrated in Figure \ref{Fig:figNine}.
\begin{figure}[ht!]
\labellist\small\hair 2pt
\pinlabel {$\alpha$} [l] at 279 334
\pinlabel {$x$} [Bl] at 283 244
\pinlabel {$z$} [l] at 422 219
\pinlabel {$\dom_1$} [l] at 302 183
\pinlabel {$\dom_2$} [l] at 200 310
\pinlabel {$y$} [tl] at 284 121
\pinlabel {$\beta$} [r] at 146 44
\endlabellist
\centering
\includegraphics[width=6cm]{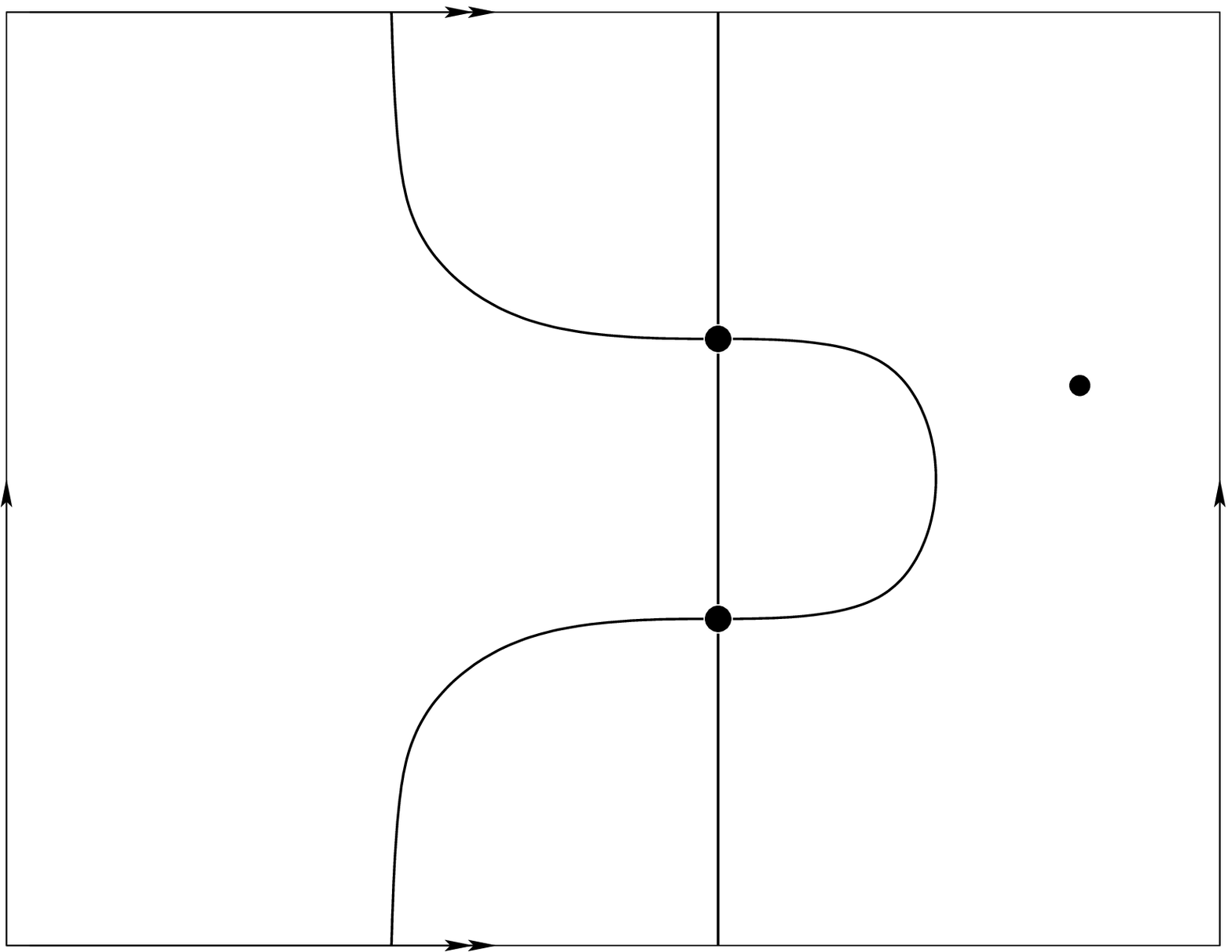}
\caption{An admissible Heegaard diagram for $\stwo\times\sone$.}
\label{Fig:figNine}
\end{figure}
By boundary orientations $\Z\left<(\dom_1-\dom_2)\right>$ are all possible
periodic domains.\vspace{0.3cm}\\
Figure \ref{Fig:figNine} shows that the chain module is generated by the points
$x$ and $y$. A straightforward computation gives $\epsilon(x,y)=0$ (see \S\ref{symproduct} for a definition) 
and, hence, both intersections belong to the same $\spinc$-structure we will denote
by $s_0$. Thus, the chain complex $\cfhat(\Sigma,\alpha,\beta;s_0)$ equals 
$\Z\otimes\{x,y\}$. The regions $\dom_1$ and $\dom_2$ are both disc-shaped
and hence $\alpha$-injective. Thus, the Riemann mapping theorem 
(see \S\ref{compstruct}) gives
\[
  \#\widehat{\mathcal{M}}_{\phi_1}=1\;\;\mbox{\rm and }\#\widehat{\mathcal{M}}_{\phi_2}=1.
\]
These two discs differ by the periodic domain generating 
$H^1(\stwo\times\sone;\Z)$. Thus, we are free to choose the 
orientation on this generator (cf.~\S\ref{structmoduli}). Hence, we may choose the signs 
on $\phi_1$ and $\phi_2$
arbitrarily. Thus, there are two equivalence classes of orientation
systems. We define $\orient^*$ to be the system of orientations
where the signs differ and $\orient_0$ where they are equal. Thus,
we get two different homology theories
\begin{eqnarray*}
  \hfhat(\stwo\times\sone,s_0;\orient*)&=&\Z\oplus\Z\\
  \hfhat(\stwo\times\sone,s_0;\orient_0)&=&\ztwo.
\end{eqnarray*}
\end{example}
However, there is a special choice of coherent orientation
systems. We point the reader to \S\ref{genhomology} for a definition
of $\hfinfty$. Additionally, instead of using $\Z$-coefficients, we
can use the ring $\Z[H_1(Y)]$ as coefficients for defining this
Heegaard Floer group. The resulting group is denoted by $\hfinftwist$. We point 
the reader to \cite{OsZa01} for a precise definition. As
a matter or completeness we cite:
\begin{theorem}[see \cite{OsZa02}, Theorem 10.12]\label{cohblah} 
Let $Y$ be a closed
oriented $3$-manifold. Then there is a unique equivalence class of
orientation system such that for each torsion $\spinc$-structure $s_0$ 
there is an isomorphism
\[
  \hfinftwist(Y,s_0)\cong\Z[U,U^{-1}]
\] 
as $\Z[U,U^{-1}]\otimes_\Z\Z[H^1(Y;\Z)]$-modules.
\end{theorem}

\section{The Homologies $\hfinfty$, $\hfplus$, $\hfminus$}\label{genhomology}
Given a pointed Heegaard diagram $(\Sigma,\alpha,\beta,z)$, we
define $\cfminus(\Sigma,\alpha,\beta,z;s)$ as the free 
$\Z[U^{-1}]$-module generated by the points of intersection
$(s_z)^{-1}(s)\subset\talpha\cap\tbeta$. For an intersection
$x$ we define
\[
  \parminus x
  =
  \sum_{y\in(s_z)^{-1}(s)}
  \sum_{\phi\in\mu^{-1}(1)}
  \#\modhatphi\cdot U^{-n_z(\phi)}y,
\]
where $\mu^{-1}(1)$ are the homotopy classes in $\pitwo(x,y)$ with expected
dimension equal to one. Note that in this theory we do not restrict to
classes with $n_z=0$. This means even with weak admissibility imposed
on the Heegaard diagram the proof of well-definedness as it was done
in \S\ref{roadto} breaks down.
\begin{definition} A Heegaard diagram $(\Sigma,\alpha,\beta,z)$ is called 
\textbf{strongly admissible} for the $\spinc$-structure $s$ if 
for every non-trivial periodic domain $\dom$ such that
$\left<c_1(s),H(\dom)\right>=2n\geq 0$ the domain $\dom$ has some 
coefficient greater than $n$.
\end{definition}
Imposing strong admissibility on the Heegaard diagram we can prove
well-definedness by showing that only finitely many homotopy classes
of Whitney discs contribute to the moduli 
space $\moduli$ (cf.~\S\ref{roadto}).
\begin{theorem}\label{newdifferential} The 
map $\parminus$ is a differential.
\end{theorem}
As mentioned in \S\ref{roadto}, in this more general case we have to
take a look at bubbling and degenerate discs. The proof follows the
same lines as the proof of Theorem \ref{differential}. With the
remarks made in \S\ref{roadto} it is easy to modify the given proof
to a proof of Theorem \ref{newdifferential} (see \cite{OsZa01}). We
define
\[
  \cfinfty(\Sigma,\alpha,\beta;s)
  =
  \cfminus(\Sigma,\alpha,\beta;s)\otimes_{\Z[U^{-1}]}\Z[U,U^{-1}]
\]
and denote by $\parinfty$ the induced differential. From the definition
we get an inclusion of $\cfminus\hookrightarrow\cfinfty$ whose 
cokernel is defined as $\cfplus(\Sigma,\alpha,\beta,s)$. Finally we
get back to $\cfhat$ by
\[
  \cfhat(\Sigma,\alpha,\beta;s)
  =
  \frac{U\cdot\cfminus(\Sigma,\alpha,\beta;s)}
  {\cfminus(\Sigma,\alpha,\beta;s)}.
\]
The associated homology theories are denoted by $\hfinfty$, $\hfminus$
and $\hfhat$. There are two long exact sequences which can be derived
easily from the definition of the Heegaard Floer homologies. To give
an intuitive picture look at the following illustration:
\[
  \begin{array}{lcccccccccc}
  \cfinfty & = & \dots & U^{-3} & U^{-2} &
  U^{-1} & U^{0} & U^{1} & U^{2} & U^{3} & \dots \\
  \cfminus & = & \dots & U^{-3} & U^{-2} &
  U^{-1} &       &       &       &       &      \\
  \cfhat   & = &       &        &        &
         & U^{0} &       &       &       &      \\
  \cfplus  & = &       &        &        &
         & U^{0} & U^{1} & U^{2} & U^{3} & \dots
  \end{array}
\]
We see why the condition
of weak admissibility is not strong enough to give a well-defined 
differential on $\cfinfty$ or $\cfminus$. However, weak admissibility 
is enough to make the differential on $\cfplus$ well-defined, since 
the complex is bounded from below with respect to the obvious
filtration given by the $U$-variable.
\begin{lem}\label{genseq} There are two long exact sequences
\begin{diagram}[size=1.5em,labelstyle=\scriptstyle]
\dots & \rTo & \hfminus(Y;s) 
& \rTo & \hfinfty(Y;s)&\rTo & \hfplus(Y;s)&\rTo & \dots\\
\dots & \rTo & \hfhat(Y;s) & \rTo & \hfplus(Y;s)&\rTo & 
\hfplus(Y;s)&\rTo & \dots,
\end{diagram}
where $s$ is a $\spinc$-structure of $Y$.
\end{lem}
The explicit description illustrated above can be derived directly
from the definition of the complexes. We leave this to the interested
reader (see also \cite{OsZa01}).

\section{Topological Invariance}\label{topoinvariance}
Given two Heegaard diagrams $(\Sigma,\alpha,\beta)$ and 
$(\Sigma',\alpha',\beta')$ of a manifold $Y$, they are equivalent
after a finite sequence of isotopies of the attaching circles,
handle slides of the $\alpha$-curves and $\beta$-curves and
stabilizations/destabilizations. Two Heegaard diagrams are equivalent
if there is a diffeomorphism of the Heegaard surface interchanging
the attaching circles. Obviously, equivalent Heegaard diagrams
define isomorphic Heegaard Floer theories. To show that Heegaard
Floer theory is a topological invariant of the manifold $Y$
we have to see that each of the moves, i.e.~isotopies, handle
slides and stabilization/destabilizations yield isomorphic
theories. We will briefly sketch the topological invariance.
This has two reasons: First of all the invariance proof
uses ideas that are standard in Floer homology theories and
hence appear frequently. The ideas provided from the invariance proof
happen to be the standard techniques for proving 
exactness of sequences, proving invariance properties, and proving the
existence of morphisms between Floer homologies. Thus, knowing the
invariance proof, at least at the level of ideas, is crucial
for an understanding of most of the papers published in this
field. We will deal with the $\hfhat$-case and and point the reader to 
\cite{OsZa01} for a general treatment.\vspace{0.3cm}\\
The invariance proof contains several steps. We start showing invariance
under the choice of path of admissible almost complex structures.
Isotopies of the attaching circles are split up into two separate
classes: Isotopies that generate/cancel intersection points and
those which do not change the chain module. The invariance under the
latter Heegaard moves immediately follows from the independence of the choice of
almost complex structures. Such an isotopy is carried by an ambient
isotopy inducing an isotopy of the symmetric product. We perturb 
the almost complex structure and thus interpret the isotopy as 
a perturbation of the almost complex structure. The former Heegaard moves have
to be dealt with separately. We mimic the generation/cancellation
of intersection points with a Hamiltonian isotopy and with it 
explicitly construct an isomorphism of the respective homologies
by counting discs with dynamic boundary conditions. Stabilizations/
destabilizations is the easiest part to deal with: it follows
from the behavior of the Heegaard Floer theory under connected
sums. Finally, handle slide invariance will require us to define
 what can be regarded as the Heegaard Floer homological version of the
pair-of-pants product in Floer homologies. This product has two
nice applications. The first is the invariance under handle
slides and the second is the association of maps to cobordisms
giving the theory the structure of a topological field theory.\vspace{0.3cm}\\

\subsection{Stabilizations/Destabilizations}\label{stabilinvariance}
We determine the groups $\hfhat(\stwo\times\sone\#\stwo\times\sone)$ as a model
calculation for how the groups behave under connected sums.
\begin{figure}[ht!]
\labellist\small\hair 2pt
\pinlabel {$z$} [l] at 240 141
\pinlabel {$x_1$} [l] at 94 116
\pinlabel {$y_1$} [l] at 385 120
\pinlabel {$x_2$} [l] at 89 74
\pinlabel {$y_2$} [l] at 380 74
\pinlabel {$\dom_1$} [t] at 161 26
\pinlabel {$\dom_2$} [l] at 105 10
\pinlabel {$\dom_3$} [tr] at 265 68
\pinlabel {$\dom_4$} [r] at 320 12
\endlabellist
\centering
\includegraphics[width=12cm]{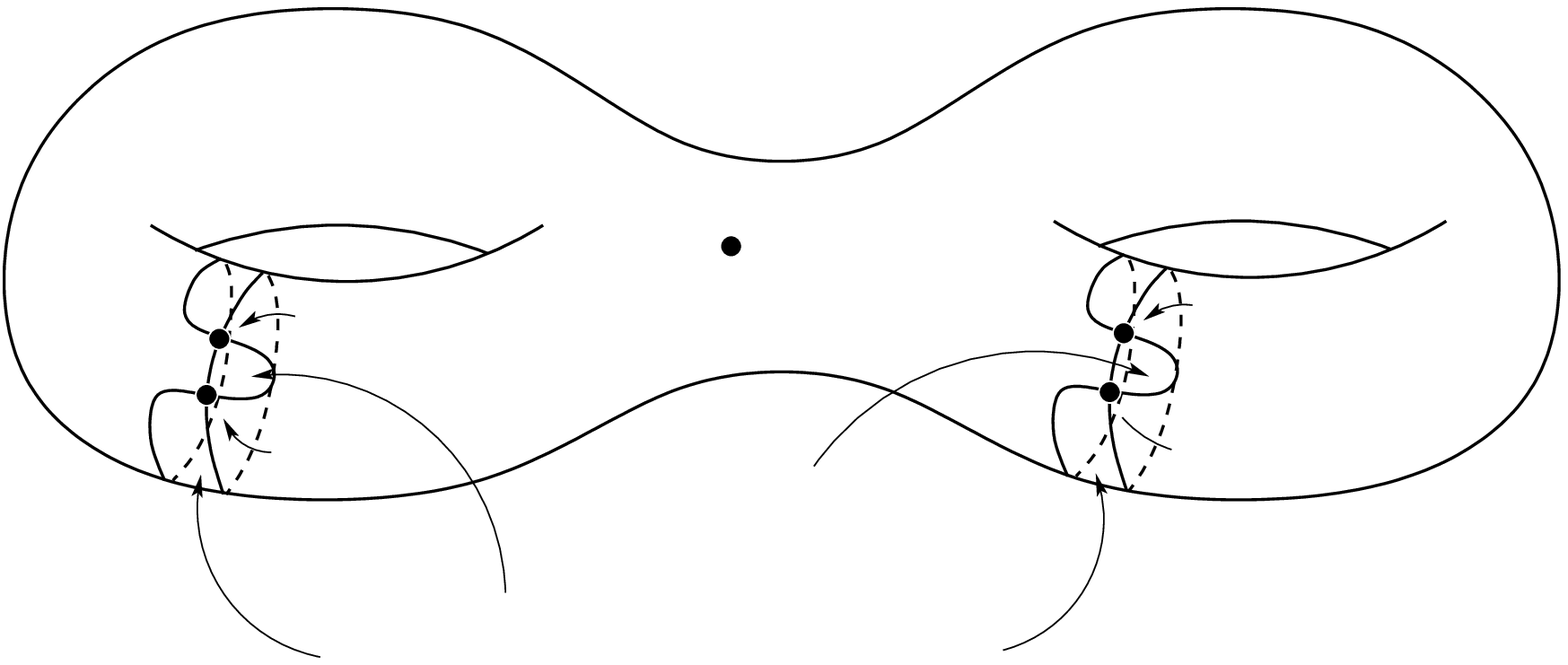}
\caption{An admissible Heegaard diagram for $\stwo\times\sone\#\stwo\times\sone$.}
\label{Fig:figTen}
\end{figure}

\begin{example}\label{exam01} We fix admissible Heegaard 
diagrams $(T^2_i,\alpha_i,\beta_i)$ $i=1,2$ for $\stwo\times\sone$ 
as in Example \ref{example01}. To perform the 
connected sum of $\stwo\times\sone$ with itself we choose 
$3$-balls such that their intersection $D$ with the Heegaard 
surface fulfills the property
\[
  \left.\mathcal{J}_s^i\right|_{D}=\mbox{\rm sym}(j_i).
\]
Figure \ref{Fig:figTen} pictures the Heegaard diagram we get for
the connected sum. Denote by $T$ a small connected sum tube inside $\Sigma=T^2_1\#T^2_2$. 
By construction the induced almost complex structure equals
\[
  \left.(\com^1\#\com^2)_s\right|_{T\times\Sigma}=sym^2(j^1\#j^2).
\]
All intersection points belong to the same $\spinc$-structure $s_0$.
For suitable $\spinc$-structures $s_1$, $s_2$ on $\stwo\times\sone$
we have that $s_0=s_1\#s_2$ and
\[
  \cfhat(\Sigma,\alpha,\beta,s_1\#s_2)
  =
  \Z
  \otimes\{(x_i,y_j)\,|\,i,j\in\{1,2\}\}
  \cong\cfhat(T^2_1,s_1)
  \otimes
  \cfhat(T^2_2,s_2).
\]
The condition $n_z=0$ implies that for every holomorphic disc
$\phi\co\D^2\lra\symg$ the low-dimensional model (cf.~\S\ref{roadto})
$\phihat\co\Dhat\lra\Sigma$ stays away from the tube $T$. Consequently
we can split up $\Dhat$ into
\[
  \Dhat=\Dhat_1\sqcup\Dhat_2,
\]
where $\Dhat_i$ are the components containing the preimage
$(\phihat)^{-1}(T^2_i\backslash D)$. Restriction to these components
determines maps $\phihat_i\co\Dhat_i\lra T^2_i$ inducing Whitney discs
$\phi_i$ in the symmetric product $\mbox{\rm Sym}^1(T^2)$. Thus, the 
moduli spaces split:
\[
 \begin{array}{rcl}
 \M_{(\com^1\#\com^2)_s}((x_i,y_k),(x_j,y_l))_{n_z=0}&
 \overset{\cong}{\lra}&
 \M_{\com_s^1}(x_i,x_j)_{n_z=0}\times\M_{\com_s^2}(y_k,y_l)_{n_z=0}\\
 \phi &\lmt& (\phi_1,\phi_2).
 \end{array}
\]
For moduli spaces with expected dimension $\mu=1$, a dimension
count forces one of the factors to be constant. So, the differential 
splits, too, i.e.~for $a_i\in\cfhat(T^2_i,s_i)$, $i=1,2$
we see that
\[
  \parhat_{(\com^1\#\com^2)_s}(a_1\otimes a_2)
  =
  \parhat_{\com_s^1}(a_1)\otimes a_2 
  + 
  a_1\otimes\parhat_{\com_s^2}(a_2).
\]
And consequently
\[
  \hfhat(\stwo\times\sone\#\stwo\times\sone,s_1\#s_2;\orient_1\wedge\orient_2)
  \cong
  \hfhat(\stwo\times\sone,s_1;\orient_1)
  \otimes
  \hfhat(\stwo\times\sone,s_2;\orient_2).
\]
\end{example}
The same line of arguments shows the general statement.
\begin{theorem}[see \cite{OsZa02}]\label{consum} For 
closed, oriented $3$-manifolds $Y_i$, $i=1,2$ the 
Heegaard Floer homology of the connected sum 
$Y_1\#Y_2$ equals the tensor product of the 
Heegaard Floer homologies of the factors, i.e.
\[
  \hfhat(Y_1\#Y_2)=H_*(\cfhat(Y_1)\otimes\cfhat(Y_2)),
\]
where the chain complex on the right carries the natural 
induced boundary.
\end{theorem}
Stabilizing a Heegaard diagram of $Y$ means, on the manifold level, to 
do a connected sum with $\sthree$. We know that $\hfhat(\sthree)=\Z$.
By the classification of finitely generated abelian groups
and the behavior of the tensor product, invariance
follows.

\subsection{Independence of the Choice of Almost Complex Structures}
\label{acsinvariance}
Suppose we are given a $1$-dimensional family of paths of 
$(j,\eta,V)$-nearly symmetric almost complex 
structures $(\com_{s,t})$. Given a Whitney disc $\phi$, we
define $\modfamilyphi$ as the moduli space of Whitney discs
in the homotopy class of $\phi$ which satisfy the equation
\[
  \frac{\partial\phi}{\partial s}(s,t)
  +
  \com_{s,t}
  \bigl(
   \frac{\partial\phi}{\partial t}(s,t)
  \bigr)
  =0.
\]
Observe that there is no free translation action on the moduli
spaces as on the moduli spaces we focused on while discussing
the differential $\parhat_z$. We define a map $\Phihat_{\modfamily}$
between the theories $(\cfhat(\Sigma,\alpha,\beta,z),\parhat_{\com_{s,i}})$
for $i=0,1$ by defining for $x\in\talpha\cap\tbeta$
\[
  \Phihat_{\com_{s,t}}(x)
  =
  \sum_{y\in\talpha\cap\tbeta}
  \sum_{\phi\in H(x,y,0)}
  \#\modfamilyphi\cdot y,
\]
where $H(x,y,0)\subset\pitwo(x,y)$ are the homotopy classes with
expected dimension $\mu=0$ and intersection number $n_z=0$.
There is an energy bound for all holomorphic Whitney discs
which is independent of the particular Whitney disc or its
homotopy class (see \cite{OsZa01}). Thus, the moduli spaces 
are Gromov-compact manifolds, i.e.~can be compactified by 
adding solutions coming from broken discs, bubbling of spheres 
and boundary degenerations (cf.~\S\ref{structmoduli}).
Since we stuck to the $\hfhat$-theory we impose the condition
$n_z=0$ which circumvents bubbling of spheres and boundary degenerations
(see \S\ref{structmoduli}).\vspace{0.3cm}\\
To check that $\Phihat$ is a chain map, we compute
\begin{eqnarray*}
  \parhat_{J_{s,1}}\circ\Phihat_{J_{s,t},z}(x)
  -
  \Phihat_{J_{s,t}}\circ\parhat_{J_{s,0},z}(x)
  &=&
  \sum_{\underset{\phi\in H(x,y,0),\psi\in H(y,z,1)}{y,z}}
  \#\modulit(\phi)\#\modhatone(\psi) z\
  \\
  &&-\sum_{\underset{\phi\in H(x,y,1),\psi\in H(y,z,0)}{y,z}}
  \#\modhatzero(\phi)\#\modulit(\psi)z\\
  &=&
  \sum_z c(x,z)\cdot z.
\end{eqnarray*}
The coefficient $c(x,z)$ is given by
\begin{equation}
  \sum_{y,I}
  \bigl(\#\modfamilyphi\cdot\#\widehat{\modspace}_{\com_{s,1},\psi}
  -
  \#\widehat{\modspace}_{\com_{s,0},\widetilde{\psi}}
  \cdot
  \#\modspace_{\com_{s,t},\widetilde{\phi}}
  \bigr),\label{coefficients}
\end{equation}
where $I$ consists of pairs 
\[
  (\phi,\widetilde{\phi})\in H(x,y,0)\times H(y,z,0)
  \;\text{ and }
  (\psi,\widetilde{\psi})\in H(x,y,1)\times H(y,z,1).
\]
Looking at the ends of the moduli spaces $\modulit(\eta)$ for an 
$\eta\in H(x,z,1)$, the gluing construction (cf.~\S\ref{structmoduli})
together with the compactification argument mentioned earlier provides
the following ends:
\begin{equation}
  \Bigl(
  \bigsqcup_{\eta=\psi*\phi}
  (\modulit(\phi)\times\modhatone(\psi))
  \Bigr)
  \sqcup
  \Bigl(
  \bigsqcup_{\eta=\widetilde{\psi}*\widetilde{\phi}}
  (\modhatzero(\widetilde{\psi})\times\modulit(\widetilde{\phi}))
  \Bigr),\label{ends}
\end{equation}
where the expected dimensions of $\phi$ and $\widetilde{\phi}$ are $1$ and 
of $\psi$ and $\widetilde{\psi}$ they are $0$. A signed count of (\ref{ends}) precisely
reproduces (\ref{coefficients}) and hence $c(x,z)=0$ -- at least in 
$\ztwo$-coefficients. To make this work in general, i.e.~with coherent orientations,
observe that we have the following condition imposed on the sections:
\[
  \orient_{s,t}(\phi)
  \wedge
  \orient_1(\psi)
  =
  -\orient_0(\widetilde{\psi})
  \wedge
  \orient_{s,t}(\widetilde{\phi}).
\]
We get an identification of orientation systems, $\xi$ say, such that $\Phi$ is
a chain map between 
\[
  (\cfhat(\Sigma,\alpha,\beta,z),\parhat_{\com_{s,0}}^{\,\orient})
  \lra
  (\cfhat(\Sigma,\alpha,\beta,z),\parhat_{\com_{s,1}}^{\,\xi(\orient)}).
\]
We reverse the direction of the isotopy and define a
map $\Phihat_{\com_{s,1-t}}$. The compositions
\[
  \Phihat_{\com_{s,1-t}}
  \circ
  \Phihat_{\com_{s,t}}
  \;\;
  \mbox{\rm and }
  \;
  \;
  \Phihat_{\com_{s,t}}
  \circ
  \Phihat_{\com_{s,1-t}}
\]
are both chain homotopic to the identity. In the following we will discuss
the chain homotopy equivalence for the map  
$\Phihat_{\com_{s,t}}\circ\Phihat_{\com_{s,1-t}}$.\vspace{0.3cm}\\
Define a path
$\com_{s,t}(\tau)$ such that $\com_{s,t}(0)=\com_{s,t}*\com_{s,1-t}$ and
$\com_{s,t}(1)=\com_{s,0}$. The existence of this path follows from the
fact that we choose the paths inside a contractible set 
(cf.~\S\ref{structmoduli} or see \cite{OsZa01}). Define 
the moduli space
\[
  \modulittaubigphi=\bigcup_{\tau\in[0,1]}\modulittauphi.
\]
\begin{theorem}\label{fammoduli} Let 
$\com_{(t_1,\dots,t_n)}$ be an $n$-parameter family 
of generic almost complex structures and $\phi$ a homotopy class of 
Whitney discs with expected dimension $\mu(\phi)$. Then $\mathcal{M}$,
defined as the union of $\mathcal{M}_{\com_{(t_1,\dots,t_n)},\phi}$
over all $\com_{(t_1,\dots,t_n)}$ in the family, is a manifold of 
dimension $\mu(\phi)+n$.
\end{theorem}
There are two types of boundary components:
the one type of boundary component coming from variations of the Whitney disc 
$\phi$ which are breaking, bubbling or degenerations
and the other type of ends coming from variations of the almost complex 
structure.\vspace{0.3cm}\\
We define a map
\[
  \Hhat_{\com_{s,t}(\tau)}(x)
  =
  \sum_{y\in\talpha\cap\tbeta}
  \sum_{\phi\in H(x,y,-1)}
  \#\modulittauphi
  \cdot
  y,
\]
where $H(x,y,-1)\subset\pitwo(x,y)$ are the homotopy classes $\phi$
with $n_z(\phi)=0$ and expected dimension $\mu(\phi)=-1$. According
to Theorem \ref{fammoduli}, the manifold $\modulittauphi$ is
$0$-dimensional. We claim that $\Hhat$ is a chain homotopy between
$\Phihat_{\com_{s,t}}\circ\Phihat_{\com_{s,1-t}}$ and the identity.
By definition, the equation
\begin{equation}
  \Phihat_{\com_{s,t}}
  \circ\Phihat_{\com_{s,1-t}}
  -
  \mbox{\rm id}
  -
  (\parhat_{\com_{s,0}}\circ 
  \Hhat_{\com_{s,t}(\tau)}
  +
  \Hhat_{\com_{s,t}(\tau)}\circ\parhat_{\com_{s,1}})
  =0\label{chheq}
\end{equation}
has to hold. 
Look at the ends of $\modulittaubig(\psi)$ for $\mu(\psi)=0$.
This is a $1$-dimensional space, and there are the ends 
\[
  \Bigl(\bigsqcup_{\psi=\eta*\phi}
  \widehat{\modspace}_{\com_{s,0},\eta}
  \times
  \modspace_{\com_{s,t}(\tau),\phi}
  \Bigr)
  \sqcup
  \Bigl(
  \bigsqcup_{\psi=\widetilde{\eta}*\widetilde{\phi}}
  \modspace_{\com_{s,t}(\tau),\widetilde{\eta}}
  \times
  \widehat{\modspace}_{\com_{s,1},\widetilde{\phi}}
  \Bigr)
\]
coming from variations of the Whitney disc, and the ends 
\[
  \modspace_{\com_{s,t}(0),\psi}
  \sqcup
  \modspace_{\com_{s,t}(1),\psi}
\]
coming from variations of the almost complex structure.
These all together precisely produce the coefficients in equation (\ref{chheq}).
Thus, the Floer homology is independent of the choice of $(j,\eta,V)$-nearly
symmetric path. Variations of $\eta$ and $V$ just change the contractible
neighborhood $\mathcal{U}$ around $\sym^g(j)$ containing the admissible
almost complex structures. So, the theory is independent of these 
choices, too. A $j'$-nearly symmetric path can be approximated by a 
$j$-symmetric path given that $j'$ is close to $j$. The set of complex
structures on a surface $\Sigma$ is connected, so step by step one can
move from a $j$-symmetric path to any $j'$-symmetric path.

\subsection{Isotopy Invariance}\label{isotopyinvariance}
Every isotopy of an attaching circle can be divided into two classes: 
creation/anhillation of pairs of intersection points and isotopies 
not affecting transversality. An isotopy of an $\alpha$-circle
of the latter type induces an isotopy of $\talpha$ in the symmetric
product. Compactness of the $\talpha$ tells us that there is an
ambient isotopy $\phi_t$ carrying the isotopy. With this isotopy
we perturb the admissible path of almost complex structures as
\[
  \widetilde{\com}_s=(\phi_1^{-1})_*\circ\com_s\circ(\phi_1)_*
\]
giving rise to a path of admissible almost complex structures.
The diffeomorphism $\phi_1$ induces an identification of the
chain modules. The moduli spaces defined by $\com_s$
and $\widetilde{\com}_s$ are isomorphic. Hence
\begin{equation}
  H_*(\cfhat(\Sigma,\alpha,\beta),\parhat_z^{\com_s})
  =
  H_*(\cfhat(\Sigma,\alpha',\beta),\parhat_z^{\widetilde{\com}_s})
  =
  H_*(\cfhat(\Sigma,\alpha',\beta),\parhat_z^{\com_s}),
  \label{eqjinvar}
\end{equation}
where the last equality follows from the considerations in 
\S\ref{acsinvariance}. This chain of equalities shows that 
the isotopies discussed can be interpreted as variations of
the almost complex structure.\vspace{0.3cm}\\
The creation/cancellation of pairs of intersection points is
done with an exact Hamiltonian isotopy supported in a small
neighborhood of two attaching circles. We cannot use the methods
from \S\ref{acsinvariance} to create an isomorphism between the associated
Floer homologies. At a 
certain point the isotopy violates transversality
as the attaching tori do not intersect transversely.  
Thus, the arguments of \S\ref{acsinvariance} for the right equality in (\ref{eqjinvar}) 
break down.\vspace{0.3cm}\\
Consider an exact Hamiltonian isotopy $\psi_t$ of an $\alpha$-curve
generating a canceling pair of intersections with a $\beta$-curve.
We will just sketch the approach used in this context, since the
ideas are similar to the ideas introduced 
in \S\ref{acsinvariance}.\vspace{0.3cm}\\
Define $\pitwoham(x,y)$ as the set of Whitney discs $\phi$ with dynamic
boundary conditions in the following sense:
\begin{eqnarray*}
  \phi(i) &=&x, \\
  \phi(-i)&=&y, \\
  \phi(0+it)&\in&\Psi_t(\talpha) \\
  \phi(1+it)&\subset&\tbeta
\end{eqnarray*}
for all $t\in\R$. Spoken geometrically, we follow the isotopy with
the $\alpha$-boundary of the Whitney disc. Correspondingly, we define
the moduli spaces of $\com_s$-holomorphic Whitney discs with dynamic 
boundary conditions as $\moduliiso(x,y)$. For $x\in\talpha\cap\tbeta$
define
\[
  \Isotopy(x)
  =
  \sum_{y\in\talpha\cap\tbeta}
  \sum_{\phi\in H_t(x,y,0)}
  \#\moduliiso_{\com_s,\phi}\cdot y,
\]
where $H_t(x,y,0)\subset\pitwoham(x,y)$ are the homotopy classes with
expected dimension $\mu=0$ and $n_z=0$. Using the low-dimensional
model introduced in \S\ref{roadto}, Ozv\'{a}th and Szab\'{o} prove
the following property.
\begin{theorem}[see \cite{OsZa01}, \S7.3] There exists a $t$-independent 
energy bound for holomorphic Whitney discs independent of its homotopy 
class.
\end{theorem}
The existence of this energy bound shows that there are 
Gromov compactifications of the moduli spaces of Whitney discs with 
dynamic boundary conditions. 
\begin{theorem} The map $\Isotopy$ is a chain map. Using the
inverse isotopy we define $\Isotopyinverse$ such that the compositions
$\Isotopy\circ\Isotopyinverse$ and $\Isotopyinverse\circ\Isotopy$ are
chain homotopic to the identity.
\end{theorem}
The proof follows the same lines as in \S\ref{acsinvariance}. We
leave the proof to the interested reader.

\subsection{Handle slide Invariance}\label{parhsinvar}
\subsubsection{The Pair-of-Pants Product}\label{popprod}
In this paragraph we will introduce the Heegaard Floer incarnation
of the pair-of-pants product and with it associate to cobordisms maps
between the Floer homologies of their boundary components.
In case the cobordisms are induced by handle slides the associated
maps are isomorphisms on the level of homology. The maps we will
introduce will count holomorphic triangles in the symmetric product
with appropriate boundary conditions. We have to discuss well-definedness
of the maps and that they are chain maps. To do that we have to
follow similar lines as it was done for the differential. Because
of the strong parallels we will shorten the discussion here. We
strongly advise the reader to first read \S\ref{roadto} before
continuing.
\begin{definition} A set of data $(\Sigma,\alpha,\beta,\gamma)$, where
$\Sigma$ is a surface of genus $g$ and $\alpha$, $\beta$, $\gamma$ three
sets of attaching circles, is called a {\bf Heegaard triple diagram}.
\end{definition}
We denote the $3$-manifolds determined be taking pairs of these attaching
circles as $Y_{\alpha\beta}$, $Y_{\beta\gamma}$ and $Y_{\alpha\gamma}$.
We fix a point $z\in\Sigma\backslash\{\alpha\cup\beta\cup\gamma\}$ and
define a product
\[
  \fhat_{\alpha\beta\gamma}
  \co
  \cfhat(\Sigma,\alpha,\beta,z)
  \otimes
  \cfhat(\Sigma,\beta,\gamma,z)
  \lra
  \cfhat(\Sigma,\alpha,\gamma,z)
\]
by counting holomorphic triangles with suitable boundary conditions: A
{\bf Whitney triangle} is a map $\phi\co\Delta\lra\symg$ with boundary
conditions as illustrated in Figure \ref{Fig:figFourteen}. We call the
respective boundary segments its {\bf $\alpha$-, $\beta$- and $\gamma$-boundary}.
The boundary points, as should be clear from the picture, are
$x\in\talpha\cap\tbeta$, $w\in\talpha\cap\tgamma$ and $y\in\tbeta\cap\tgamma$.
The set of homotopy classes of Whitney discs connecting 
$x$, $w$ and $y$ is denoted by $\pitwo(x,y,w)$.
\begin{figure}[ht!]
\labellist\small\hair 2pt
\pinlabel {$w$} [tr] at 7 34
\pinlabel {$x$} [B] at 126 235
\pinlabel {$y$} [tl] at 247 27
\pinlabel {$\talpha$} [Br] at 71 136
\pinlabel {$\tbeta$} [Bl] at 191 136
\pinlabel {$\tgamma$} [t] at 126 26
\endlabellist
\centering
\includegraphics[width=4cm]{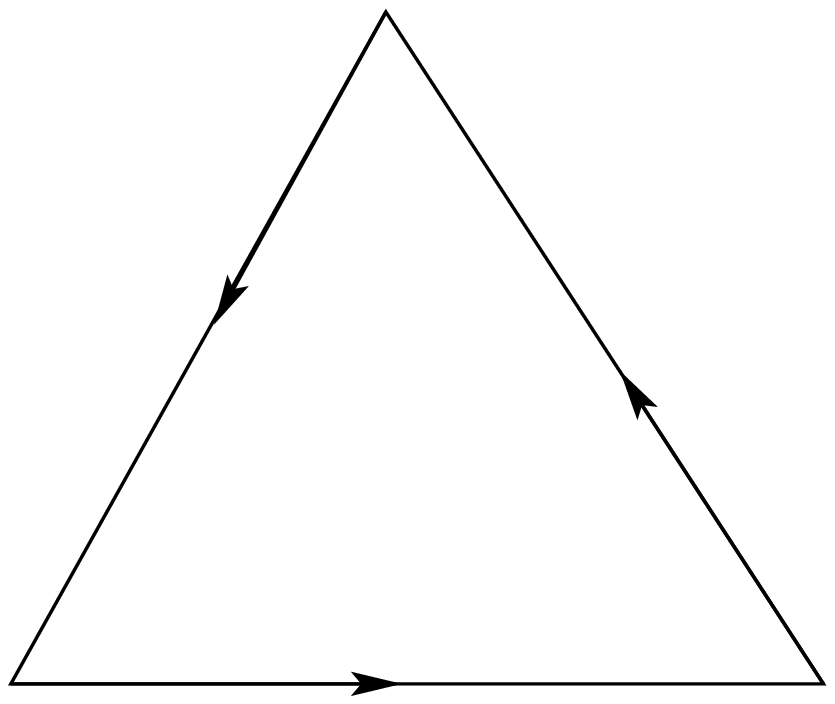}
\caption{A Whitney triangle and its boundary conditions.}
\label{Fig:figFourteen}
\end{figure}

Denote by $\modtriangle_\phi$ the moduli space of holomorphic triangles in
the homotopy class of $\phi$. Analogous to the case of discs we denote
by $\mu(\phi)$ its expected/formal dimension. For $x\in\talpha\cap\tbeta$
define
\[
  \fhat_{\alpha\beta\gamma}(x\otimes y)
  =
  \sum_{w\in\talpha\cap\tgamma}
  \sum_{\phi\in H(x,y,w,0)}\#\modtriangle_\phi\cdot w,
\]
where $H(x,y,w,0)\subset\pitwo(x,y,w)$ is the subset with $\mu=0$
and $n_z=0$. The set of homotopy classes of Whitney discs fits into an
exact sequence
\begin{equation}
  0
  \lra
  \pitwo(\symg)
  \lra
  \pitwo(x,y,w)
  \lra
  \ker(n_z)
  \lra 
  0, \label{exseq2}
\end{equation}
where $n_z$ provides a splitting for the sequence.
We define
\[
  X_{\alpha\beta\gamma}
  =
  \frac
  {
   (\Delta\times\Sigma)
   \cup 
   e_\alpha\times U_\alpha
   \cup 
   e_\beta\times U_\beta
   \cup 
   e_\gamma\times U_\gamma
  }
  {
   (e_\alpha\times\Sigma)
   \sim
   (e_\alpha\times\partial U_\alpha), 
   (e_\beta\times\Sigma)
   \sim
   (e_\beta\times\partial U_\beta),
   (e_\gamma\times\Sigma)
   \sim
   (e_\gamma\times\partial U_\gamma)
  },  
\]
where $U_\alpha$, $U_\beta$ and $U_\gamma$ are the handlebodies determined
by the $2-$handles associated to the attaching circles $\alpha$, $\beta$
and $\gamma$, and $e_\alpha$, $e_\beta$ and $e_\gamma$ are the edges of the
triangle $\Delta$. The manifold $X_{\alpha\beta\gamma}$ is $4$-dimensional
with boundary
\[
  \partial X_{\alpha\beta\gamma}
  =
  Y_{\alpha\beta}
  \sqcup
  Y_{\beta\gamma}
  \sqcup
  -Y_{\alpha\gamma}.
\]
\begin{lem}\label{boring} The kernel of $n_z$ equals 
$H_2(X_{\alpha\beta\gamma};\Z)$
\end{lem}
Combining (\ref{exseq2}) with Lemma \ref{boring} we
get an exact sequence
\begin{equation}
  0
  \lra
  \pitwo(\symg)
  \lra
  \pitwo(x,y,w)
  \overset{\mathcal{H}}{\lra}
  H_2(X_{\alpha\beta\gamma};\Z)
  \lra 
  0, \label{exseq3}
\end{equation}
where $\mathcal{H}$ is defined similarly as for discs (cf.~\S\ref{symproduct}). 
Of course there is a low-dimensional model for triangles and the discussion we
have done for discs carries over verbatim for triangles. The condition
$n_z=0$ makes the product $f_{\alpha\beta\gamma}$ well-defined
in case $H_2(X_{\alpha\beta\gamma};\Z)$ is trivial. Analogous to our
discussion for Whitney discs and the differential, we have to include 
a condition controlling the periodic triangles, i.e.~the triangles 
associated to elements in $H_2(X_{\alpha\beta\gamma};\Z)$. A domain 
$\dom$ of a triangle is called {\bf triply-periodic}
if its boundary consists of a sum of $\alpha$-,$\beta$- and $\gamma$-curves
such that $n_z=0$.
\begin{definition} A pointed triple diagram $(\Sigma,\alpha,\beta,\gamma,z)$
is called {\bf weakly admissible} if all triply-periodic domains $\dom$
which can be written as a sum of doubly-periodic domains have both positive 
and negative coefficients.
\end{definition}
This condition is the natural transfer of weak-admissibility from discs
to triangles. One can show that for given $j,k\in\Z$ there exist
just a finite number of Whitney triangles $\phi\in\pitwo(x,y,w)$ with
$\mu(\phi)=j$, $n_z(\phi)=k$ and $\dom(\phi)\geq0$.\vspace{0.3cm}\\
For a given homotopy class $\psi\in\pitwo(x,y,w)$ with $\mu(\psi)=1$ 
we compute the ends by shrinking a properly embedded arc to a 
point (see the description of convergence in \S\ref{structmoduli}). There 
are three different ways to do this in a triangle. Each time we get 
a concatenation of a disc with a triangle. By boundary orientations we 
see that each of these boundary components contributes to one of the 
terms in the following sum
\begin{equation}
  \fhat_{\alpha\beta\gamma}\circ(\parhat_{\alpha\beta}(x)\otimes y)
  +
  \fhat_{\alpha\beta\gamma}\circ(x\otimes\parhat_{\beta\gamma}(y))
  -\parhat_{\alpha\gamma}\circ\fhat_{\alpha\beta\gamma}(x\otimes y).
  \label{productchain}
\end{equation}
Conversely, the coefficient at any of these terms is given by a
product of signed counts of moduli spaces of discs and moduli
spaces of triangles and hence -- by gluing -- comes from one of 
these contributions. The sum in (\ref{productchain}) vanishes, showing
that $\fhat_{\alpha\beta\gamma}$ descends to a pairing 
$\fhat^*_{\alpha\beta\gamma}$ between the Floer homologies.
\subsubsection{Holomorphic rectangles}
Recall that the set of biholomorphisms of the unit disc is
a $3$-dimensional connected family. If we additionally 
fix a point we decrease the dimension of that family by one.
A better way to formulate this is to say that the set of
biholomorphishms of the unit disc with one fixed point is a
$2$-dimensional family. Fixing two further points reduces
to a $0$-dimensional set. If we additionally fix a fourth
point the rectangle together with these four points uniquely
defines a conformal structure. Variation of the fourth point
means a variation of the conformal structure. Indeed one can
show that there is a uniformization of a holomorphic 
rectangle, i.e.~a rectangle with fixed conformal structure,
which we denote by $\square$,
\[
  \square\lra[0,l]\times[0,h],
\]
where the ratio $l/h$ uniquely determines the conformal structure.
With this uniformization we see that $\modspace(\square)\cong\R$.
The uniformization is area-preserving and converging to one of the 
ends of $\modspace(\square)$ means to stretch the rectangle 
infinitely until it breaks at the end into a concatenation of two 
triangles.
\begin{theorem} Given another set of attaching circles 
$\delta$ defining a map $\fhat_{\alpha\gamma\delta}$, the 
following equality holds:
\begin{equation}
  \fhat^*_{\alpha\beta\gamma}
  (\fhat^*_{\alpha\gamma\delta}(\,\cdot\otimes\,\cdot)
  \otimes
  \,\cdot)
  -
  \fhat^*_{\alpha\beta\delta}
  (\,\cdot\otimes
  \fhat^*_{\beta\gamma\delta}
  (\,\cdot\otimes\,\cdot))
  =0. \label{associativity}
\end{equation}
This property is called {\bf associativity}. 
\end{theorem}
\begin{figure}[ht!]
\centerline{\psfig{file=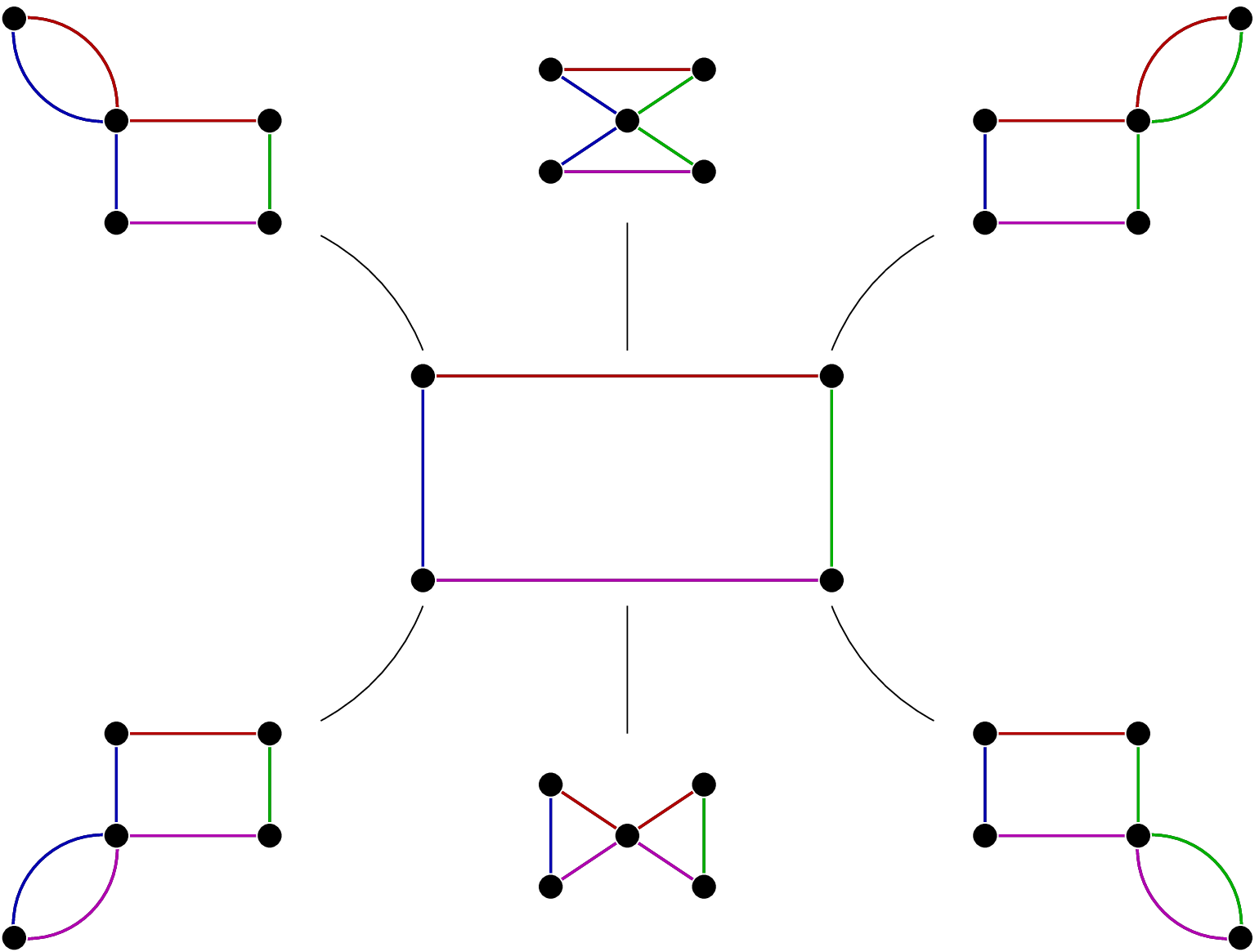,height=4cm}}
\caption{Ends of the moduli space of holomorphic rectangles.}
\label{Fig:figSix}
\end{figure} 

If we count holomorphic Whitney rectangles with boundary conditions
in $\alpha$, $\beta$, $\gamma$ and $\delta$ and with $\mu=1$ (see Definition \ref{maslovindex}) the
ends of the associated moduli space will look like pictured in Figure \ref{Fig:figSix}.
Note that we are talking about holomorphicity with respect to
an arbitrary conformal structure on the rectangle. There will be
two types of ends. We will have a degeneration into a concatenation
of triangles by variation of the conformal structure on the rectangle
and breaking into a concatenation of a rectangle with a disc by variation
of the rectangle. By Figure \ref{Fig:figSix} an appropriate count of
holomorphic rectangles will be a natural candidate for a chain homotopy
proving equation (\ref{associativity}). Define a pairing 
\[
  H\co
  \cfhat(\Sigma,\alpha,\beta,z)
  \otimes
  \cfhat(\Sigma,\beta,\gamma,z)
  \otimes
  \cfhat(\Sigma,\gamma,\delta,z)
  \lra
  \cfhat(\Sigma,\alpha,\delta,z)
\]
by counting holomorphic Whitney rectangles with boundary components
as indicated in Figure \ref{Fig:figSeven}
\begin{figure}[ht!]
\labellist\small\hair 2pt
\pinlabel {$\talpha$} [r] at 10 65
\pinlabel {$\tbeta$} [B] at 79 110
\pinlabel {$\tgamma$} [t] at 79 25
\pinlabel {$\tdelta$} [l] at 155 65
\endlabellist
\centering
\includegraphics[height=2cm]{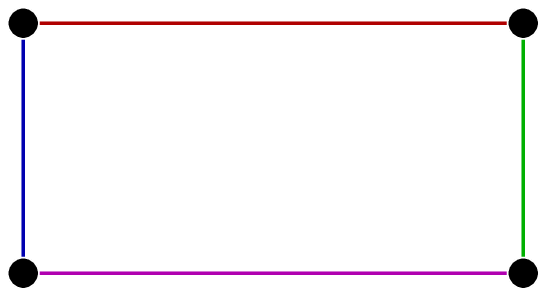}
\caption{The boundary conditions of rectangles for the definition of $H$.}
\label{Fig:figSeven}
\end{figure}
and $\mu=0$. By counting ends of the moduli space of holomorphic
rectangles with $\mu=1$ we have six contributing ends. These ends
are pictured in Figure \ref{Fig:figSix}. The four ends coming from
breaking contribute to
\begin{equation}
  \parhat
  \circ
  H(\,\cdot\otimes\,\cdot\otimes\,\cdot) 
  +
  H
  \circ
  \parhat(\,\cdot\otimes\,\cdot\otimes\,\cdot). 
  \label{asso01}
\end{equation}
In addition there are two ends coming from degenerations of the 
conformal structure on the rectangle. These give rise to
\begin{equation}
  \fhat_{\alpha\beta\gamma}
  (\fhat_{\alpha\gamma\delta}(\,\cdot\otimes\,\cdot)
  \otimes
  \,\cdot)
  -
  \fhat_{\alpha\beta\delta}
  (\,\cdot\otimes
  \fhat_{\beta\gamma\delta}
  (\,\cdot\otimes\,\cdot)).
  \label{asso02}
\end{equation}
We see that the sum of (\ref{asso01}) and (\ref{asso02}) vanishes, showing that
$H$ is a chain homotopy proving associativity.

\subsubsection{Special Case -- Handle Slides}
Handle slides provide special Heegaard triple diagrams. Let 
$(\Sigma,\alpha,\beta,z)$ be an admissible pointed Heegaard diagram
and define $(\Sigma,\alpha,\gamma,z)$ by handle sliding $\beta_1$ over
$\beta_2$. We push the $\gamma_i$ off the $\beta_i$ to make them
intersect transversely in two cancelling points. This defines a 
triple diagram, and obviously $Y_{\beta\gamma}$ equals the connected 
sum $\#^g(\stwo\times\sone)$.\vspace{0.3cm}\\
A very important observation is that the Heegaard Floer groups of
connected sums of $\stwo\times\sone$ admit a top-dimensional generator.
By Example \ref{example01} and Theorem \ref{consum},
\[
  \hfhat(\#^{g-1}(\stwo\times\sone),\orient^*)
  \cong\Z^{2g-2}
  \cong H_*(T^g;\Z),
\]
where the last identification is done using the 
$\bigwedge\,\!\!\!^*(H_1/Tor)$-module structure (see \cite{OsZa01}).
We claim that the behavior of the Heegaard Floer groups
under connected sums can be carried over to the module 
structure, and thus it remains to show the assertion for
the case $g=1$. But this is not hard to see. \vspace{0.3cm}\\
Each pair $(\beta_i,\gamma_i)$ has two intersections 
$\xpi$ and $\xmi$. Which one is denoted how is determined by
the following criterion: there is a disc-shaped domain
connecting $\xpi$ with $\xmi$ with boundary in $\beta_i$ and 
$\gamma_i$. The point
\[
  \xp=\{x^+_1,\dots,x^+_g\}
\]
is a cycle whose associated homology class is the top-dimensional
generator we denote by $\hattheta_{\beta\gamma}$. For a detailed
treatment of the top-dimensional generator we point the reader 
to \cite{OsZa01}.\vspace{0.3cm}\\
Plugging in the generator we define a map
\[
  \Fhat_{\alpha\beta\gamma}
  =\fhat^*_{\alpha\beta\gamma}(\,\cdot\otimes\hattheta_{\beta\gamma})
  \co
  \hfhat(\Sigma,\alpha,\beta,z)
  \lra
  \hfhat(\Sigma,\alpha,\gamma,z)
\]

between the associated Heegaard Floer groups. Our intention is
to show that this is an isomorphism.\vspace{0.3cm}\\
We can slide the $\gamma_1$ back over $\gamma_2$ to give 
another set of attaching circles we denote by $\delta$.
Of course we make the curves intersecting all other sets of
attaching circles transversely and introduce pairs of intersections
points of the $\delta$-curves with the $\gamma$-and $\beta$-curves.
Let $\Fhat_{\alpha\gamma\delta}$ be the associated map. Then the
associativity given in (\ref{associativity}) translates into
\[
  \fhat^*_{\alpha\beta\gamma}
  (\fhat^*_{\alpha\gamma\delta}(\,\cdot\otimes\hattheta_{\gamma\delta})
  \otimes
  \hattheta_{\beta\gamma})
  -
  \fhat^*_{\alpha\beta\delta}
  (\,\cdot\otimes
  \fhat^*_{\beta\gamma\delta}
  (\hattheta_{\beta\gamma}\otimes\hattheta_{\gamma\delta}))
  =0.
  \label{assocobmaps}
\]
The proof of the following lemma will be done in detail. It is 
the first explicit calculation using the low-dimensional model in 
a non-trivial manner.
\begin{lem}\label{thetatransform} Given the map
$\fhat_{\alpha\gamma\delta}$, we have
\[
  \fhat_{\beta\gamma\delta}
  (\hattheta_{\beta\gamma}\otimes\hattheta_{\gamma\delta})
  =
  \hattheta_{\beta\delta}.
\]
Hence, we have $\Fhat_{\beta\gamma\delta}(\hattheta_{\beta\gamma})=\hattheta_{\beta\delta}$.
\end{lem}
\begin{proof} The complement of the $\beta$-circles in $\Sigma$
is a sphere with holes. We have a precise description of
how the sets $\gamma$ and $\delta$ look like relative to $\beta$.
The Heegaard surface cut open along the $\beta$-curves can be
identified with a sphere with holes by using an appropriate
diffeomorphism. Doing so, the diagram $(\Sigma,\beta,\gamma,\delta)$
will look like given in Figure \ref{Fig:figTwo}. In
\begin{figure}[ht!]
\labellist\small\hair 2pt
\pinlabel {$\hattheta_{\gamma\delta}$} [l] at 198 383
\pinlabel {$\hattheta_{\beta\delta}$} [b] at 188 318
\pinlabel {$\hattheta_{\beta\gamma}$} [l] at 175 235
\pinlabel {$\hattheta^-_{\beta\delta}$} [r] at 500 285
\pinlabel {$\dom_1$} [tr] at 24 328
\pinlabel {$\dom_2$} [l] at 392 377
\pinlabel {$\dom_3$} [r] at 526 385
\pinlabel {$z$} [l] at 469 214
\pinlabel {$\gamma$} [b] at 193 158
\pinlabel {$\delta$} [l] at 160 69
\pinlabel {$\beta$} [br] at 122 67
\endlabellist
\includegraphics[height=7cm]{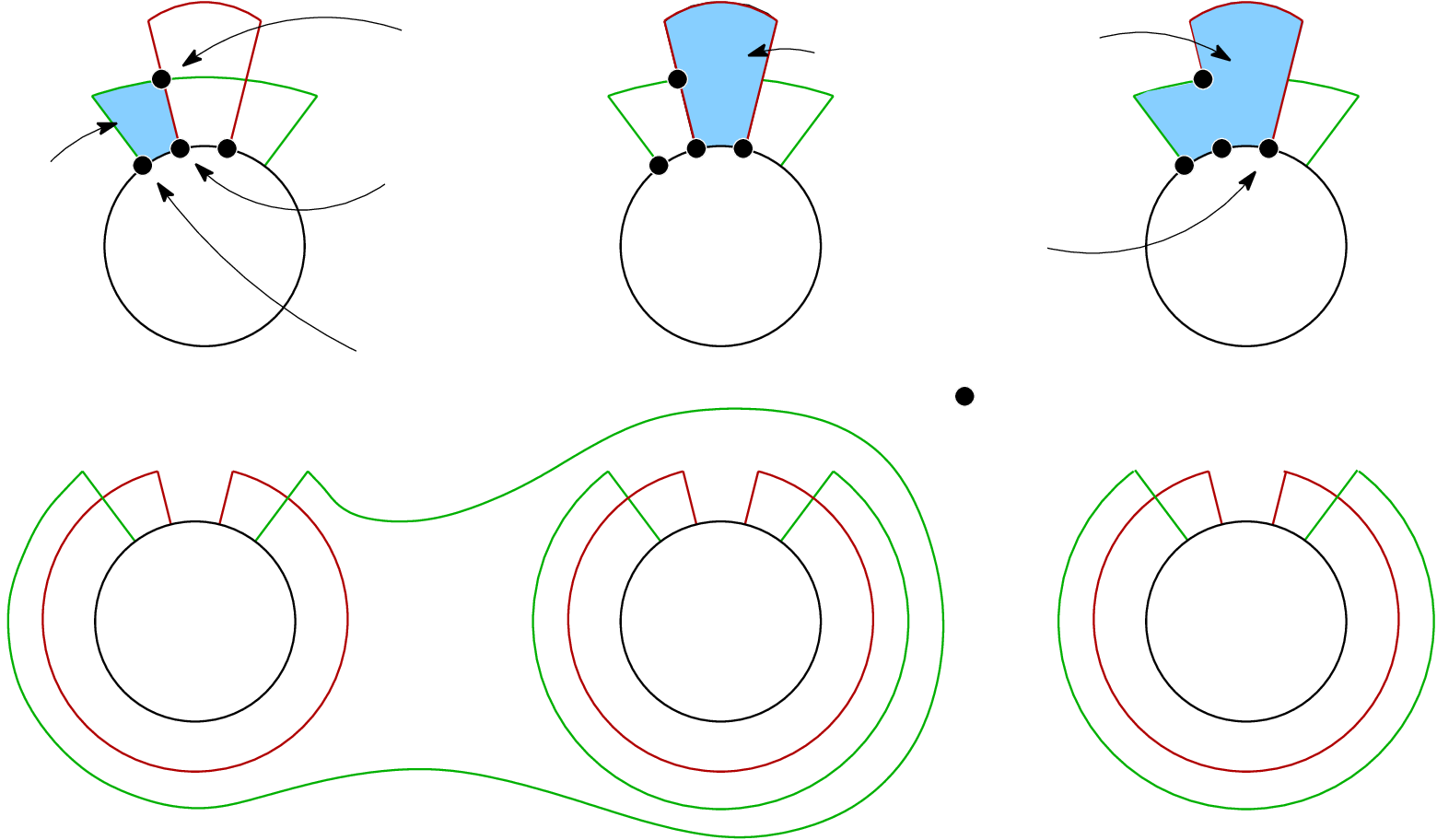}
\caption{The Heegaard surface cut open along the $\beta$-curves.}
\label{Fig:figTwo}
\end{figure}
 each component we have to have a close look at the domains
$\dom_1$, $\dom_2$ and $\dom_3$. To improve the illustration in
the picture we have separated them. There are exactly two domains
contributing to holomorphic triangles with boundary points in 
$\{\hattheta_{\beta\gamma},\hattheta_{\gamma\delta}\}$, namely
$\dom_1$ and $\dom_3$. The domain $\dom_3$ can be written as a
sum of $\dom_1$ and $\dom_2$, the former carrying $\mu=0$, the 
latter carrying $\mu=1$. Consequently, every homotopy class of triangles
using $\dom_3$-domains can be written as a concatenation of
a triangle with a disc with the expected dimensions greater than 
or equal to those mentioned. Consequently, the expected dimension
of the triangle using a $\dom_3$-domain is strictly bigger than
zero and thus does not contribute to
$\Fhat_{\beta\gamma\delta}(\hattheta_{\beta\gamma}\otimes\hattheta_{\gamma\delta})$.
All holomorphic triangles relevant to us have domains which are
a sum of $\dom_1$-domains. Taking boundary conditions into account we
see that we need a $\dom_1$-domain in each component. Thus, there is a unique homotopy
class of triangles interesting to us. By the Riemann mapping
theorem there is a unique holomorphic map 
$\phihat\co\Dhat\lra\Sigma$ from a surface with boundary $\Dhat$ 
whose associated domain equals the sum of $\dom_1$-domains. The map $\phihat$
is a biholomorphism and thus $\Dhat$ is a disjoint union of triangles. The
uniqueness of $\phihat$ tells us that the number of elements
in the associated moduli space equals the number of non-equivalent 
$g$-fold branched coverings $\Dhat\lra\D^2$. Since $\Dhat$ is a union of 
$g$ discs, this covering is unique, too (up to equivalence) and thus 
the associated moduli space is a one-point space.
\end{proof}
Lemma \ref{thetatransform} and (\ref{assocobmaps}) combine to
give the composition law
\[
  \Fhat_{\alpha\beta\delta}
  =
  \Fhat_{\alpha\gamma\delta}
  \circ
  \Fhat_{\alpha\beta\gamma}.
\]
We call a holomorphic triangle {\bf small} if it is 
supported within the thin strips of isotopy between 
$\beta$ and $\delta$.
\begin{lem}[see \cite{OsZa01}, Lemma 9.10]
\label{filtiso} Let $F\co A\lra B$ be a
map of filtered groups such that $F$ can be decomposed into
$F_0+l$, where $F_0$ is a filtration-preserving isomorphism
and $l(x)<F_0(x)$. Then, if the filtration on $B$ is bounded
from below, the map $F$ is an isomorphism of groups.
\end{lem}
There are two important observations to make. The first is
that we can equip the chain complexes with a filtration, called
the {\bf area filtration} (cf.~\cite{OsZa01}), which is indeed
bounded from below. In this situation the top-dimensional
generator $\hattheta_{\beta\delta}$ is generated by a single
intersection point $\xp\in\tbeta\cap\tdelta$. The map
$\Fhat_{\alpha\beta\delta}$ is induced by
\[
  \fhat_{\alpha\beta\delta}(\,\cdot\otimes\xp),
\]
which in turn can be decomposed into a sum of $f_0$ and 
$l$, where $f_0$ counts small holomorphic triangles and 
$l$ those triangles whose support is not contained in the 
thin strips of isotopy between $\beta$ and $\delta$. The 
map $f_0$ is filtration preserving and $l$, if the 
$\delta$-curves are close enough to the 
$\beta$-curves, strictly decreasing. By Lemma \ref{filtiso}
the map $\Fhat_{\alpha\beta\delta}$ is an isomorphism
between the associated Heegaard Floer homologies.\vspace{0.3cm}\\
To conclude topological invariance we have to see that the
following claim is true.
\begin{theorem} Two pointed admissible Heegaard diagrams
associated to a $3$-manifold are equivalent after a finite
sequence of Heegaard moves, each of them connecting two admissible
Heegaard diagrams, which can be done in the complement
of the base-point $z$.
\end{theorem}
The only situation where the point $z$ seems to be an obstacle
arises when trying to isotope an attaching circle, $\alpha_1$ say, 
over the base-point $z$. But observe that cutting the 
$\alpha$-circles out of $\Sigma$ we get a sphere with holes. We 
can isotope $\alpha_1$ freely and pass the holes by handle slides. Thus, 
the requirement not to pass $z$ is not an obstruction at all. Instead 
of passing $z$ we can go the other way around the surface by isotopies 
and handle slides.
\section{Knot Floer Homologies}\label{knotfloerhomology}
Knot Floer homology is a variant of the Heegaard Floer homology of a
manifold. Recall that the Heegaard diagrams used in Heegaard Floer
theory come from handle decompositions relative to a splitting
surface. Given a knot $K\subset Y$, we can restrict to a subclass of Heegaard diagrams by
requiring the handle decomposition to come from a handle decomposition 
of $\overline{Y\backslash\nu K}$ relative to its boundary.
Note that in the literature the knot Floer variants are 
{\bf defined for homologically trivial knots only}. However, 
the definition can
be carried over nearly one-to-one to give a well-defined topological
invariant for arbitrary knot classes. But the generalization comes
at a price. In the homologically trivial case it is possible
to subdivide the groups in a special manner giving rise to a
refined invariant, which cannot be defined in the non-trivial case.
Given a knot $K\subset Y$, we can specify a certain subclass of 
Heegaard diagrams.
\begin{definition} \label{knotdiagram} A Heegaard 
diagram $(\Sigma,\alpha,\beta)$ is said to
be {\bf subordinate} to the knot $K$ if $K$ is isotopic to a knot lying
in $\Sigma$ and $K$ intersects $\beta_1$ once, transversely and is
disjoint from the other $\beta$-circles.
\end{definition}
Since $K$ intersects $\beta_1$ once and is disjoint from the other 
$\beta$-curves we know that $K$
intersects the core disc of the $2$-handle, represented by $\beta_1$, once
and is disjoint from the others (after possibly isotoping the knot $K$).
\begin{lem} Every pair $(Y,K)$ admits a Heegaard diagram
subordinate to $K$.
\end{lem}
\begin{proof}
By surgery theory (see \cite{GoSt}, p. 104) 
we know that there is a handle decomposition of $Y\backslash\nu K$, i.e.
\[
  Y\backslash\nu K
  \cong
  (T^2\times [0,1])
  \cupb 
  h^1_{2}\cupb\dots h^1_{g}
  \cupb 
  h^2_1\cupb\dots\cupb h^2_g
  \cupb h^3
\]
We close up the boundary $T^2\times\{0\}$ with an 
additional $2$-handle $h^{2*}_1$ and a $3$-handle $h^3$ to obtain
\begin{equation}
  Y\cong
  h^3\cupb h^{2*}_1
  \cupb
  (T^2\times I)
  \cupb 
  h^1_2\cupb\dots h^1_g
  \cupb 
  h^2_1\cupb\dots\cupb h^2_g
  \cupb h^3.\label{handledecomp02}
\end{equation}
We may interpret $h^3\cupb h^{2*}_1\cupb(T^2\times[0,1])$ as a $0$-handle $h^0$
and a $1$-handle $h^{1*}_1$. Hence, we obtain the following decomposition of
$Y$:
\[
  h^0
  \cupb
  h^{1*}_1
  \cupb
  h^{1}_2
  \cupb
  \dots
  \cupb
  h^1_g
  \cupb
  h^2_1
  \cupb
  \dots
  \cupb
  h^2_g
  \cupb
  h^3.
\]
We get a Heegaard diagram $(\Sigma,\alpha,\beta)$ where
$\alpha=\alpha_1^*\cup\{\alpha_2,\dots,\alpha_g\}$ are the co-cores 
of the $1$-handles and $\beta=\{\beta_1,\dots,\beta_g\}$ are 
the attaching circles of the $2$-handles.
\end{proof}
Having fixed such a Heegaard diagram $(\Sigma,\alpha,\beta)$ we can encode 
the knot $K$ in a pair of points. After isotoping $K$ onto $\Sigma$, 
we fix a small interval $I$ in $K$ containing the intersection point 
$K\cap\beta_1$. This interval should be chosen small enough such 
that $I$ does not contain any other intersections of $K$ with other 
attaching curves. The boundary $\partial I$ of $I$ determines two 
points in $\Sigma$ that lie in the complement of the attaching circles, 
i.e.~$\partial I=z-w$, where the orientation of $I$ is given by the 
knot orientation. This leads to a doubly-pointed Heegaard diagram 
$(\Sigma,\alpha,\beta,w,z)$. Conversely, a doubly-pointed Heegaard 
diagram uniquely determines a topological knot class: Connect 
$z$ with $w$ in the complement of the attaching circles $\alpha$ 
and $\beta\backslash\beta_1$ with an arc $\delta$ that crosses 
$\beta_1$ once. Connect $w$ with $z$ in the complement of $\beta$
using an arc $\gamma$. The union $\delta\cup\gamma$ is represents the 
knot klass $K$ represents. The orientation on $K$ is given by orienting $\delta$ such 
that $\partial\delta=z-w$. If we use a different path 
$\widetilde{\gamma}$ in the complement of $\beta$, we observe that 
$\widetilde{\gamma}$ is isotopic to $\gamma$ (in $Y$): Since  
$\Sigma\backslash\beta$ is a sphere with holes an isotopy can 
move $\gamma$ across the holes by doing handle slides. Isotope 
the knot along the core discs of the $2$-handles to cross the 
holes of the sphere. Indeed, the knot class does not depend
on the specific choice of $\delta$-curve.\vspace{0.3cm}\\
The knot chain complex $\cfkhat(Y,K)$ is the free $\ztwo$-module 
(or $\Z$-module) generated by the intersections $\talpha\cap\tbeta$. 
The boundary operator $\parhat^w$, for $x\in\talpha\cap\tbeta$, is 
defined by
\[
  \parhat^w(x)
  =
  \sum_{y\in\talpha\cap\tbeta}
  \sum_{\phi\in H(x,y,1)}
  \#\modhatphi\cdot y,
\]
where $H(x,y,1)\subset\pitwo(x,y)$ are the homotopy classes
with $\mu=1$ and $n_z=n_w=0$. We denote by $\hfkhat(Y,K)$
the associated homology theory $H_*(\cfkhat(Y,K),\parhat^w)$.
The crucial observation for showing invariance is, that two 
Heegaard diagrams subordinate to a given knot can be connected 
by moves that {\it respect the knot complement}.
\begin{lem}(\cite{OsZa04})\label{helplem} Let 
$(\Sigma,\alpha,\beta,z,w)$ and $(\Sigma',\alpha',\beta',z',w')$ 
be two Heegaard diagrams subordinate to a given knot $K\subset Y$. 
Let $I$ denote the interval inside $K$ connecting $z$ with $w$,
interpreted as sitting in $\Sigma$. Then these two diagrams 
are isomorphic after a sequence of the following moves:
\begin{enumerate}
  \item[($m_1$)] Handle slides and isotopies among the 
  $\alpha$-curves. These isotopies may not cross~$I$.
  \item[($m_2$)] Handle slides and isotopies among 
  the $\beta_2,\dots,\beta_g$. These isotopies may 
  not cross $I$.
  \item[($m_3$)] Handle slides of $\beta_1$ over 
  the $\beta_2,\dots,\beta_g$ and isotopies.
  \item[($m_4$)] Stabilizations/destabilizations.
\end{enumerate}
\end{lem}
For the convenience of the reader we include a short proof of this lemma.
\begin{proof} 
By Theorem 4.2.12 of \cite{GoSt} we can transform two 
relative handle decompositions into each other by 
isotopies, handle slides and handle creation/annihilation of 
the handles written at the right of $T^2\times[0,1]$ in 
$(\ref{handledecomp02})$. Observe 
that the $1$-handles may be isotoped along the boundary 
$T^2\times\{1\}$. Thus, we can transform two Heegaard diagrams 
into each other by handle slides, isotopies, creation/annihilation 
of the $2$-handles $h^2_i$ and we may slide the $h^1_i$ over 
$h^1_j$ and over $h^{1*}_1$ (the latter corresponds to $h^1_i$ 
sliding over the boundary $T^2\times\{1\}\subset T^2\times I$ 
by an isotopy). But we are not allowed to move $h^{1*}_1$ off 
the $0$-handle. In this case we would lose the relative 
handle decomposition. In terms of Heegaard diagrams 
we see that these moves exactly translate into the moves given 
in ($m_1$) to ($m_4$). Just note that sliding the $h^1_i$ over $h^{1*}_1$,
in the dual picture, looks like sliding $h^{2*}_1$ over the $h^2_i$. 
This corresponds to move ($m_3$).
\end{proof}
\begin{prop}\label{knotfloer} Let $K\subset Y$ be an arbitrary knot. 
The knot Floer homology group $\hfkhat(Y,K)$ is a topological invariant
of the knot type of $K$ in $Y$. These homology groups split with 
respect to $\spinc(Y)$.
\end{prop}
\begin{proof} Given one of the moves $(m_1)$ to $(m_4)$, the associated
Heegaard Floer homologies are isomorphic, which is shown using one
of the isomorphisms given in \S\ref{topoinvariance}. Each of these
maps is defined by counting holomorphic discs with punctures, whose
properties are shown by defining maps by counting holomorphic discs
with punctures.\vspace{0.3cm}\\
{\bf Isotopies/Almost Complex Structure.} Denote by $J$ the path
of almost complex structures used in the definition of the Heegaard
Floer homologies. Let $M$ be an isotopy or perturbation of $J$. Let $\Phihat$ be the 
isomorphism induced by $M$. We split the isomorphism up into 
\[
  \Phihat=\Phihat^w+\Phihat^{\not=},
\]
where $\Phihat^w$ is defined by counting holomorphic discs with
punctures (for a precise definition look into \S\ref{acsinvariance} 
and \S\ref{isotopyinvariance}) that fulfill $n_w=0$. Let us denote 
with $\modspace_0$ the associated moduli space used to define the map $\Phihat$. 
The index indicates the value of the index $\mu$. The chain map property 
of $\Phihat$ was shown by counting ends of $\modspace_1$ which contains 
the same objects we needed to define $\Phihat$ but now with the index fulfilling
$\mu=1$ (see Definition \ref{maslovindex}). We 
restrict our attention to $\modspace^w_0$ and $\modspace^w_1$, the superscript $w$
indicates that we look at the holomorphic elements in $\modspace_0$ (or $\modspace_1$ respectively)
with intersection number $n_w=0$: The additivity of the intersection number $n_w$ 
and the positivity of intersections guarantees that the ends 
of $\modspace^w_1$ lie within the space $\modspace^w_0$ provided that $M$ respects the point $w$. 
If $M$ is an isotopy, respecting $w$ means, that no attaching circle crosses the point $w$. If
$M$ is a perturbation of $J$, respecting $w$ means, that we perturb $J$ through nearly 
symmetric almost complex structures such that $V$ (cf.~Definition \ref{defnearlysym}) also 
contains $\{w\}\times\symgmo$. Hence, we have the equality
\[
  (\partial\modspace_1)^w=\partial\modspace^w_1.
\] 
Thus, $\Phihat^w$ has to be a chain map between the respective knot 
Floer homologies. To show that $\Phihat$ is an isomorphism, we invert the 
move $M$ we have done and construct the associated morphism $\Psihat$. To show that
$\Psihat$ is the inverse, we construct a chain homotopy equivalence between
$\Psihat\circ\Phihat$ and the identity (or between $\Phihat\circ\Psihat$ and
the identity) by counting elements of $\modspace^{ch}_0$ which are defined by constructing a 
family of moduli spaces $\modspace^\tau_{-1}$, $\tau\in[0,1]$, and combining them to
\[
  \modspace^{ch}_0:=\bigsqcup_{\tau\in[0,1]}\modspace^\tau_{-1}.
\]
The spaces $\modspace^\tau_{-1}$ are defined like done in \S\ref{acsinvariance} and
\S\ref{isotopyinvariance}. We show the chain homotopy equation by counting ends 
of $\modspace^{ch}_1$. Restricting our attention to $\modspace^{ch,w}$, this space 
consists of the union of spaces $\modspace^{\tau,w}_{-1}$, $\tau\in[0,1]$ 
(cf.~\S\ref{acsinvariance} and \S\ref{isotopyinvariance}). We obtain the equality
\[
 (\partial\modspace^{ch}_0)^w=\partial\modspace^{ch,w}_0.
\]
And hence we see that $\Phihat^w$ is an isomorphism.\vspace{0.3cm}\\
{\bf Handle slides.} In case of the knot Floer homology we are able
to define a pairing 
\[
  \fhat_{\alpha\beta\gamma}
  \co
  \cfkhat(\Sigma,\alpha,\beta,w,z)
  \otimes
  \cfkhat(\Sigma,\beta,\gamma,w,z)
  \lra
  \cfkhat(\Sigma,\alpha,\gamma,w,z)
\]
induced by a doubly-pointed Heegaard triple diagram $(\Sigma,\alpha,\beta,\gamma,w,z)$.
We have to see, that in case the triple is induced by a
handle slide, the knot Floer homology $\hfkhat(\Sigma,\beta,\gamma,w,z)$
 carries a top-dimensional generator $\hattheta_{\beta\gamma}$, 
analogous to the discussion for the Heegaard Floer homologies,
with similar properties (recall the composition law). It is
easy to observe that, in case of a handle slide, the points 
$w$ and $z$ lie in the same component of 
$\Sigma\backslash\{\beta\cup\gamma\}$. Hence, we have an identification
\[
  \hfkhat(\Sigma,\beta,\gamma,w,z)
  =
  \hfhat(\#^{g}(\stwo\times\sone)).
\]
Counting triangles with $n_w=0$, the positivity of intersections and 
the additivity of the intersection number $n_w$ guarantees that the 
discussion carries over verbatim and gives invariance here. 
\end{proof}
\begin{rem} If a handle were slid over $\beta_1$, we would leave the
class of subordinate Heegaard diagrams. Recall that subordinate
Heegaard diagrams come from relative handle decompositions. 
\end{rem}
\subsubsection{Admissibility}\label{admsec}
The admissibility condition given in Definition \ref{admissintro} suffices to
give a well-defined theory. However, since we have an additional point $w$ in play, 
we can relax the admissibility condition.
\begin{definition}\label{extweakadm} We call a 
doubly-pointed Heegaard diagram
$(\Sigma,\alpha,\beta,w,z)$ {\bf extremely weakly admissible}
for the $\spinc$-structure $s$ if for every non-trivial periodic 
domain, with $n_w=0$ and $\left<c_1(s),\homology(\dom)\right>=0$, 
the domain has both positive and negative coefficients.
\end{definition}
With a straightforward adaptation of the proof of well-definedness in
the case of $\parhat_z$ we get the following result 
(see \cite{OsZa01}, Lemma 4.17, cf.~Definition \ref{admissintro} 
and cf.~proof of Theorem \ref{wdefined}).
\begin{theorem} Let $(\Sigma,\alpha,\beta,w,z)$ be an extremely
weakly admissible Heegaard diagram. Then $\parhat^w$ is
well-defined and a differential.\hfill$\square$
\end{theorem}
Note that Ozsv\'{a}th and Szab\'{o} impose weak admissibility of the
Heegaard diagram $(\Sigma,\alpha,\beta,z)$. The introduction of
our relaxed condition is done since we there are setups (see \cite{Saha01}) where 
it is convenient to relax the admissibility condition
like introduced.

\subsubsection{Other knot Floer homologies}
By permitting variations of $n_z$ in the differential
we define the homology $\hfkminus$: Let $\cfkminus(Y,K)$ be 
the $\Z[U^{-1}]$-module (or $\ztwo[U^-{1}]$-module) generated by 
the intersection points $\talpha\cap\tbeta$. A differential 
$\parminus_w$ is defined by
\[
  \parminus_w(x)
  =
  \sum_{y\in\talpha\cap\tbeta}
  \sum_{\phi\in H(x,y,1)}
  \#\modhatphi\cdot y,
\]
where $H(x,y,1)\subset\pitwo(x,y)$ are the homotopy classes with
$n_w=0$ (possibly $n_z\not=0$) and $\mu=1$. To make this a well-defined
map we may impose the strong admissibility condition on the underlying
Heegaard diagram or relax it like it was done for weak admissibility in
Definition \ref{extweakadm}. Using this construction, and continuing like
in \S\ref{genhomology}, we define variants we denote by $\hfkinfty$ and
$\hfkplus$. The groups are naturally connected by exact sequences 
analogous to those presented in Lemma \ref{genseq}.

\subsection{Refinements}
If the knot $K$ is null-homologous, we get, using a Mayer-Vietoris computation, that
\begin{equation}
  \spinc(Y_0(K))=\spinc(Y)\times\Z.\label{spinceq}
\end{equation}
Alternatively, by interpretation of $\spinc$-structures as homology classes
of vector fields, i.e.~homotopy classes over the $2$-skeleton of $Y$,
we can prove this result and see that there is a very geometric realization
of the correspondence (\ref{spinceq}). Given a $\spinc$-structure $t$
on $Y_0(K)$, we associate to it the pair $(s,k)$, where $s$ is the restriction
of $t$ on $Y$ and $k$ an integer we will define in a moment. Beforehand,
we would like to say in what way the phrase {\it restriction of $t$ onto $Y$} 
makes sense. Pick a vector field $v$ in the homology class of $t$ and restrict
this vector field to $Y\backslash\nu K$. Observe that we may regard $Y\backslash\nu K$
as a submanifold of $Y_0(K)$. The restricted vector field may be 
interpreted as sitting on $Y$. We extend $v$ to the tubular
neighborhood $\nu K$ of $K$ in $Y$, which determines a $\spinc$-structure $s$ on $Y$. However, 
the induced $\spinc$-structure does not depend on the special
choice of extension of $v$ on $\nu K$, since $K$ is homologically trivial.\vspace{0.3cm}\\
To a $\spinc$-structure $t$ we can associate a link $L_t$ and its homology class
determines the $\spinc$-structure. Denote by $\mu_0$ a meridian of $K$ in $Y$,
interpreted as sitting in $Y_0(K)$. Then $L_t$ can be written as a sum
\[
  L_t=k\cdot\mu_0+\dots,
\]
and thus we can compute $k$ with
\[
  k
  =
  lk^Y(L,\lambda)
  =
  \#^Y(L,F)
  =
  \#^{Y_0(K)}(L,\Fhat)
  =
  \bigl<\frac{1}{2}c_1(t),[\Fhat]\bigr>,
\]
where $\lambda$ is a push-off of $K$ in $Y$ and $\Fhat$ is obtained 
by taking a Seifert surface $F$ of $K$ in $Y$ and capping it off with 
a disc in $Y_0(K)$. \vspace{0.3cm}\\
We can try to separate intersection points $\talpha\cap\tbeta$
with respect to $\spinc$-structures of $Y_0(K)$. This
defines a refined invariant $\cfkhat(Y,K,t)$, for $t\in\spinc(Y_0(K))$, and we
have 
\[
  \cfkhat(Y,K,s)=\bigoplus_{t\in H_s}\cfkhat(Y,K,t),
\]
where $H_s\subset\spinc(Y_0(K))$ are the elements extending 
$s\in\spinc(Y)$. We have to show that $\parhat^w$ preserves
this splitting. We point the interested reader to \cite{OsZa04}.

\section{Maps Induced By Cobordisms}\label{parcobmaps}
The pairing introduced in \S\ref{popprod} can be used to associate
maps to cobordisms. In general, every cobordism between two connected $3$-manifolds
$Y$ and $Y'$ can be decomposed into $1$-handles, $2$-handles and
$3$-handles (cf.~Proposition 4.2.13 in \cite{GoSt}). All cobordisms 
appearing through our work will be
induced by surgeries on a $3$-manifold. A surgery corresponds to
a $2$-handle attachment to the trivial cobordism $Y\times I$.
For this reason we will not discuss $1$-handles and $3$-handles.
We will give the construction for cobordisms obtained
by attachments of one single $2$-handle. For a definition of the general, 
very similar construction, we point the interested reader 
to \cite{OsZa03}.\vspace{0.3cm}\\
Given a framed knot $K\subset Y$, we fix an admissible 
Heegaard diagram subordinate to $K$. Without loss of generality, we 
can choose the diagram such that $\beta_1=\mu$ is a meridian
of the first torus component of $\Sigma$. The framing of $K$
is given, by pushing $K$ off itself onto the Heegaard surface.
The resulting knot on $\Sigma$ is determined by 
$\lambda+n\cdot \mu$, for a suitable $n\in\Z$. With this done, we 
can represent the surgery by the Heegaard triple
diagram $(\Sigma,\alpha,\beta,\gamma)$ where $\gamma_i$, $i\geq2$, are
isotopic push-offs of the $\beta_i$, perturbed, such that
$\gamma_i$ intersects $\beta_i$ in a pair of cancelling intersection
points. The curve $\gamma_1$ equals $\lambda+n\cdot \mu$.
\begin{prop} The cobordism 
$X_{\alpha\beta\gamma}\cup_{\partial}(\#^{g-1}D^3\times\sone)$
is diffeomorphic to the cobordism $W_K$ given by the framed 
surgery along $K$.
\end{prop}
We define 
\[
  \Fhat_{W_K}=\fhat_{\alpha\beta\gamma}^*
\]
as the map induced by the cobordism $W_K$. Of course, for this to make
sense, we have to show that $\Fhat_{W_K}$ does not depend
on the choices made in its definition. This is shown by the
following recipe: Suppose we are given maps $\Fhat_1$ and $\Fhat_2$, induced by
two sets of data that can be connected via a Heegaard move.
 Then these maps fit into a commutative box
\[
\begin{diagram}[size=1.5em,labelstyle=\scriptstyle]
\hfhat       & & \rTo^{\Fhat_1} & & \hfhat \\
\dTo^\cong   & &                & & \dTo_\cong \\
\hfhat       & & \rTo^{\Fhat_2} & & \hfhat
\end{diagram}
\]
where the associated Heegaard Floer homologies are connected by
the isomorphism induced by the move done to connect the diagrams.
If we did a handle slide, we use associativity together with a
conservation property analogous to Lemma \ref{thetatransform} to show a 
composition law reading
\[
  \Fhat_{\alpha\gamma\gamma'}\circ\Fhat_{\alpha\beta\gamma}
  =\Fhat_{\alpha\beta\gamma'}.
\]
In a similar vein one covers handle slides among the $\alpha$-circles. Invariance
under Isotopies and changes of almost complex structures is
shown by proving, that the isomorphisms induced by these moves make the
corresponding diagram commute.\vspace{0.3cm}\\
Given a framed link $L=K_1\sqcup\dots\sqcup K_m$, observe that
we can obviously define a map 
\[
  \Fhat_{L}
  \co
  \hfhat(Y)
  \lra
  \hfhat(Y_L),
\] 
where $Y_L$ is the manifold obtained by surgery along $L$ in $Y$,
in the same way we did for a single attachment. We claim that associativity, together 
with a conservation law like given in Lemma \ref{thetatransform}, 
will suffice to show that the map $\Fhat_L$ associated to multiple 
attachments is a composition 
\[
  \Fhat_{L}=\Fhat_{K_m}\circ\dots\circ\Fhat_{K_1}
\]
of the maps $\Fhat_{K_i}$ associated to the single attachments along the $K_i$. The 
associativity will prove that the maps in this chain {\it commute}. Although we
have to be careful by saying {\it they commute}. The maps, as we
change the order of the attachments, are defined differently and, thus,
differ depending on the attachment order.\vspace{0.3cm}\\
There is a procedure for defining maps associated to $1$-handle
attachments and $3$-handle attachments. Their construction is 
not very enlightening, and the cobordisms appearing in our
discussions will mostly be induced by surgeries.

\section{The Surgery Exact Triangle}\label{parsurextri}
Denote by $K$ a knot in $Y$ and let $n$ be a framing of that knot.
We will briefly recall the notion of framings to fix the notation.
Given a tubular neighborhood $\nu K\hookrightarrow Y$ of $K$, we
fix a meridian $\mu$ of the boundary $\partial\nu K$. A framing is
given by a push-off $n$ of $K$, sitting on $\partial\nu K$,
such that $\#(\mu,n)=1$. The pair $\mu,\lambda$ determines a basis
for $H_1(\partial\nu K;\Z)$. Any other framing $\lambda'$ can be
written as $\lambda'=m\cdot\mu+\lambda$, for an integer $m\in\Z$, and 
vice versa any of these linear combinations determines a framing 
on $K$. Thus, when writing $n$ as a framing for $K$ it makes sense 
to talk about the framing $n+\mu$. If the knot is homologically
trivial, it bounds a Seifert surface which naturally induces a
framing on the knot called {\bf the Seifert framing}. This serves
as a canonical framing, and having fixed this framing we can talk
about framings as an integer $n\in\Z$. This identification will be
done whenever it makes sense.\vspace{0.3cm}\\
There is a long exact sequence
\begin{equation}
  \dots
  \overset{\partial_*}{\lra}
  \hfhat(Y)
  \overset{\Fhat_1}{\lra}
  \hfhat(Y_K^n)
  \overset{\Fhat_2}{\lra}
  \hfhat(Y_K^{n+\mu})
  \overset{\partial_*}{\lra}
  \dots, \label{surextri}
\end{equation}
where $\Fhat_i$ denote the maps associated to the cobordisms
induced by the surgeries. The map $\Fhat_2$ is induced by
a surgery along a meridian of $K$ with framing $-1$.
The exactness of the sequence is proved by showing that $\Fhat_1$ -- on the chain level -- can 
be perturbed within its chain homotopy class to fit into a short exact sequence of chain 
complexes and chain maps (see \cite{OsZa02})
\begin{equation}
  0
  \lra
  \cfhat(Y)
  \overset{\widetilde{\Fhat_1}}{\lra}
  \cfhat(Y_K^n)
  \overset{\Fhat_2}{\lra}
  \cfhat(Y_K^{n+\mu})
  \lra
  0.\label{chainexact}
\end{equation}
The map $\partial_*$ in (\ref{surextri}) denotes the 
induced coboundary. This enables us to prove the existence 
of the surgery exact triangle.
\begin{theorem}\label{exacttriangle} In the situation described
above, let $\nu$ denote a meridian of $\mu$ and $\Fhat_3$ the map
induced by surgery along $\nu$ with framing $-1$. There is a long 
exact sequence
\begin{diagram}[size=1.5em,labelstyle=\scriptstyle]
\hfhat(Y) &       & \rTo^{\Fhat_1} &      & \hfhat(Y_K^n) \\
          & \luTo^{\Fhat_3} &                    & \ldTo_{\Fhat_2}& \\
          &       & \hfhat(Y_K^{n+\mu})  &      &
\end{diagram}
which is called {\bf surgery exact triangle}.
\end{theorem}
\begin{figure}[ht!]
\labellist\small\hair 2pt
\pinlabel {$n$} [Bl] at 3 182
\pinlabel {$n\!+\!\mu$} [Bl] at 93 182
\pinlabel {$n$} [Bl] at 175 182
\pinlabel {$n$} [Bl] at 323 182
\pinlabel {$n$} [Bl] at 463 182
\pinlabel {$n$} [Bl] at 624 182
\pinlabel {$n$} [Bl] at 723 182
\pinlabel {$K$} [l] at 3 8
\pinlabel {$K$} [l] at 93 8
\pinlabel {$K$} [l] at 175 8
\pinlabel {$K$} [l] at 323 8
\pinlabel {$K$} [l] at 463 8
\pinlabel {$K$} [l] at 624 8
\pinlabel {$K$} [l] at 723 8
\pinlabel {$-1$} [t] at 143 80
\pinlabel {$0$} [t] at 294 80
\pinlabel {$-1$} [t] at 431 80
\pinlabel {$-1$} [t] at 595 80
\pinlabel {$\mu$} [B] at 143 101
\pinlabel {$\mu$} [B] at 294 101
\pinlabel {$\mu$} [B] at 431 101
\pinlabel {$\mu$} [B] at 595 101
\pinlabel {$\nu$} [l] at 516 110
\pinlabel {$-1$} [Bl] at 495 127
\pinlabel {$\nu$} [l] at 679 110
\pinlabel {$0$} [Bl] at 663 127
\endlabellist
\centering
\includegraphics[height=3cm]{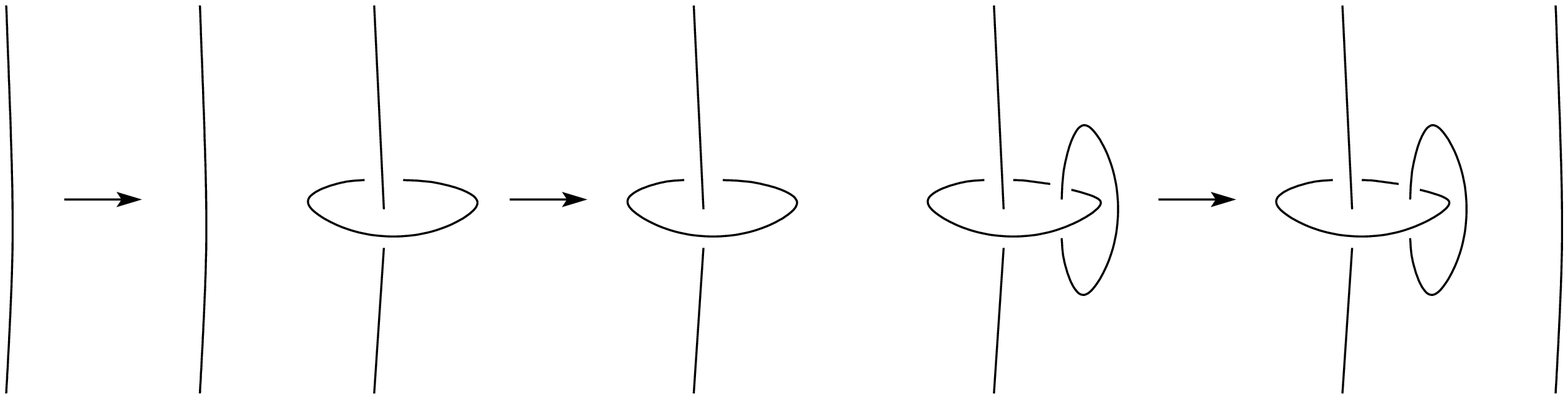}
\caption{The topological situation in the exact triangle.}
\label{Fig:figFour}
\end{figure}

\begin{proof} Observe that the topological situation is very symmetric.
The long exact sequence (\ref{surextri}) corresponds to the topological
 situation pictured in Figure \ref{Fig:figFour}. Each arrow in Figure \ref{Fig:figFour} corresponds to an exact sequence
of type (\ref{surextri}). With the identifications given, we can 
concatenate the three sequences to give the surgery exact sequence
of Theorem \ref{exacttriangle}.
\end{proof}
A second proof, one more appealing to our aesthetic sense, although 
only valid for $\ztwo$-coefficients, 
was also developed by 
Ozsv\'{a}th and Szab\'{o}. We will discuss the proof
in the remainder of this paragraph. It contains a very interesting
algebraic approach for showing exactness of a sequence.\vspace{0.3cm}\\
The composition $\fhat_2\circ\fhat_1$ in the sequence
\begin{equation}
  \cfhat(Y)
  \overset{\fhat_1}{\lra}
  \cfhat(Y_K^n)
  \overset{\fhat_2}{\lra}
  \cfhat(Y_K^{n+\mu})
\label{chainex}
\end{equation}
is null-chain homotopic. Let $(\Sigma,\alpha,\beta,z)$ be a Heegaard diagram
subordinate to the knot $K\subset Y$. We can choose the data such
that $\beta_1$ is a meridian of the first torus component of $\Sigma$.
A Heegaard diagram of $Y_K^n$ can be described by $(\Sigma,\alpha,\gamma,z)$
where $\gamma_i$, $i\geq2$ are isotopic push-offs of the $\beta_i$ such that
$\beta_i$ and $\gamma_i$ meet in two cancelling intersections transversely.
The curve $\gamma_1$ equals $n\cdot\beta_1+\lambda$ where $\lambda$ is the
longitude of the first torus component of $\Sigma$ determining the
framing on $K$. We define a fourth set of attaching circles $\delta$ where
$\delta_i$, $i\geq2$ are push-offs of the $\gamma_i$ which meet the
 $\gamma_i$ and $\delta_i$ in two cancelling intersections. The 
 curve $\delta_1$ equals
$(n+1)\beta_1+\lambda$. Thus, $(\Sigma,\alpha,\delta)$ is a Heegaard diagram
of $Y_K^{n+\mu}$. By associativity (\ref{associativity}), the composition
$\fhat_2\circ\fhat_1$ is chain homotopic to
\[
  \fhat_{\alpha\beta\delta}
  (\,\cdot\otimes
  \fhat_{\beta\gamma\delta}
  (\hattheta_{\beta\gamma}
  \otimes\hattheta_{\gamma\delta})
  ),
\]
where the chain homotopy $H$ is given by counting holomorphic rectangles
with suitable boundary conditions (cf.~\S\ref{parhsinvar}). To compute
$\fhat_{\beta\gamma\delta}
  (\hattheta_{\beta\gamma}
  \otimes\hattheta_{\gamma\delta})$
we use a model calculation. Figure \ref{Fig:figNineteen} illustrates the
Heegaard triple diagram.
\begin{figure}[ht!]
\labellist\small\hair 2pt
\pinlabel {$\hattheta_{\beta\gamma}$} [B] at 268 202
\pinlabel {$\hattheta_{\gamma\delta}$} [B] at 348 228
\pinlabel {$z$} [l] at 233 162
\pinlabel {$\delta_1$} [t] at 4 52
\pinlabel {$\beta_1$} [l] at 95 16
\pinlabel {$\hattheta_{\gamma\delta}$} [l] at 145 10
\pinlabel {$\gamma_1$} [t] at 183 45
\pinlabel {$\beta_2$} [t] at 254 37
\pinlabel {$\gamma_2$} [t] at 275 32
\pinlabel {$\delta_2$} [t] at 294 20
\endlabellist
\centering
\includegraphics[width=12cm]{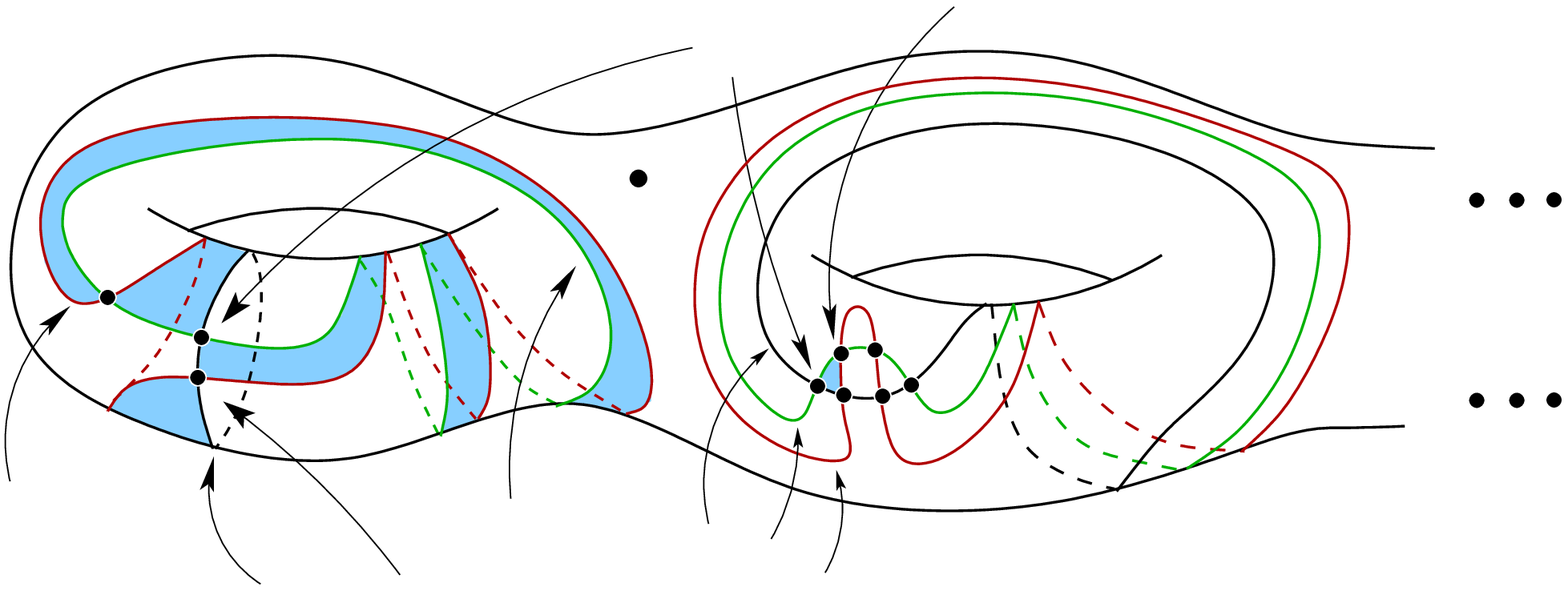}
\caption{Heegaard triple diagram for computation of 
$\fhat_{\beta\gamma\delta}(\hattheta_{\beta\gamma}\otimes
\hattheta_{\gamma\delta})$. }
\label{Fig:figNineteen}
\end{figure}

There are exactly two homotopy classes of Whitney triangles we have to
count. Each domain associated to the homotopy classes is given by a
disjoint union of triangles. Thus, the moduli spaces associated to these
homotopy classes each carry one single element 
(cf.~Lemma \ref{thetatransform}). Hence, in $\ztwo$-coefficients
\[
 \fhat_{\beta\gamma\delta}
  (\hattheta_{\beta\gamma}
  \otimes\hattheta_{\gamma\delta})=2\cdot\hattheta_{\beta\delta}=0.
\]
In general we have to see that we can choose the signs of the associated
elements differently. But observe that the domains of both homotopy 
classes contributing in our signed count differ by a triply-periodic 
domain. We can choose the signs on these 
elements differently.\vspace{0.3cm}\\
This discussion carries over verbatim for any of the maps in the
surgery exact sequence. The symmetry of the situation, as indicated
in Figure \ref{Fig:figFour}, makes it possible to carry over the proof
given here.\vspace{0.3cm}\\
There is an algebraic trick to show exactness on the homological
level. Let
\[
  H\co\cfhat(Y)\lra\cfhat(Y_K^{n+\mu})
\]
denote the null-homotopy of $\fhat_2\circ\fhat_1$ 
(cf.~\S\ref{parhsinvar}). Define 
the chain complex $A_{\fhat_1,\fhat_2}$ to be given by the module 
$A=\cfhat(Y)\oplus\cfhat(Y_K^n)\oplus\cfhat(Y_K^{n+\mu})$ with
the differential
\[
  \partial
  =
  \left(
  \begin{matrix}
  \parhat_Y & 0 	      & 0 \\
  \fhat_1   & \parhat_{Y_K^n} & 0 \\
  H         & \fhat_2         & \parhat_{Y_K^{n+\mu}}
  \end{matrix}
  \right).
\]
\begin{lem}\label{exlem} The 
sequence (\ref{chainex}) is exact on the homological
level at $\cfhat(Y_K^n)$ if $H_*(A_{\fhat_1,\fhat_2})=0$.
\end{lem}
\begin{proof}
Suppose we are given an element $b\in\cfhat(Y_K^n)\cap\ker(\fhat_2)$ 
with $\parhat_{Y_K^n} b=0$. Since 
$H_*(A_{\fhat_1,\fhat_2},\partial)$ is trivial 
there is an element $(x,y,w)\in A$ such that 
$(0,b,0)=\partial(x,y,w)$. Thus, we have
\[
  b=\fhat_1(x) + \parhat_{Y_K^n}(y)
\]
proving, that $[b]\in\mbox{\rm im}(\Fhat_1)$. 
\end{proof}
\begin{definition} For a chain map $f\co A\lra B$ between 
$\ztwo$-vector spaces we define its
{\bf mapping cone} to be the chain complex $M(f)$, given by
the module $A\oplus B$ with differential
\[
  \partial_f=
  \left(
  \begin{matrix}
  \partial_A & 0 \\
  f   & \partial_B
  \end{matrix}.
  \right)
\]
The mapping cone is a chain complex.
\end{definition}
From the definition of mapping cones there is a short exact
sequence of chain complexes
\[
  0
  \lra
  \cfhat(Y_K^{n+\mu})
  \overset{\fhat_1}{\lra}
  A_{\fhat_1,\fhat_2}
  \overset{\fhat_2}{\lra}
  M(\fhat_1)
  \lra
  0
\]
inducing a long exact sequence between the associated homologies.
The connecting morphism of this long exact sequence is induced by
\[
  (H,\fhat_2)
  \co
  M(\fhat_1)
  \lra
  \cfhat(Y_K^{n+\mu}).
\]
The triviality of $H_*(A_{\fhat_1,\fhat_2},\partial)$ is the 
same as saying that $(H,\fhat_2)_*$ is an isomorphism. 
\begin{lem}[\cite{OsZa07}, Lemma 4.2]\label{Seidellem} Let 
$\{A_i\}_{i\in\Z}$ be a collection of modules and let
\[
  \{f_i\co A_i\lra A_{i+1}\}_{i\in\Z}
\]
be a collection of chain maps such that $f_{i+1}\circ f_i$, $i\in\Z$ 
is chain homotopically trivial by a chain homotopy 
$H_i\co A_i\lra A_{i+2}$. The maps
\[
  \psi_i=f_{i+2}\circ H_i+H_{i+1}\circ f_i
  \co
  A_i
  \lra
  A_{i+3}
\]
should induce isomorphisms between the associated homologies. Then the maps
$(H_i,f_{i+1})\co M(f_i)\lra A_{i+2}$ induce isomorphisms on the 
homological level.
\end{lem}
If we can show that the sequence
\[
  \dots
  \overset{\fhat_3}{\lra}
  \cfhat(Y)
  \overset{\fhat_1}{\lra}
  \cfhat(Y_K^n)
  \overset{\fhat_2}{\lra}
  \cfhat(Y_K^{n+\mu})
  \overset{\fhat_3}{\lra}
  \dots
\]
satisfies the assumptions of Lemma~\ref{Seidellem}, then for every pair
$\fhat_i$ and $\fhat_{i+1}$, the associated map $(H,\fhat_{i+1})_*$
is an isomorphism. With the arguments from above, i.e.~analogous to
Lemma \ref{exlem}, we conclude that 
$\mbox{\rm im}(\Fhat_i)=\ker(\Fhat_{i+1})$. Hence, Theorem \ref{surextri}
follows.

\section{The Contact Element and $\loss$}\label{parcontact}
\subsection{Contact Structures}\label{contactstruct}
A $3$-dimensional contact manifold is a pair $(Y,\xi)$ where $Y$ is a 
$3$-dimensional manifold and $\xi\subset TY$ a hyperplane bundle that 
can be written as the kernel of a $1$-form $\alpha$ with the property
\begin{equation}
  \alpha\wedge d\alpha\not=0. \label{contcond}
\end{equation}
Those $1$-forms satisfying $(\ref{contcond})$ are called {\bf contact
forms}. Given a contact manifold $(Y,\xi)$, the associated contact
form is not unique. Suppose $\alpha$ is a contact form of $\xi$ then, given 
a non-vanishing function $\lambda\co Y\lra\R^+$, we can change the contact form 
to $\lambda\alpha$ without affecting the contact 
condition (\ref{contcond}):
\[
  \lambda\alpha\wedge d(\lambda\alpha)
  =
  \lambda\alpha\wedge d\lambda\wedge\alpha
  +\lambda^2\alpha\wedge d\alpha
  =\lambda^2\alpha\wedge d\alpha
  \not=0.
\]
The existence of a contact form implies that the normal direction
$TY/\xi$ is trivial. We define a section $R_\alpha$ by
\[
  \alpha(R_\alpha)\not=0
  \;\mbox{\rm and }\;\;
  \iota_{R_\alpha}d\alpha=0.
\]
This vector field is called {\bf Reeb field} of the contact form $\alpha$.
The contact condition implies that $d\alpha$ is a non-degenerate form
on $\xi$. Thus, $\iota_{R_\alpha}d\alpha=0$ implies that for each 
point $p\in Y$ the vector $(R_\alpha)_p$ is an element of
$T_pY\backslash\xi_p$. Thus, $R_\alpha$ is a section of $TY/\xi$.
\begin{definition} Two contact manifolds $(Y,\xi)$ and $(Y',\xi')$ are
called {\bf contactomorphic} if there is a diffeomorphism 
$\phi\co Y\lra Y'$ preserving the contact structures, i.e.~such that
$T\phi(\xi)=\xi'$. The map $\phi$ is a {\bf contactomorphism}.
\end{definition}
It is a remarkable property of contact manifolds that there is a
unique standard model for these objects.
\begin{definition} The pair $(\R^3,\xistd)$, where $\xistd$ is the contact 
structure given by the kernel of the $1$-form $dz-y\,dx$, is called
{\bf standard contact space}.
\end{definition}
Every contact manifold is locally contactomorphic to the standard contact
space. This is known as {\bf Darboux's theorem}. As a consequence we
will not be able to derive contact invariants by purely local arguments,
in contrast to differential geometry where for instance curvature is
a constraint to the existing local model. 
\begin{theorem}[Gray Stability, cf.~\cite{Geiges}]
Each smooth homotopy of contact structures $(\xi_t)_{t_\in[0,1]}$ 
is induced by an ambient isotopy $\phi_t$, i.e.~the 
condition $T\phi_t(\xi_0)=\xi_t$ applies for all $t\in[0,1]$.
\end{theorem}
An isotopy induced homotopy of contact structures is called {\bf contact isotopy}.
So, a homotopy of contact structures can be interpreted as an isotopy and, vice versa, an
isotopy induces a homotopy of contact structures. As in the case of 
vector fields, we have a natural connection to isotopies, i.e.~objects 
whose existence and form will be closely related to the manifold's 
topology.\vspace{0.3cm}\\
A {\bf contact vector field} $X$ is a vector field whose local flow preserves
the contact structure. An embedded surface
$\Sigma\hookrightarrow Y$ is called {\bf convex} if there is a neighborhood
of $\Sigma$ in $Y$ in which a contact vector field exists that is transverse to
$\Sigma$. The existence of a contact vector field immediately 
implies that there is a neighborhood $\Sigma\times\R\hookrightarrow Y$ 
of $\Sigma$ in which the contact structure is invariant in 
$\R$-direction. Thus, convex surfaces are the objects along 
which we glue contact manifolds together. 
\begin{definition} A knot $K\subset Y$ is called {\bf Legendrian} if
it is tangent to the contact structure.
\end{definition}
The contact condition implies that, on a 
$3$-dimensional contact manifold $(Y,\xi)$, only 
$1$-dimensional submanifolds, i.e.~knots and 
links, can be tangent to $\xi$. Every Legendrian knot admits a
tubular neighborhood with a convex surface as boundary. Hence, it is
possible to mimic surgical constructions to define the contact geometric 
analogue of surgery theory, called {\bf contact surgery}. Contact surgery in arbitrary dimensions 
was introduced by Eliashberg in \cite{eliash2}. His construction, in dimension $3$, corresponds to
$(-1)$-contact surgeries. For $3$-dimensional contact manifolds Ding and Geiges gave
in \cite{DiGei04} a definition of contact-$r$-surgeries (cf.~also \cite{DiGei}) for 
arbitrary $r\in\Q>0$. It is nowadays one of the most significant tools for $3$-dimensional 
contact geometry. Its importance relies in the following theorem.
\begin{theorem}[see \cite{DiGei}] Given a contact 
manifold $(Y,\xi)$, there is a link
$\mathbb{L}=\mathbb{L}^+\sqcup\mathbb{L}^-$ 
in $\sthree$ such that contact-$(+1)$-surgery along the link $\mathbb{L}^+$ 
and contact-$(-1)$-surgery along $\mathbb{L}^-$ in $(\sthree,\xistd)$
yields $(Y,\xi)$.
\end{theorem}
Moreover, if we choose cleverly, we can accomplish $\mathbb{L}^+$ to have just 
one component. Using $(-1)$-contact surgeries only, we can transform an arbitrary
overtwisted contact manifold into an arbitrary (not necessarily overtwisted) contact
manifold. For a definition of overtwistedness we point the reader to \cite{Geiges}. 
Thus, starting with a knot $K$ so that $(+1)$-contact surgery along $K$
yields an overtwisted contact manifold $(Y',\xi')$, for any contact manifold $(Y,\xi)$, we can 
find a link $\mathbb{L}^-$, such that $(-1)$-contact surgery along $\mathbb{L}^-$ in $(Y',\xi')$
yields $(Y,\xi)$. An example for such a knot $K$ is the Legendrian shark, i.e.~the Legendrian
realization of the unknot with $tb=-1$ and $rot=0$.

\subsection{Open Books}\label{obvsaob}
For a detailed treatment of open books we point the reader 
to \cite{Etnyre01}.
\begin{definition} An {\bf open book} on a closed, oriented $3$-manifold
$Y$ is a pair $(B,\pi)$ defining a fibration
\[
  P\hookrightarrow Y\backslash B\overset{\pi}{\lra}\sone,
\]
where $P$ is an oriented surface with boundary $\partial P=B$.
For every component $B_i$ of $B$ there is a neighborhood 
$\iota\co D^2\times\sone\hookrightarrow\nu B_i\subset Y$ such that
the core $C=\{0\}\times\sone$ is mapped onto $B_i$ under $\iota$ and 
$\pi$ commutes with the projection $(D^2\times\sone)\backslash C\lra\sone$
given by $(r\cdot \exp(it),\exp(is))\lmt \exp(it)$. The submanifold $B$ is called 
{\bf binding} and $P$ the {\bf page of the open book}.
\end{definition}
An {\bf abstract open book} is a pair $(P,\phi)$ consisting of 
an oriented genus-$g$ surface $P$ with boundary and a homeomorphism 
$\phi\co P\lra P$ that is the identity near the boundary of $P$. 
The surface $P$ is called {\bf page} and $\phi$ 
the {\bf monodromy}. Given an abstract open book $(P,\phi)$, we
may associate to it a $3$-manifold. Let $c_1,\dots,c_k$ denote 
the boundary components of $P$. Observe that
\begin{equation}
  (P\times[0,1])/(p,1)\sim(\phi(p),0) \label{ob:01}
\end{equation}
is a $3$-manifold. Its boundary is given by the tori
\[
  \left((c_i\times[0,1])/(p,1)\sim(p,0)\right)\cong c_i\times\sone.
\]
Fill in each of the holes with a full torus $\disc^2\times\sone$: we glue
a meridional disc $\disc^2\times\{\star\}$ onto $\{\star\}\times\sone\subset c_i\times\sone$.
In this way we define a closed, oriented $3$-manifold $Y(P,\phi)$. 
Denote by $B$ the union of the cores of the tori 
$\disc^2\times\sone$. The set $B$ is called {\bf binding}. 
By definition of abstract open books we obtain an open book structure
\[
  P\hookrightarrow Y(P,\phi)\backslash B\lra\sone
\]
on $Y(P,\phi)$. Conversely, given an open book by cutting a small tubular neighborhood
$\nu B$ out of $Y$, we obtain a $P$-bundle over $\sone$. Thus, there
is a homeomorphism $\phi\co P\lra P$ such that
\[
  Y\backslash\nu B\cong (P\times[0,1])/(p,1)\sim(\phi(p),0).
\]
Inside the standard neighborhood $\nu B$, as given in the definition, the 
homeomorphism $\phi$ is the identity. So, the pair $(P,\phi)$ defines an 
abstract open book.
\begin{definition} Two abstract open books $(P,\phi)$ and $(P,\phi')$
are called {\bf equivalent} if there is a homeomorphism $h\co P\lra P$, which is 
the identity near the boundary, such that
$\phi\circ h=\phi'\circ h$. We denote by $\mbox{\rm ABS}(Y)$ the 
set of abstract open books $(P,\phi)$ with $Y(P,\phi)=Y$, up to equivalence.
\end{definition}
Two open books are called equivalent if they are diffeomorphic. The set 
of equivalence classes of open books is
denoted by $\mbox{\rm OB}(Y)$. An abstract open book defines an open 
book up to diffeomorphism. With the construction given above we
define a map
\[
  \Psi
  \co
  \mbox{\rm ABS}(Y)
  \lra
  \mbox{\rm OB}(Y)
\]
and its inverse. Thus, to some point, open books and abstract open books
are the same objects. Sometimes, it is more convenient to deal with 
abstract open books rather than open books themselves.

\subsection{Open Books, Contact Structures and Heegaard Diagrams}
\label{obcshd}
Given an open book $(B,\pi)$ or an abstract open book $(P,\phi)$, define
a surface $\Sigma$ by gluing together two pages at their boundary
\[
  \Sigma=P_{1/2}\cup_\partial P_{1}.
\]
The manifold $Y$ equals the union $H_0\cup H_1$ where
 $H_i=\pi^{-1}([i/2,(i+1)/2])$, $i=0,1$. Any curve $\gamma$ in $Y$ 
running from $H_0$ to $H_1$, when projected onto $\sone$, has to 
intersect $\{1/2,1\}$ at some point. Thus, the curve 
$\gamma$ is forced to intersect $\Sigma$. The submanifolds $H_i$
are handlebodies of genus $g(\Sigma)$ and
\[
  Y=H_0\cup_\partial H_1
\]
is a Heegaard decomposition of $Y$. 
\begin{definition}
A system $a=\{a_1,\dots,a_n\}$ of disjoint, properly embedded arcs on $P$
is called {\bf cut system} if $P\backslash\{a_1,\dots,a_n\}$ is
 topologically a disc.
\end{definition}
A system of arcs is a cut system if and only if it defines a basis for 
the first homology of $(P,\partial P)$.\vspace{0.3cm}\\
We interpret the curve $a_i$ as sitting on $P_{1/2}$ and 
$\overline{a_i}$, i.e.~the curve $a_i$ with reversed orientation, as 
sitting inside $P_1$. These two can be combined to 
$\alpha_i=a_i\cup_\partial\overline{a_i}$, $i=1,\dots,n$, which all sit in
$\Sigma$. 
Referring to the relation between open books and abstract open books
discussed in \S\ref{obvsaob}, observe that 
\[
  H_1
  =
  \pi^{-1}
  ([1/2,1])
  =
  (P\times[1/2,1])/\!\!\sim
\]
where $\sim$ identifies points $(p,0)$ with $(\phi(p),1)$ for $p\in P$
and points $(p,t)$ with $(p,t')$ for $p\in\partial P$ and $t,t'\in[1/2,1]$.
Thus $a_i\times[1/2,1]$ determines a disc in $H_1$ whose boundary is
$\alpha_i$.
This means we can interpret the set
$\{\alpha_1,\dots,\alpha_{n}\}$ as a set of attaching circles for
the handlebody $H_1$. The gluing of the two handlebodies $H_0$ 
and $H_1$ is given by the
pair $(id,\phi)$ where $id$ is the identity on $P_{1/2}$ and $\phi$
the monodromy, interpreted as a map $P_1\lra P_0$. These two maps
combine to a map $\partial H_1\lra\partial H_0$. Define $b_i$, 
$i=1,\dots, n$, as small push-offs of the $a_i$ that intersect 
these transversely in a single point (see Figure~\ref{Fig:figzpoint}). Then by the
gluing of the two handlebodies $H_0$ and $H_1$ the $\alpha$-curves define 
a Heegaard diagram with $\beta$-curves given by$\beta_i=b_i\cup\overline{\phi(b_i)}$, $i=1,\dots,n$. Thus the following lemma is immediate.
\begin{lem} The triple $(\Sigma,\alpha,\beta)$ is a Heegaard diagram of
$Y$.\hfill$\square$
\end{lem}
Given an abstract open book $(P,\phi)$, define $P'$ by attaching a
$1$-handle to $P$, i.e.~$P'=P\cup h^1$. Choose a knot $\gamma$ 
in $P'$ that intersects the co-core of $h^1$ once, transversely. 
The monodromy $\phi$ can be extended as the identity over $h^1$, 
and, thus, may be interpreted as a homeomorphism of $P'$. We denote by
$D_\gamma^\pm$ the positive/negative Dehn twist along~$\gamma$.
\begin{definition}\label{girsta} The abstract open book $(P',D_\gamma^\pm\circ\phi)$ is
called a {\bf positive/negative Giroux stabilization} of $(P,\phi)$.
\end{definition}
We will see that
open books, up to positive Giroux stabilizations, correspond one-to-one
to isotopy classes of contact structures. 
\begin{lem}\label{Stabilstabil} Stabilizations 
preserve the underlying $3$-manifold, i.e.~the manifolds
$Y(P',\phi')$ and $Y(P,\phi)$ are isomorphic.
\end{lem}
A priori, it is not clear that stabilizations preserve the 
associated $3$-manifold. A proof of this lemma can be found in
 \cite{Etnyre01}. But in the following we will discuss an alternative
proof. Our proof uses a construction introduced by Lisca, Ozsv\'{a}th, Stipsicz 
and Szab\'{o} (see \cite{LOSS}, Alternative proof of Theorem 2.11). 
\begin{lem}[\cite{LOSS}]\label{stabstab} There is a 
cut system $\{a_1,\dots,a_{n}\}$ on $(P,\phi)$ that is 
disjoint from $\gamma\cap P$. 
\end{lem}
\begin{proof} Denote by $\gamma'$ the arc $\gamma\cap P$. If $P\backslash \gamma'$
is connected, we choose $a_1$ to be a push-off of $\gamma'$ and then
extend it to a cut system of $P$. This is possible since $H_1(P,\partial P)$
is torsion free and $[a_1]$ a primitive element in it. 
If $P\backslash\gamma'$ disconnects into the components $P_1$ and $P_2$, then 
we may choose cut systems on $P_i$, $i=1,2$, arbitrarily. The union
of these cut systems will be a cut system of $P$ and disjoint from 
$\gamma'$.
\end{proof}
The given cut system on $P$ can be extended to a cut system on $P'$.
We can choose $a_{n+1}$ as the co-core of $h^1$. The set of curves
$a_1,\dots,a_{n+1}$ is a cut system of $P'$. Choose the $b_i$,
$i=1,\dots,n+1$, as small isotopic push-offs of the $a_i$. Then, for
$i=1,\dots,n$, we have
\[\begin{array}{rclcl}
  \phi'(b_i)&=&\phi\circ D_\gamma^\pm(b_i)
  &=&
  \phi(b_i)\\
  \phi'(b_{n+1})
  &=&
  D_\gamma^\pm\circ\phi(b_{n+1})
  &=&
  D_\gamma^\pm(b_{n+1}).
  \end{array}
\]
Consequently, $\phi'(b_{n+1})$ looks like $\gamma$ outside the
handle $h^1$. The curve $\beta_{n+1}$ has to be disjoint from 
all $\alpha_i$, $i<n+1$.
\begin{proof}[Proof of Lemma \ref{Stabilstabil}]
On the level of cobordisms the pair $\alpha_{n+1}$ and $\beta_{n+1}$
which meet in a single point correspond to a cancelling pair of
handles attached to the boundary $Y(P,\phi)\times\{1\}$ of 
$Y(P,\phi)\times I$. Thus, we have 
\[
  Y(P',\phi')=\sthree\#Y(P,\phi).
\]
\end{proof}

A contact structure $\xi$ is {\bf supported} by an open book $(B,\pi)$
of $Y$ if $\xi$ is contact isotopic to a contact structure $\xi'$
which admits a contact form $\alpha$ such that $d\alpha$ is a
positive area form on each page $P_\theta=\pi^{-1}(\theta)$ and
$\alpha>0$ on $\partial P_\theta$. We gave the definition as a matter of completeness, 
but a detailed understanding of this definition will not be interesting 
to us. For a detailed treatment we 
point the reader to \cite{Etnyre01}. Every contact structure is 
supported by an open book decomposition.
\begin{theorem}[cf.~\cite{Etnyre01}]
There is a one-to-one correspondence between isotopy classes of 
contact structures and open book decompositions up to positive
Giroux stabilization.
\end{theorem}
Given a Legendrian knot $L\subset(Y,\xi)$, we know by definition that
its tangent vector at every point of $L$ lies in $\xi$. The tangent bundle of a closed, oriented 
$3$-manifold is orientable, which especially implies the triviality of  
$\left.TY\right|_L$. The coorientability of
$\xi$ implies that $\left.\xi\right|_L$ is trivial, too. By definition of
Legendrian knots the tangent vector of $L$ lies in $\xi$. The 
$2$-dimensionality implies that $\xi$, in addition, contains a normal 
direction. The triviality of the tangent bundle over $L$ implies that
this normal direction determines a framing of $L$. This framing which
is determined by the contact structure is called {\bf contact framing}.
In case of contact surgery it plays the role of the canonical $0$-framing,
i.e.~we measure contact surgery coefficients with respect to the 
contact framing. Note that if $L$ is homologically trivial, a Seifert surface
determines a second framing on $L$. Surgery coefficients in a surgery 
presentation of a manifold are usually determined by measuring the 
surgery framing with respect to this canonical Seifert framing 
(cf.~\S\ref{parsurextri}). Measuring the contact framing with respect to 
the Seifert framing determines a number $tb(L)\in\Z$ which is called
the {\bf Thurston-Bennequin invariant}. This is certainly an invariant
of $L$ under {\bf Legendrian isotopies}, i.e.~isotopies of $L$ through
Legendrian knots. By definition, the coefficients are related by
\[
  \mbox{\rm smooth surgery coefficient}
  =
  \mbox{\rm contact surgery coefficient}
  +
  tb(L).
\]
It is possible to find an open book decomposition which supports $\xi$
such that $L$ sits on a page of the open book. Furthermore, we can
arrange the page framing and the contact framing to coincide. This is
the most important ingredient for applications of Heegaard Floer homology
in the contact geometric world. The proof relies on the fact that
it is possible to find CW-decompositions of contact manifolds which are
adapted to the contact structure. These are called {\bf contact cell 
decompositions}. The $1$-cells in such a decomposition are Legendrian arcs.
With these decompositions it is possible to directly construct an open
book supporting the contact structure. Since the $1$-cells are Legendrian 
arcs we can include a fixed Legendrian knot into the decomposition and in 
this way modify the open book such that the result follows. For details we 
point the reader to \cite{Etnyre01}.
\begin{lem}[cf.~\cite{LOSS}]\label{obsurgery} Let 
$L\subset(Y,\xi)$ be a Legendrian knot and $(P,\phi)$
an abstract open book supporting $\xi$ such that $L$ sits on a page
of the underlying open book. Let $(Y_L^\pm,\xi_L^\pm)$ denote 
the $3$-manifold obtained by $(\pm1)$-contact surgery along $L$. 
Then $(P,D_\gamma^{\mp}\circ\phi)$ is an abstract open book supporting 
the contact structure $\xi_L^\pm$.
\end{lem}

\subsection{The Contact Class}\label{conclass}
Given a contact manifold $(Y,\xi)$, we fix an open book decomposition
$(P,\phi)$ which supports $\xi$. This open book defines a Heegaard
decomposition and, with the construction stated in the last paragraph, we 
are able to define a Heegaard diagram. We now put in an additional
datum. The curves $b_i$ are isotopic push-offs of the $a_i$. We
choose them like indicated in Figure \ref{Fig:figzpoint}: We push the $b_i$
off the $a_i$ by following with $\partial b_i$ the positive boundary 
orientation of $\partial P$. 
\begin{figure}[ht!]
\labellist\small\hair 2pt
\pinlabel {Page $P\!\times\!\{1/2\}$ of the open book} [bl] at 29 187
\pinlabel {$z$} [bl] at 189 112
\pinlabel {$a_i$} [t] at 76 22
\pinlabel {$b_i$} [t] at  153 22
\pinlabel {$\partial P$} [l] at 210 27
\pinlabel {$\partial P$} [l] at 210 153
\endlabellist
\centering
\includegraphics[height=3cm]{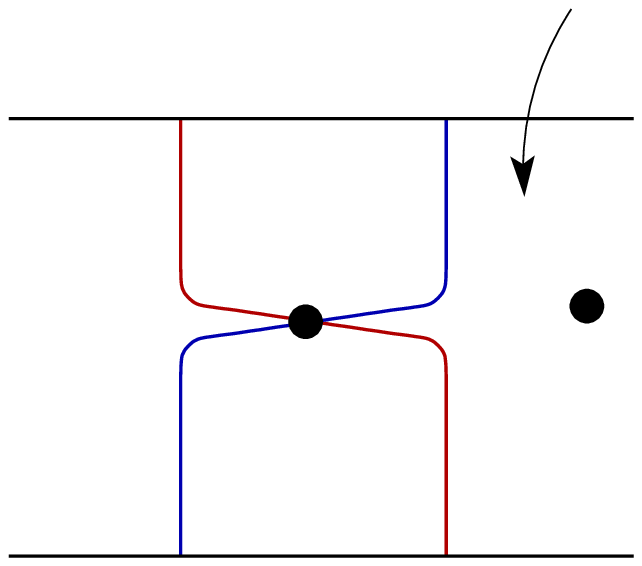}
\caption{Positioning of the point $z$ and choice of $b_i$.}
\label{Fig:figzpoint}
\end{figure}

The point $z$ is placed
outside the thin strips of isotopy between the $a_i$ and $b_i$.
We denote by $x_i$ the unique intersection point between $a_i$ 
and $b_i$. Define
\[
  EH(P,\phi,\{a_1,\dots,a_{n}\})=\{x_1,\dots,x_{n}\}.
\]
By construction of the Heegaard diagram $EH$ is a cycle in the Heegaard
Floer homology associated to the data $(-\Sigma,\alpha,\beta,z)$. We choose
the negative surface orientation since with this orientation there can be
no holomorphic Whitney disc emanating from $EH$ (cf.~Figure~\ref{Fig:figzpoint}).
\begin{lem}[see \cite{OsZa02}]\label{hfcoh} The 
Heegaard Floer cohomology $\hfhat\,^{\!*}(Y)$ is 
isomorphic to $\hfhat(-Y)$.
\end{lem}
The Heegaard diagram $(-\Sigma,\alpha,\beta)$ is a Heegaard diagram
for $-Y$ and, thus, represents the Heegaard Floer cohomology of $Y$.
Instead of switching the surface orientation we can swap the boundary
conditions of the Whitney discs at their $\alpha$-boundary 
and $\beta$-coundary, i.e.~we will be interested in Whitney discs 
in $(\Sigma,\beta,\alpha)$. 
The element $EH$ can be interpreted as sitting in the Heegaard Floer
cohomology of $Y$. The push-off $b_i$ is chosen such that there is no
holomorphic disc emanating from $x_i$.
\begin{theorem} The class $EH(P,\phi,\{a_1,\dots,a_{n}\})$ is independent
of the choices made in its definition. Moreover, the associated 
cohomology class $c(Y,\xi)$ is an isotopy invariant of the contact
 structure $\xi$, up to sign. We call $c(Y,\xi)$ {\bf contact element}.
\end{theorem}
The proof of this theorem relies on several steps we would like to sketch:
An {\bf arc slide} is a geometric move allowing us to change the cut 
system. Any two cut systems can be transformed into each other by a finite
 sequence of arc slides. Let $a_1$ and $a_2$ be two adjacent
  arcs. Adjacent means that in $P\backslash\{a_1,\dots,a_{n}\}$ one of the 
  boundary segments associated to $a_1$ and $a_2$ are connected via one 
  segment $\tau$ of $\partial P$. An arc slide of 
  $a_1$ over $a_2$ (or vice versa) is a curve in the isotopy class of 
 $a_1\cup\tau\cup a_2$. We denote it by $a_1+a_2$.
\begin{lem} Any two cut systems can be transformed into each other 
with a finite number of arc slides.
\end{lem}
It is easy to observe that an arc slide affects the associated Heegaard
diagram by two handle slides. The change under the $\alpha$-circles is
given by a handle slide of $\alpha_1$ over $\alpha_2$. But the associated
$\beta$-curve moves with the $\alpha$-curve, i.e.~we have to additionally
slide $\beta_1$ over $\beta_2$. We have to see that these handle slides
preserve the contact element. 
To be more precise: After the first handle slide we moved out of the set of
Heegaard diagrams induced by open books. Thus, we cannot see the contact element
in that diagram. After the second handle slide, however, we move back into that
set and, hence, see the contact element again. We have to check that the composition
of the maps between the Heegaard Floer cohomologies induced by the handle 
slides preserves the contact element. This is a
straightforward computation.
\begin{definition}\label{DefOne} Let a Heegaard diagram $(\Sigma,\alpha,\beta)$ 
and a homologically essential, simple, closed curve $\delta$ on $\Sigma$ be given. The
Heegaard diagram $(\Sigma,\alpha,\beta)$ is called
{\bf $\delta$-adapted} if the following conditions hold. 
\begin{enumerate}
  \item It is induced by an open book and the pair $\alpha$, $\beta$ is 
  induced by a cut system (cf.~\S\ref{obcshd}) for this open book.
  \item The curve $\delta$ intersects $\beta_1$ once and does not intersect any 
  other of the $\beta_i$, $i\geq 2$.
\end{enumerate}
\end{definition}
We can always find $\delta$-adapted Heegaard diagrams. This is already
stated in \cite{HKM} and \cite{LOSS} but not proved.
\begin{lem}\label{LemOne} Let $(P,\phi)$ be an open book and 
$\delta\subset P$ a homologically essential closed curve. There is a 
choice of cut system on $P$ that induces a $\delta$-adapted 
Heegaard diagram.
\end{lem}
Observe that  $a_1,\dots,a_{n}$ to be a cut system of a page $P$ 
essentially means to be a basis of $H_1(P,\partial P)$:
Suppose the curves are not linearly independent. In this
case we are able to identify a surface $F\subset P$, $F\not=P$,
bounding a linear combination of some of the curves $a_i$.
But this means the cut system disconnects the page $P$ in 
contradiction to the definition. Conversely, suppose the curves
in the cut system are homologically linearly independent.
In this case the curves cannot disconnect the page. If they
disconnected, we could identify a surface $F$ in $P$ with boundary
a linear combination of some of the $a_i$. But this contradicts
their linear independence. The fact that $\Sigma\backslash\{a_1,\dots,a_n\}$
is a disc shows that every element in $H_1(P,\partial P)$ can be
written as a linear combination of the curves $a_1,\dots,a_n$.
\begin{proof} Without loss of generality, we assume that $P$ 
has connected boundary: Suppose the boundary of $P$ has two
components. Choose a properly embedded arc connecting both
components of $\partial P$. Define this curve to be the first
curve $a_0$ in a cut system. Cutting out this curve $a_0$, we obtain a
surface with connected boundary. The curve $a_0$ determines two segments $S_1$
and $S_2$ in the connected boundary. We can continue using
the construction process for connected binding we state below. We 
just have to check the boundary points of the curves to 
remain outside of the segments $S_1$ and $S_2$. Given that $P$ has more
than two boundary components, we can, with this algorithm, inductively 
decrease the number of boundary components.\vspace{0.3cm}\\
The map $\phi$ is an element of the mapping 
class group of $P$. Thus, if $\{a_1,\dots,a_{n}\}$
is a cut system, then $\{\phi(a_1),\ldots,\phi(a_{n})\}$ is a 
cut system, too. It suffices to show that there is a cut system 
$\{a_1,\ldots,a_{n}\}$ such that $\delta$ intersects $a_i$ 
once if and only if $i=1$.
\begin{figure}[ht!]
\labellist\small\hair 2pt
\pinlabel $\gamma$ [bl] at 199 214
\endlabellist
\centering
\includegraphics[height=3cm]{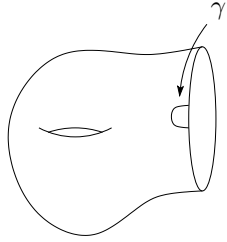}
\caption{Possible choice of curve $\gamma$.}
\label{Fig:figfour}
\end{figure}

We start by taking a band sum of $\delta$ with a small 
arc $\gamma$ as shown in Figure \ref{Fig:figfour}. We are 
free to choose the arc $\gamma$.
Denote the result of the band sum by $a_2$. The arc $a_2$ indeed bounds a compressing disc 
in the respective handlebody because its boundary lies 
on $\partial P$. Because of our prior observation it suffices
to show that $a_2$ is a primitive class 
in $H_1(P,\partial P)$.
Since $H_1(P,\partial P)$ is torsion free the primitiveness of
$a_2$ implies that we can extend $a_2$ to a basis of $H_1(P,\partial P)$.
The curves defining this basis can easily be chosen to be
not closed, with their boundary lying on $\partial P$.\vspace{0.3cm}\\
Writing down the long exact sequence of the pair $(P,\partial P)$
\begin{diagram}[size=2em,labelstyle=\scriptstyle]
H_2(P)&\rTo&H_2(P,\partial P)&\rTo^{\partial_{*}}&H_1(\partial P)      
      &\rTo &H_1(P)&\rTo^{\iota_{*}}&H_1(P,\partial P)&\rTo&0\\
\br{=}&    &  \br{$\cong$}   &                   & \br{$\cong$}        
      &     &      &                &                 &    &\\
  0   &\rTo&\Z\bigl<[P]\bigr>     &\rTo^{\partial_{*}}
      &\Z\bigl<[\partial P]\bigr>&\rTo &H_1(P)&\rTo^{\iota_{*}}
      &H_1(P,\partial P)&\rTo&0
\end{diagram}
we see that $\partial_*$ is surjective since $\partial_*[P]=[\partial P]$. Hence, exactness of the 
sequence implies that the inclusion $\iota\co P\lra(P,\partial P)$ 
induces an isomorphism on homology. Note that the zero at the end 
of the sequence appears because $\partial P$ is assumed to be 
connected. Let $g$ denote the genus of $P$. Of course $H_1(P;\Z)$ is $\Z^{2g}$, which can be seen 
by a Mayer-Vietoris argument or from handle decompositions of 
surfaces (compute the homology using a handle decomposition). Since 
$\delta$ was embedded it follows from the lemma below
that it is a primitive class in $H_1(P;\Z)$. The isomorphism 
$\iota_*$ obviously sends $\delta$ to $a_2$, i.e.~$\iota_*[\delta]=[\gamma]$. 
Thus, $a_2$ is primitive in $H_1(P,\partial P)$.

Cut open the surface along $\delta$. We obtain two new boundary components, $C_1$
and $C_2$ say, which we can connect with the boundary of $P$ with two arcs. These
two arcs, in $P$, determine a properly embedded curve, $a_1$ say, whose boundary
lies on $\partial P$. Furthermore, $a_1$ intersects $\delta$ in one single point, transversely.
The curve $a_1$ is primitve, too. To see, that we can extend to a cut system such that $\delta$ is
disjoint from $a_3,\dots,a_n$, cut open the surface $P$ along $\delta$ and $a_1$.
We obtain a surface $P'$ with one boundary component. The curves $\delta$ and $a_1$ determine
$4$ segments, $S_1,\dots, S_4$ say, in this boundary. We extend $a_2$ to a cut system 
$a_2,\dots,a_n$ of $P'$ and arrange the boundary points of the curves $a_3,\dots, a_n$ to 
be disjoint from $S_1,\dots,S_4$. The set $a_1,\dots,a_n$ is a cut system of $P$ with the
desired properties.
\end{proof}
As a consequence of the proof we may arrange $\delta$
to be a push-off of $a_2$ outside a small neighborhood where the
band sum is performed. Geometrically spoken, we cut open $\delta$
at one point, and move the boundaries to $\partial P$ to get
$a_2$. Given a positive Giroux stabilization, we can find a special 
cut system which is adapted to the curve $\gamma$. It is not hard to 
see that there is only one homotopy class of triangles that connect
the old with the new contact element and that the associated moduli 
space is a one-point space.
\begin{lem}\label{primitive} An embedded circle $\delta$ in an 
orientable, compact surface $\Sigma$ which is homologically essential 
is a primitive class of $H_1(\Sigma,\Z)$.
\end{lem}
\begin{proof}
Cut open the surface $\Sigma$ along $\delta$. We obtain a connected surface $S$
with two boundary components since $\delta$ is homologically essential in $\Sigma$. We can recover the surface $\Sigma$ by connecting 
both boundary components of $S$ with a $1$-handle and then capping off with a disc. There is 
a knot $K\subset S\cup h^1$ 
intersecting the co-core of $h^1$ only once and intersecting 
$\delta$ only once, too. To construct this knot take a union of 
two arcs in $S\cup h^1$ in the following way: Namely, define  $a$ as the core 
of $h^1$, i.e.~as $D^1\times\{0\}\subset D^1\times D^1\cong h^1$ and let $b$ be 
a curve in $S$, connecting the two components of the attaching sphere $h^1$ in $\partial S$. We 
define $K$ to be $a\cup b$. 
Obviously,
\[
  \pm1
  =
  \#(K,\delta)
  =
  \bigl<PD[K],[\delta]\bigr>.
\]
Since $H_1(\Sigma;\Z)$ is torsion, free 
$H^1(\Sigma;\Z)\cong\mbox{\rm Hom}(H_1(\Sigma;\Z),\Z)$. Thus, 
$[\delta]$ is primitive.
\end{proof}
Recall that a positive/negative Giroux stabilization of an open book $(P,\phi)$ 
is defined as the open book $(P',D^\pm_\gamma\circ\phi)$ where $P'$ is defined
by attaching a $1$-handle to $P$ and $\gamma$ is a embedded, simple closed curve
in $P'$ that intersects the co-core of $h^1$ once (see Definition \ref{girsta}). Using the proofs of 
Lemma \ref{Stabilstabil} and Lemma \ref{stabstab}, we see that there is a cut system
$\{a_1,\dots,a_{n+1}\}$ of the
stabilized open book such that $\gamma$ intersects only $a_{n+1}$ which
is the co-core of $h^1$. Denote by $\alpha=\{\alpha_1,\dots,\alpha_n\}$ 
the associated attaching
circles. We define a map
\[
  \Phi
  \co
  \cfhat(\Sigma,\alpha,\beta,z)
  \lra
  \cfhat(\Sigma\#T^2,\alpha\cup\{\alpha_{n+1}\},\beta\cup\{\beta_{n+1}\},z)
\]
by assigning to $x\in\talpha\cap\tbeta$ the element $\Phi(x)=(x,q)$ where 
$q$ is the unique intersection point $\gamma\cap a_{n+1}$. This is an 
isomorphism by reasons similar to those given in 
Example \ref{exam01}.\vspace{0.3cm}\\
With our preparations done, we can easily prove one of the most 
significant properties of the contact element: Its 
functoriality under $(+1)$-contact surgeries. We will outline the
proof since it can be regarded as a model proof.
\begin{theorem}[\cite{OsZa05}] Let $(Y',\xi')$ be obtained 
from $(Y,\xi)$ by $(+1$)-contact surgery along a Legendrian knot $L$. 
Denote by $W$ the associated cobordism. Then the map
\[
  \Fhat_{-W}
  \co
  \hfhat(-Y)
  \lra
  \hfhat(-Y')
\]
preserves the contact element, i.e.~$\Fhat_{-W}(c(Y,\xi))=c(Y',\xi')$.
\end{theorem}
\begin{proof}
\begin{figure}[ht!]
\labellist\small\hair 2pt
\pinlabel {$x'_1$} [B] at 149 230
\pinlabel {$z$} [r] at 201 239
\pinlabel {$\dom_z$} [l] at 255 242
\pinlabel {$\gamma_1$} [B] at 272 167
\pinlabel {$\beta_1$} [t] at 12 129
\pinlabel {$\hattheta_1$} [t] at 97 129
\pinlabel {$x_1$} [t] at 147 73
\pinlabel {$\alpha_1$} [t] at 283 97
\pinlabel {Domain of a holomorphic triangle} [t] at 133 23
\pinlabel {$_1$} [t] at 38 208
\pinlabel {$_2$} [l] at 70 236
\endlabellist
\centering
\includegraphics[height=4.5cm]{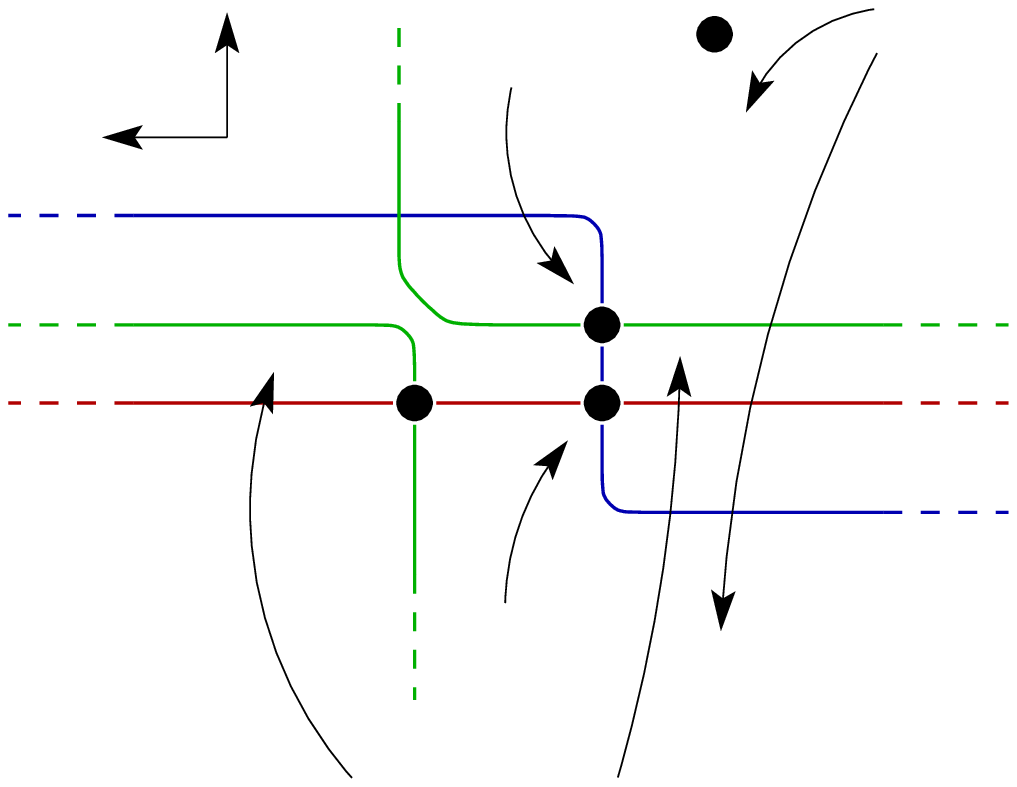}
\caption{Significant part of the Heegaard triple diagram.}
\label{Fig:figEleven}
\end{figure}
 Let an open book $(P,\phi)$ adapted to $(Y,\xi,L)$ be 
given. By Lemma \ref{obsurgery}, a $(+1)$-contact surgery acts on the
monodromy as a composition with a negative Dehn twist. Without loss
of generality, the knot $L$ just intersects $\beta_1$ once, transversely
and is disjoint from the other $\beta$-circles. Moreover, we can arrange
the associated Heegaard triple to look as indicated in 
Figure \ref{Fig:figEleven}.
The contact element $c(Y,\xi)$ is represented by the point
 $\{x_1,\dots,x_n\}$. Obviously, there is only 
one domain which carries a holomorphic triangle. It is the small holomorphic
triangle connecting $x_1$ and $x_1'$ (cf.~\S\ref{parhsinvar}). Thus, 
there is only one domain with positive coefficients, with $n_z=0$, 
connecting the points $\{x_1,\dots,x_n\}$ with $\{x_1',\dots,x_n'\}$. 
By considerations similar to those given at the end of the proof of 
Lemma \ref{thetatransform}, we see that the associated moduli space 
is a one-point space. Hence, the result follows.
\end{proof}

\subsection{The Invariant $\loss$}\label{invariantLOSS}
Ideas very similar to those used to define the contact element can
be utilized to define an invariant of Legendrian knots we will 
briefly call LOSS. This invariant is due to
{\bf L}isca, {\bf O}zsv\'{a}th, {\bf S}tipsicz and {\bf S}zab\'{o}
and was defined in \cite{LOSS}. It is basically the contact element
but now it is interpreted as sitting in a filtered Heegaard Floer complex.
The filtration is constructed with respect to a fixed Legendrian 
knot:
\begin{figure}[ht!]
\labellist\small\hair 2pt
\pinlabel {Page $P\!\times\!\{1/2\}$ of the open book} [bl] at 131 189
\pinlabel $w$ [l] at 95 116
\pinlabel $z$ [b] at 167 104
\pinlabel $w$ [tl] at 332 53
\pinlabel $z$ [t] at 408 85
\endlabellist
\centering
\includegraphics[height=3cm]{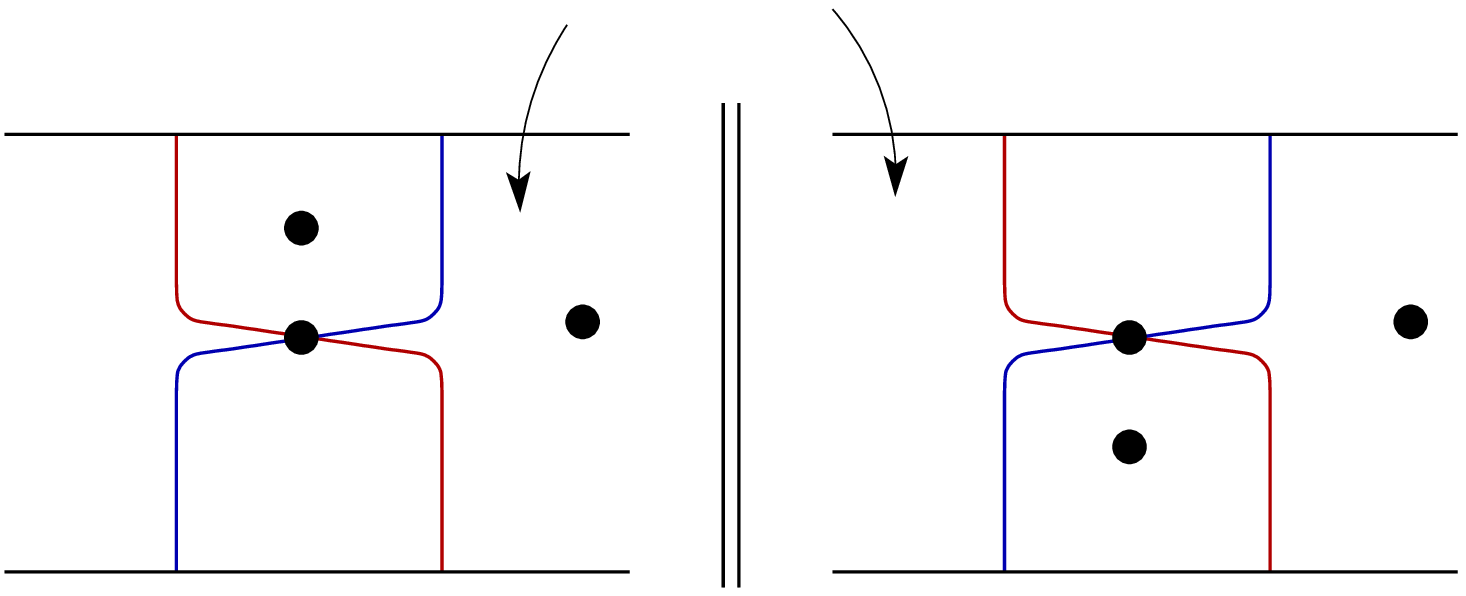}
\caption{Positioning of the point $w$ depending on the knot orientation.}
\label{Fig:wpointpos}
\end{figure}

Let $(Y,\xi)$ be a contact manifold and $L\subset Y$ a Legendrian 
knot. There is an open book decomposition of $Y$, subordinate to $\xi$, such 
that $L$ sits on the page $P\times\{1/2\}$ of the open book (cf.~\S\ref{obcshd}).
Choose a cut system that induces an $L$-adapted Heegaard diagram (cf.~\S\ref{conclass}, Definition \ref{DefOne} and Lemma \ref{LemOne}).
Figure \ref{Fig:wpointpos} illustrates the positioning of a point $w$ in
the Heegaard diagram induced by the open book. 
Similar to the case
of the contact element those intersection points $\alpha_i\cap\beta_i$ who sit
on $P\times\{1/2\}$ determine one specific generator of $\cfhat(-Y)$. This element may be interpreted as
sitting in $\cfkhat(-Y,L)$, and it is a cycle there, too. The induced
element in the knot Floer homology is denoted by $\loss(L)$.
\begin{rem} \begin{enumerate}
\item[(1)] Since this is an important issue we would like to recall the
relation between the pair $(w,z)$ and the knot orientation. In homology
we connect $z$ with $w$ in the complement of the $\alpha$-curves and 
$w$ with $z$ in the complement of the $\beta$-curves (oriented as is
obvious from the definition). In {\bf cohomology} we orient in the opposite
manner, i.e.~we move from $z$ to $w$ in the complement of the $\beta$-curves
and from $w$ to $z$ in the complement of the $\alpha$-curves.
\item[(2)] Observe, that the definition of the invariant $\loss$ as well as
the contact element always comes with a specific presentation of the groups
$\hfhat$ and $\hfkhat$. If we want to compare for instance invariants of
two different Legendrian knots, we have to get rid of the presentation in
the background. This can be done by modding out a certain mapping class group action on the
homologies. We point the reader to \cite{OzSti}.
\end{enumerate}
\end{rem}
Analogous to the properties of the contact element the invariant $\loss$ is
preserved when a $(+1)$-contact surgery is performed in its complement (see~\cite{OzSti}):
Suppose we are given a contact manifold $(Y,\xi)$ with two Legendrian knots $L$ and $S$ sitting in it.
Performing a $(+1)$-contact surgery along $S$, denote by $W$ the associated cobordism. Furthermore, we
denote by $(Y_S,\xi_S)$ the result of the contact surgery. The cobordism $-W$ induces a map
\[
  \Fhat_{-W}
  \co
  \hfkhat(-Y,L)
  \lra
  \hfkhat(-Y_S,L_S)
\]
such that $\Fhat_{-W}(\loss(L))=\loss(L_S)$. Here, $L_S$ denotes the knot $L$ in the manifold
$Y_S$. Observe, that the cobordism maps constructed for the hat-theory can be defined the same
way for knot Floer homologies.

Finally, the contact element and the invariant $\loss$ are connected, too. Performing
a $(+1)$-contact surgery along the knot $L$, denote by $W$ the associated cobordism and by $(Y_L,\xi_L)$ 
the result of the contact surgery. The cobordism $-W$ induces a map
\[
  \Gamma_{-W}
  \co
  \hfkhat(-Y,L)
  \lra
  \hfhat(-Y_L^+)
\] 
such that $\Gamma_{-W}(\loss(L))=c(\xi_L^+)$ (see~\cite{Saha01}). This map is not defined by
counting holomorphic triangles. It needs a specific construction we do not outline here. We
point the interested reader to \cite{Saha01}.

\addcontentsline{toc}{section}{Bibliography}
\bibliographystyle{amsplain}
\providecommand{\bysame}{\leavevmode\hbox to3em{\hrulefill}\thinspace}
\providecommand{\MR}{\relax\ifhmode\unskip\space\fi MR }
\providecommand{\MRhref}[2]{%
  \href{http://www.ams.org/mathscinet-getitem?mr=#1}{#2}
}
\providecommand{\href}[2]{#2}

\end{document}

%% file: fontsforthisfile.tex
\DeclareFontFamily{OT1}{ptm}{}
\DeclareFontShape{OT1}{ptm}{m}{n} { <-> ptmr}{}
\DeclareFontShape{OT1}{ptm}{m}{it}{ <-> ptmri}{}
\DeclareFontShape{OT1}{ptm}{m}{sl}{ <->ptmro}{}
\DeclareFontShape{OT1}{ptm}{m}{sc}{ <-> ptmrc}{}
\DeclareFontShape{OT1}{ptm}{b}{n} { <-> ptmb}{}
\DeclareFontShape{OT1}{ptm}{b}{it}{ <-> ptmbi}{}     
\DeclareFontShape{OT1}{ptm}{bx}{n} {<->ssub * ptm/b/n}{}
\DeclareFontShape{OT1}{ptm}{bx}{it}{<->ssub * ptm/b/it}{}

\DeclareSymbolFont{bold}{OT1}{ptm}{b}{n}
\DeclareMathAlphabet{\mathbf}{OT1}{ptm}{b}{n}  
\DeclareMathAlphabet{\mathrm}{OT1}{ptm}{m}{n}

\DeclareFontFamily{OT1}{psy}{}      
\DeclareFontShape{OT1}{psy}{m}{n}{ <-> s * [0.9] psyr}{}
\DeclareFontFamily{OMS}{ptm}{}     
\DeclareFontShape{OMS}{ptm}{m}{n}{ <8> <9> <10> gen * cmsy }{}
\DeclareFontFamily{OMS}{cmtt}{}     
\DeclareFontShape{OMS}{cmtt}{m}{n}{ <8> <9> <10> gen * cmsy }{}

\SetSymbolFont{operators}{normal}{OT1}{ptm}{m}{n}   
\SetSymbolFont{operators}{bold}{OT1}{ptm}{b}{n}     
\DeclareSymbolFont{emsy}{OT1}{ptm}{m}{it}
\DeclareSymbolFont{emsr}{OT1}{ptm}{m}{n}
\DeclareSymbolFont{emcmr}{OT1}{cmr}{m}{n}   
\DeclareSymbolFont{emsymb}{OT1}{psy}{m}{n}  
\DeclareMathSymbol a{\mathalpha}{emsy}{"61}
\DeclareMathSymbol b{\mathalpha}{emsy}{"62}
\DeclareMathSymbol c{\mathalpha}{emsy}{"63}
\DeclareMathSymbol d{\mathalpha}{emsy}{"64}
\DeclareMathSymbol e{\mathalpha}{emsy}{"65}
\DeclareMathSymbol f{\mathalpha}{emsy}{"66}
\DeclareMathSymbol g{\mathalpha}{emsy}{"67}
\DeclareMathSymbol h{\mathalpha}{emsy}{"68}
\DeclareMathSymbol i{\mathalpha}{emsy}{"69}
\DeclareMathSymbol j{\mathalpha}{emsy}{"6A}
\DeclareMathSymbol k{\mathalpha}{emsy}{"6B}
\DeclareMathSymbol l{\mathalpha}{emsy}{"6C}
\DeclareMathSymbol m{\mathalpha}{emsy}{"6D}
\DeclareMathSymbol n{\mathalpha}{emsy}{"6E}
\DeclareMathSymbol o{\mathalpha}{emsy}{"6F}
\DeclareMathSymbol p{\mathalpha}{emsy}{"70}
\DeclareMathSymbol q{\mathalpha}{emsy}{"71}
\DeclareMathSymbol r{\mathalpha}{emsy}{"72}
\DeclareMathSymbol s{\mathalpha}{emsy}{"73}
\DeclareMathSymbol t{\mathalpha}{emsy}{"74}
\DeclareMathSymbol u{\mathalpha}{emsy}{"75}
\DeclareMathSymbol v{\mathalpha}{emsy}{"76}
\DeclareMathSymbol w{\mathalpha}{emsy}{"77}
\DeclareMathSymbol x{\mathalpha}{emsy}{"78}
\DeclareMathSymbol y{\mathalpha}{emsy}{"79}
\DeclareMathSymbol z{\mathalpha}{emsy}{"7A}
\DeclareMathSymbol A{\mathalpha}{emsy}{"41}
\DeclareMathSymbol B{\mathalpha}{emsy}{"42}
\DeclareMathSymbol C{\mathalpha}{emsy}{"43}
\DeclareMathSymbol D{\mathalpha}{emsy}{"44}
\DeclareMathSymbol E{\mathalpha}{emsy}{"45}
\DeclareMathSymbol F{\mathalpha}{emsy}{"46}
\DeclareMathSymbol G{\mathalpha}{emsy}{"47}
\DeclareMathSymbol H{\mathalpha}{emsy}{"48}
\DeclareMathSymbol I{\mathalpha}{emsy}{"49}
\DeclareMathSymbol J{\mathalpha}{emsy}{"4A}
\DeclareMathSymbol K{\mathalpha}{emsy}{"4B}
\DeclareMathSymbol L{\mathalpha}{emsy}{"4C}
\DeclareMathSymbol M{\mathalpha}{emsy}{"4D}
\DeclareMathSymbol N{\mathalpha}{emsy}{"4E}
\DeclareMathSymbol O{\mathalpha}{emsy}{"4F}
\DeclareMathSymbol P{\mathalpha}{emsy}{"50}
\DeclareMathSymbol Q{\mathalpha}{emsy}{"51}
\DeclareMathSymbol R{\mathalpha}{emsy}{"52}
\DeclareMathSymbol S{\mathalpha}{emsy}{"53}
\DeclareMathSymbol T{\mathalpha}{emsy}{"54}
\DeclareMathSymbol U{\mathalpha}{emsy}{"55}
\DeclareMathSymbol V{\mathalpha}{emsy}{"56}
\DeclareMathSymbol W{\mathalpha}{emsy}{"57}
\DeclareMathSymbol X{\mathalpha}{emsy}{"58}
\DeclareMathSymbol Y{\mathalpha}{emsy}{"59}
\DeclareMathSymbol Z{\mathalpha}{emsy}{"5A}
\DeclareMathSymbol{\bullet}{\mathalpha}{emsymb}{"B7}
\DeclareMathSymbol{\regis}{\mathalpha}{emsymb}{"D2}
\def\Bullet{\leavevmode\unkern{$\m@th\bullet$}\kern.32em\ignorespaces}
\def\Regis{\leavevmode\raise.5ex\hbox{$\m@th\regis$}}
\DeclareMathSymbol +{\mathbin}{emcmr}{`+}
\DeclareMathSymbol ={\mathrel}{emcmr}{`=}  
\DeclareMathSymbol{\Gamma}{\mathalpha}{emcmr}{"00}
\DeclareMathSymbol{\Delta}{\mathalpha}{emcmr}{"01}
\DeclareMathSymbol{\Theta}{\mathalpha}{emcmr}{"02}
\DeclareMathSymbol{\Lambda}{\mathalpha}{emcmr}{"03}
\DeclareMathSymbol{\Xi}{\mathalpha}{emcmr}{"04}
\DeclareMathSymbol{\Pi}{\mathalpha}{emcmr}{"05}
\DeclareMathSymbol{\Sigma}{\mathalpha}{emcmr}{"06}
\DeclareMathSymbol{\Upsilon}{\mathalpha}{emcmr}{"07}
\DeclareMathSymbol{\Phi}{\mathalpha}{emcmr}{"08}
\DeclareMathSymbol{\Psi}{\mathalpha}{emcmr}{"09}
\DeclareMathSymbol{\Omega}{\mathalpha}{emcmr}{"0A}
\DeclareMathSizes{7.6}{8}{6}{5}

%% file: layoutforthisfile.tex
\newskip\stdskip                      
\newcommand{\stdspace}{\hskip 0.75em plus 0.15em \ignorespaces}
\let\qua\stdspace  

%% file: abbreviations.tex
\newcommand{\sone}{\mathbb{S}^1}
\newcommand{\lra}{\longrightarrow}
\newcommand{\lmt}{\longmapsto}
%

%
%
\newcommand{\ztwo}{\mathbb{Z}_2}
\newcommand{\Z}{\mathbb{Z}}
\newcommand{\RP}{\mbox{\rm RP}}
\newcommand{\R}{\mathbb{R}}
\newcommand{\N}{\mathbb{N}}
\newcommand{\C}{\mathbb{C}}
\newcommand{\Q}{\mathbb{Q}}
\newcommand{\spinc}{\mbox{\rm Spin}^c}
\newcommand{\spiny}{\mbox{\rm Spin}^c_3}
\newcommand{\homology}{\mathcal{H}}
\newcommand{\phitilde}{\widetilde{\phi}}
\newcommand{\phibar}{\overline{\phi}}
\newcommand{\grading}{\mbox{\rm gr}}
\newcommand{\riemop}{\partial_{\com_s}}
\newcommand{\banbund}{\mathcal{B}}

\def\co{\colon\thinspace}
\newcommand{\stwo}{\mathbb{S}^2}
\newcommand{\pd}{\text{PD}}
\newcommand{\sprung}{\\[0.3cm]}
\newcommand{\fund}{\pi_1}
\newcommand{\kerg}{\mbox{\rm Ker}_G\,}
\def\Ker#1{\mbox{\rm Ker}_{#1}\,}
\newcommand{\im}{\mbox{\rm Im}\,}
\newcommand{\id}{\mbox{\rm id}}
\newcommand{\sothree}{\mathbb{SO}_3}
\newcommand{\sthree}{\mathbb{S}^{3}}
\newcommand{\disc}{\mbox{\rm D}}
\newcommand{\systwo}{C^{\infty}(Y,\stwo)}
\newcommand{\inner}{\text{int}}
\newcommand{\bk}{\backslash}

%
%
\newcommand{\contstand}{(\mathbb{R}^3,\xi_0)}
\newcommand{\cont}{(M,\xi)}
\newcommand{\sym}{\xi_{sym}}
\newcommand{\xistd}{\xi_{std}}
\newcommand{\cyl}{\mbox{\rm Cyl}_{r_0}^\mu}
\newcommand{\stap}{\mbox{\rm S}_+}
\newcommand{\stam}{\mbox{\rm S}_-}
\newcommand{\stapm}{\mbox{\rm S}_\pm}
\newcommand{\modulo}{\;\;\mbox{\rm mod}\,}

%
%
\newcommand{\crit}{\mbox{\rm Crit}}
\newcommand{\MC}{\mbox{\rm MC}}
\newcommand{\bmorse}{\partial^{\mbox{\rm \begin{tiny}M\!C\end{tiny}}}}
\newcommand{\M}{\mathcal{M}}
\newcommand{\stable}{\mbox{\rm W}^s}
\newcommand{\unstable}{\mbox{\rm W}^u}
\newcommand{\ind}{\mbox{\rm ind}}
\newcommand{\Mhat}{\widehat{\M}}
\newcommand{\modphi}{\mathcal{M}_\phi}
%
%
%

%
%
\newcommand{\com}{\mathcal{J}}
\newcommand{\moduli}{\mathcal{M}_{\mathcal{J}_s}(x,y)}
\newcommand{\modulit}{\mathcal{M}_{J_{s,t}}}
\newcommand{\modulittau}{\mathcal{M}_{\com_{s,t}(\tau)}}
\newcommand{\modulittauphi}{\mathcal{M}_{\com_{s,t}(\tau),\phi}}
\newcommand{\modulittaubig}{\mathcal{M}_{\com_{s,t}(\tau)}}
\newcommand{\modulittaubigphi}{\mathcal{M}_{\com_{s,t}(\tau),\phi}}
\newcommand{\modhat}{\widehat{\mathcal{M}}_{J_s}(x,y)}
\newcommand{\modhatxy}{\widehat{\mathcal{M}}(x,y)}
\newcommand{\modhatyw}{\widehat{\mathcal{M}}(y,w)}
\newcommand{\modhatone}{\widehat{\mathcal{M}}_{J_{s,1}}}
\newcommand{\modhatzero}{\widehat{\mathcal{M}}_{J_{s,0}}}
\newcommand{\modhatphi}{\widehat{\mathcal{M}}_{\phi}}
\newcommand{\modhatphib}{\widehat{\mathcal{M}}_{[\phi]}}
\newcommand{\moduliiso}{\mathcal{M}^{t}}
\newcommand{\modspace}{\mathcal{M}}
\newcommand{\modfamily}{\mathcal{M}_{\com_{s,t}}}
\newcommand{\modfamilyphi}{\mathcal{M}_{\com_{s,t},\phi}}
\newcommand{\modtriangle}{\mathcal{M}^{\Delta}}

\newcommand{\fglue}{f_{\mbox{\rm {\tiny glue}}}}
\newcommand{\phihat}{\widehat{\phi}}
\newcommand{\Phihat}{\widehat{\Phi}}
\newcommand{\Psihat}{\widehat{\Psi}}
\newcommand{\Phiinfty}{\Phi^\infty}
\newcommand{\Hhat}{\widehat{H}}
\newcommand{\D}{\mathbb{D}}
\newcommand{\Dhat}{\widehat{\D}}
\newcommand{\cops}{\partial_{J_s}}
\newcommand{\talpha}{\mathbb{T}_\alpha}
\newcommand{\tbeta}{\mathbb{T}_\beta}
\newcommand{\tgamma}{\mathbb{T}_\gamma}
\def\marge#1{\marginpar{\scriptsize{#1}}}
\def\br#1{\begin{rotate}{90}#1\end{rotate}}
\newcommand{\hfhat}{\widehat{\mbox{\rm HF}}}
\newcommand{\sfh}{\mbox{\rm SFH}}
\newcommand{\sbottom}{\underline{s}}
\newcommand{\tbottom}{\underline{t}}
\newcommand{\cfhat}{\widehat{\mbox{\rm CF}}}
\newcommand{\cfinfty}{\mbox{\rm CF}^\infty}
\def\cfbb#1{\mbox{\rm CF}^+_{\leq #1}}
\def\cfbo#1{\mbox{\rm CF}^-_{\geq -#1}}
\newcommand{\cfleq}{\mbox{\rm CF}^{\leq 0}}
\newcommand{\cfcirc}{\mbox{\rm CF}^\circ}
\newcommand{\cfkinfty}{\mbox{\rm CFK}^{\infty}}
\newcommand{\cfkhat}{\widehat{\mbox{\rm CFK}}}
\newcommand{\cfkminus}{\mbox{\rm CFK}^{-}}
\newcommand{\cfkpstar}{\mbox{\rm CFK}^{+,*}}
\newcommand{\cfkostar}{\mbox{\rm CFK}^{0,*}}
\newcommand{\hfkcirc}{\mbox{\rm HFK}^\circ}
\newcommand{\hfkhat}{\widehat{\mbox{\rm HFK}}}
\newcommand{\hfkplus}{\mbox{\rm HFK}^+}
\newcommand{\hfkminus}{\mbox{\rm HFK}^-}
\newcommand{\hfkinfty}{\mbox{\rm HFK}^\infty}
\def\cfinftyfilt#1#2{\mbox{\rm CFK}^{#1,#2}}
\newcommand{\hfinfty}{\mbox{\rm HF}^\infty}
\newcommand{\hfinftwist}{\underline{\mbox{\rm {HF}}}^\infty}
\newcommand{\fhat}{\widehat{f}}
\newcommand{\fcirc}{f^\circ}
\newcommand{\Fhat}{\widehat{F}}
\newcommand{\Fcirc}{F^\circ}
\newcommand{\hattheta}{\widehat{\Theta}}
\newcommand{\shattheta}{\widehat{\theta}}
\newcommand{\cfminus}{\mbox{\rm CF}^-}
\newcommand{\hfminus}{\mbox{\rm HF}^-}
\newcommand{\cfplus}{\mbox{\rm CF}^+}
\newcommand{\hfplus}{\mbox{\rm HF}^+}
\newcommand{\hfcirc}{\mbox{\rm HF}^\circ}
\def\hfbb#1{\hfplus_{\leq #1}}
\def\hfbo#1{\hfminus_{\geq -#1}}
\newcommand{\Hs}{\mathcal{H}_s}
\newcommand{\gr}{\mbox{gr}}
\newcommand{\parinfty}{\partial^\infty}
\newcommand{\parhat}{\widehat{\partial}}
\newcommand{\parplus}{\partial^+}
\newcommand{\parminus}{\partial^-}
\newcommand{\symg}{\mbox{\rm Sym}^g(\Sigma)}
\newcommand{\symgg}{\mbox{\rm Sym}^{2g}(\Sigma)}
\newcommand{\symc}{\mbox{\rm Sym}^g(\mathbb{C})}
\newcommand{\pitwo}{\pi_2}
\newcommand{\pitwoham}{\pi_2^{t}}
\newcommand{\symcon}{\mbox{\rm Sym}^g(\Sigma_1\#\Sigma_2)}
\newcommand{\symgone}{\mbox{\rm Sym}^{g_1}(\Sigma_1)}
\newcommand{\symgtwo}{\mbox{\rm Sym}^{g_2}(\Sigma_2)}
\newcommand{\symgmo}{\mbox{\rm Sym}^{g-1}(\Sigma)}
\newcommand{\symggmo}{\mbox{\rm Sym}^{2g-1}(\Sigma)}
\newcommand{\dom}{\mathcal{D}}
\newcommand{\bigtrans}{\left.\bigcap\hspace{-0.27cm}\right|\hspace{0.1cm}}
\newcommand{\tlt}{\times\ldots\times}
\newcommand{\Isotopy}{\widehat{\Gamma}_{\Psi_t}}
\newcommand{\Isotopyinverse}{\widehat{\Gamma}_{\Psi_{1-t}}}
\newcommand{\orient}{\mathnormal{o}}
\newcommand{\ob}{\mathnormal{ob}}
\newcommand{\SL}{\mbox{\rm SL}}
\newcommand{\rhotilde}{\widetilde{\rho}}
\newcommand{\domstar}{\dom_*}
\newcommand{\domststar}{\dom_{**}}
\newcommand{\betaprime}{\beta'}
\newcommand{\betapp}{\beta''}
\newcommand{\betatilde}{\widetilde{\beta}}
\newcommand{\deltaprime}{\delta'}
\newcommand{\tbetaprime}{\mathbb{T}_{\beta'}}
\newcommand{\talphaprime}{\mathbb{T}_{\alpha'}}
\newcommand{\tdelta}{\mathbb{T}_{\delta}}
\newcommand{\phidelta}{\phi^{\Delta}}
\newcommand{\domtilde}{\widetilde{\dom}}
\newcommand{\loss}{\widehat{\mathcal{L}}}
\newcommand{\bargamma}{\overline{\Gamma}}
\newcommand{\alphaprime}{\alpha'}
\newcommand{\ga}{\Gamma_{\alpha;\beta',\beta''}}
\newcommand{\gbone}{\Gamma_{\alpha;\beta,\widetilde{\beta}}^{w,1}}
\newcommand{\gbtwo}{\Gamma_{\alpha;\beta,\widetilde{\beta}}^{w,2}}
\newcommand{\gbthree}{\Gamma_{\alpha;\beta,\widetilde{\beta}}^{w,3}}
\newcommand{\gbfour}{\Gamma_{\alpha;\beta,\widetilde{\beta}}^{w,4}}
\newcommand{\gcone}{\Gamma_{\alpha;\delta,\delta'}^{w,1}}
\newcommand{\gctwo}{\Gamma_{\alpha;\delta,\delta'}^{w,2}}
\newcommand{\gcthree}{\Gamma_{\alpha;\delta,\delta'}^{w,3}}
\newcommand{\gcfour}{\Gamma_{\alpha;\delta,\delta'}^{w,4}}
\newcommand{\xpi}{x^+_i}
\newcommand{\xmi}{x^-_i}
\newcommand{\xp}{x^+}
\newcommand{\xm}{x^-}
\newcommand{\deltatilde}{\widetilde{\delta}}
\newcommand{\tbetatilde}{\mathbb{T}_{\widetilde{\beta}}}

\newcommand{\eab}{\epsilon_{\alpha\beta}}
\newcommand{\ead}{\epsilon_{\alpha\delta}}
\newcommand{\hqhat}{\widehat{\mbox{\rm HQ}}}
\newcommand{\cupb}{\cup_\partial}
\newcommand{\oa}{\overline{a}}
\newcommand{\ab}{\alpha\beta}
\newcommand{\ad}{\alpha\delta}
\newcommand{\adb}{\alpha\delta\beta}
\newcommand{\tila}{\widetilde{a}}
\newcommand{\tilb}{\widetilde{b}}
\def\pdehn#1#2{D_{#1}^{+,#2}} 
\def\ndehn#1#2{D_{#1}^{-,#2}}

%
%
\newcommand{\bund}{\mathcal{P}}
\newcommand{\diag}{\Delta^{\!\!E}}
\newcommand{\inter}{m_{\diag}}
\newcommand{\ozs}{Ozsv\'{a}th}
\newcommand{\sza}{Szab\'{o}}
%

%% file: pictures.tex
%
%
\newcommand{\pictureONE}
{
\begin{figure}[ht!]
\centerline{\psfig{file=moduli,height=5cm}}
\caption{A $1$-dimensional moduli space with broken ends.}
\label{Fig:figOne}
\end{figure}
}

%
%
\newcommand{\pictureTWO}
{
\begin{figure}[ht!]
\labellist\small\hair 2pt
\pinlabel {$\hattheta_{\gamma\delta}$} [l] at 198 383
\pinlabel {$\hattheta_{\beta\delta}$} [b] at 188 318
\pinlabel {$\hattheta_{\beta\gamma}$} [l] at 175 235
\pinlabel {$\hattheta^-_{\beta\delta}$} [r] at 500 285
\pinlabel {$\dom_1$} [tr] at 24 328
\pinlabel {$\dom_2$} [l] at 392 377
\pinlabel {$\dom_3$} [r] at 526 385
\pinlabel {$z$} [l] at 469 214
\pinlabel {$\gamma$} [b] at 193 158
\pinlabel {$\delta$} [l] at 160 69
\pinlabel {$\beta$} [br] at 122 67
\endlabellist
\includegraphics[height=7cm]{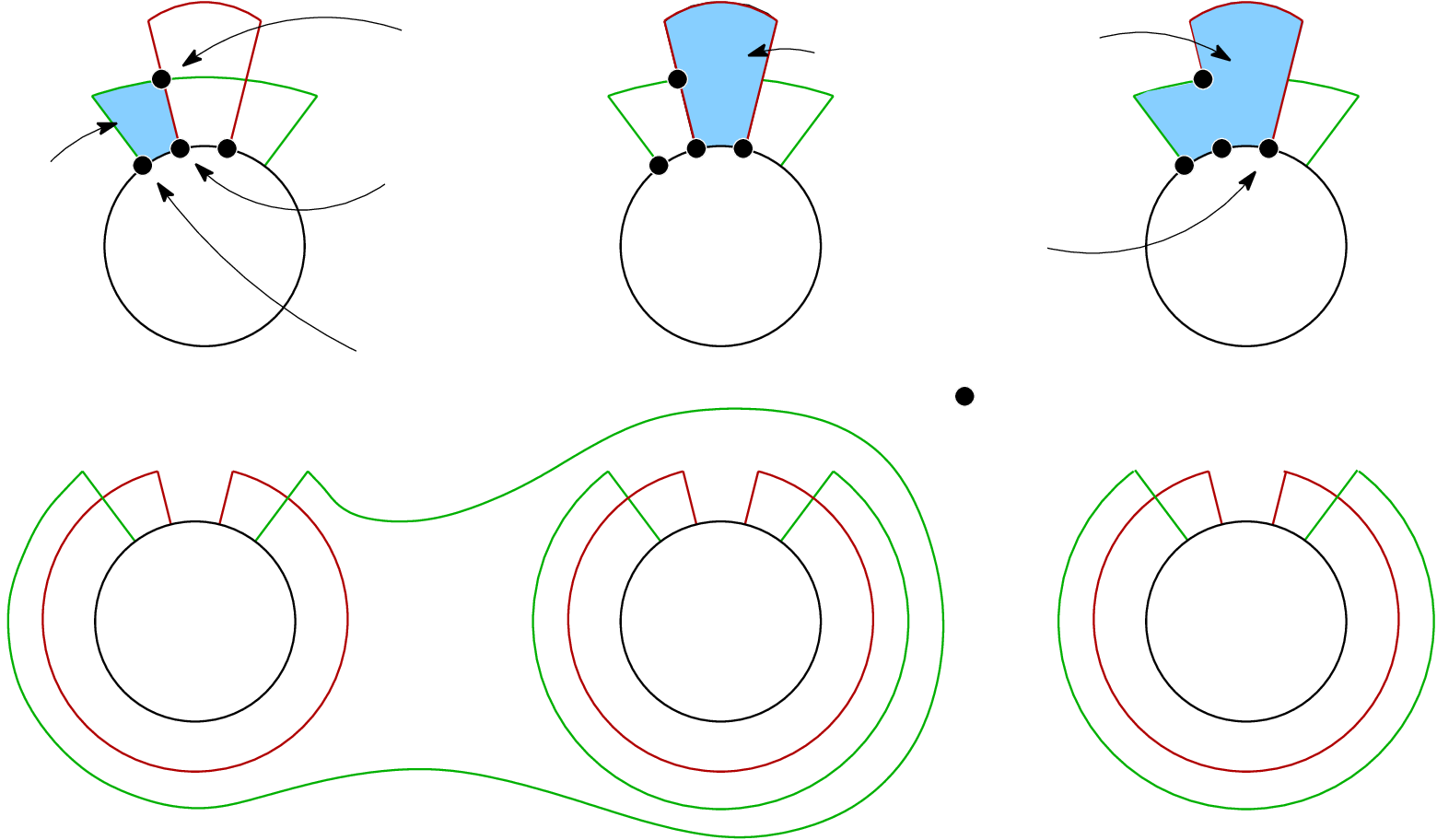}
\caption{The Heegaard surface cut open along the $\beta$-curves.}
\label{Fig:figTwo}
\end{figure}
}

%
%
\newcommand{\pictureTHREE}
{
\begin{figure}[ht!]
\centerline{\psfig{file=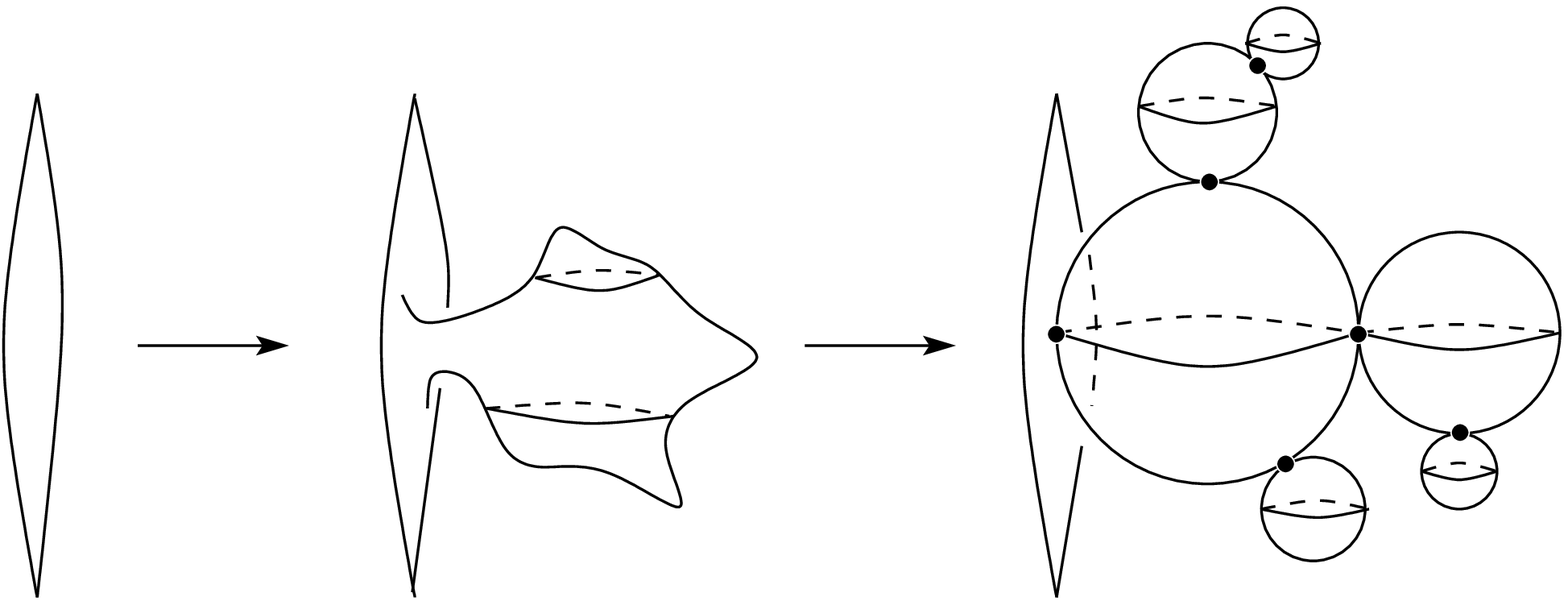,height=3cm}}
\caption{Bubbling of spheres.}
\label{Fig:figThree}
\end{figure}
}

%
%
\newcommand{\pictureFOUR}
{\begin{figure}[ht!]
\labellist\small\hair 2pt
\pinlabel {$n$} [Bl] at 3 182
\pinlabel {$n\!+\!\mu$} [Bl] at 93 182
\pinlabel {$n$} [Bl] at 175 182
\pinlabel {$n$} [Bl] at 323 182
\pinlabel {$n$} [Bl] at 463 182
\pinlabel {$n$} [Bl] at 624 182
\pinlabel {$n$} [Bl] at 723 182
\pinlabel {$K$} [l] at 3 8
\pinlabel {$K$} [l] at 93 8
\pinlabel {$K$} [l] at 175 8
\pinlabel {$K$} [l] at 323 8
\pinlabel {$K$} [l] at 463 8
\pinlabel {$K$} [l] at 624 8
\pinlabel {$K$} [l] at 723 8
\pinlabel {$-1$} [t] at 143 80
\pinlabel {$0$} [t] at 294 80
\pinlabel {$-1$} [t] at 431 80
\pinlabel {$-1$} [t] at 595 80
\pinlabel {$\mu$} [B] at 143 101
\pinlabel {$\mu$} [B] at 294 101
\pinlabel {$\mu$} [B] at 431 101
\pinlabel {$\mu$} [B] at 595 101
\pinlabel {$\nu$} [l] at 516 110
\pinlabel {$-1$} [Bl] at 495 127
\pinlabel {$\nu$} [l] at 679 110
\pinlabel {$0$} [Bl] at 663 127
\endlabellist
\centering
\includegraphics[height=3cm]{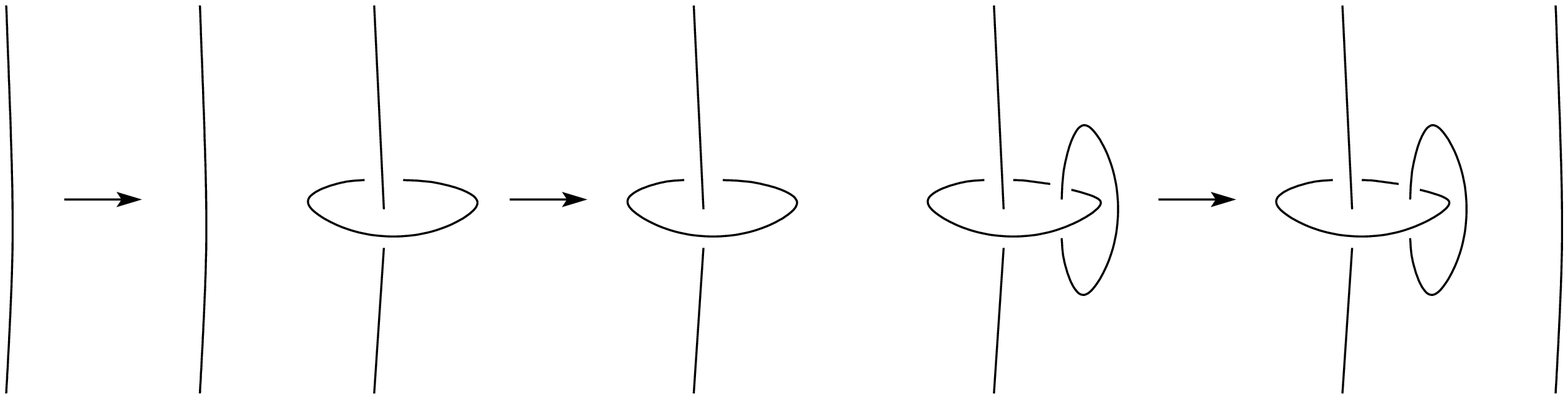}
\caption{The topological situation in the exact triangle.}
\label{Fig:figFour}
\end{figure}
}

%
%
\newcommand{\pictureFIVE}
{\epsfig{file=splitting02,height=4cm,angle=0}}

%
%
\newcommand{\pictureSIX}
{\begin{figure}[ht!]
\centerline{\psfig{file=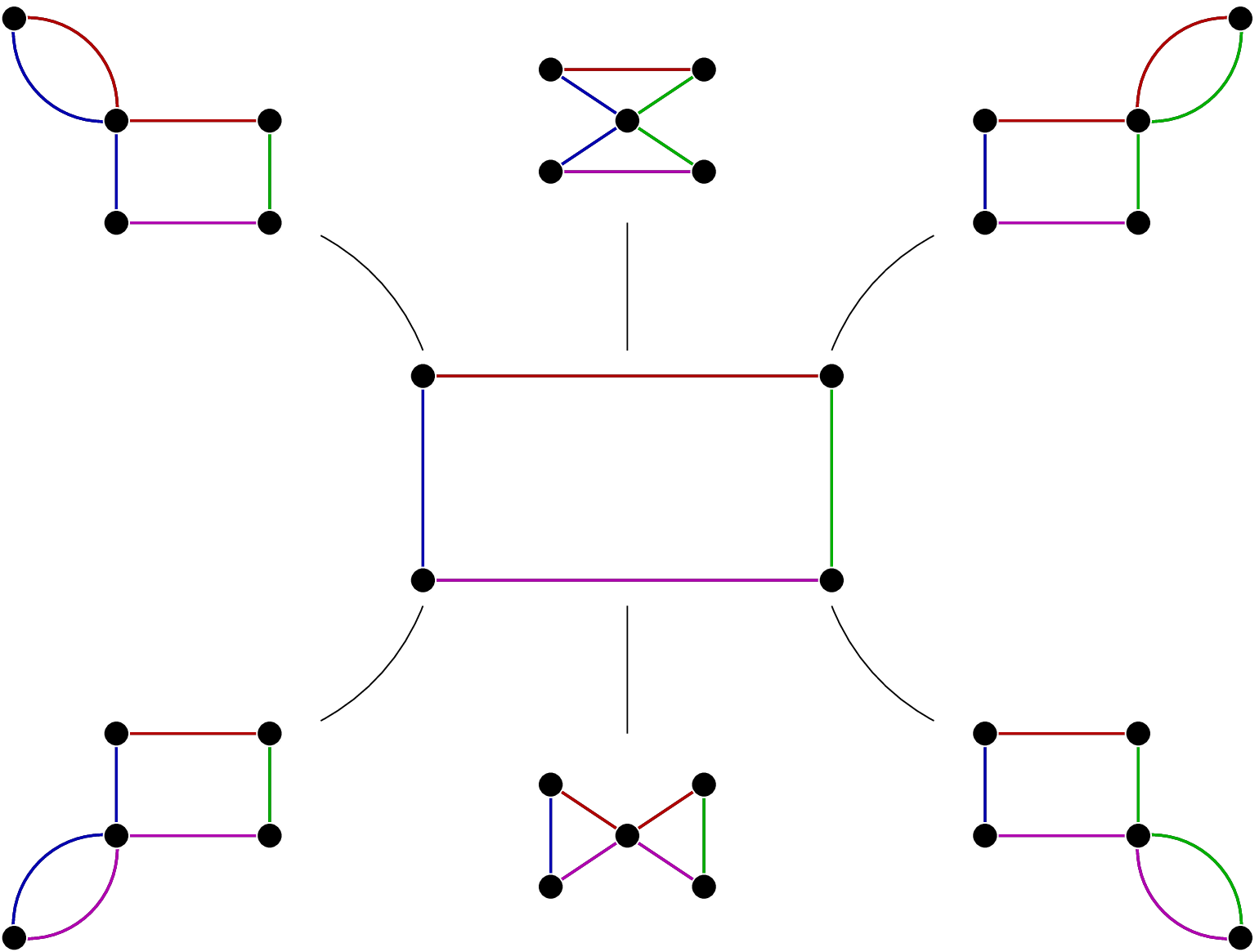,height=4cm}}
\caption{Ends of the moduli space of holomorphic rectangles.}
\label{Fig:figSix}
\end{figure} 
}

%
%
\newcommand{\pictureSEVEN}
{\begin{figure}[ht!]
\labellist\small\hair 2pt
\pinlabel {$\talpha$} [r] at 10 65
\pinlabel {$\tbeta$} [B] at 79 110
\pinlabel {$\tgamma$} [t] at 79 25
\pinlabel {$\tdelta$} [l] at 155 65
\endlabellist
\centering
\includegraphics[height=2cm]{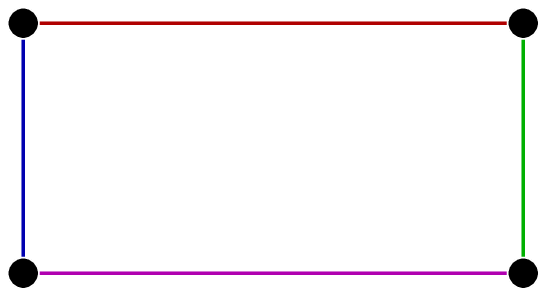}
\caption{The boundary conditions of rectangles for the definition of $H$.}
\label{Fig:figSeven}
\end{figure}
}

%
%
\newcommand{\pictureEIGHT}
{\begin{figure}[ht!]
\labellist\small\hair 2pt
\pinlabel {Page $P\!\times\!\{1/2\}$ of the open book} [bl] at 29 187
\pinlabel {$z$} [bl] at 189 112
\pinlabel {$a_i$} [t] at 76 22
\pinlabel {$b_i$} [t] at  153 22
\endlabellist
\centering
\includegraphics[height=3cm]{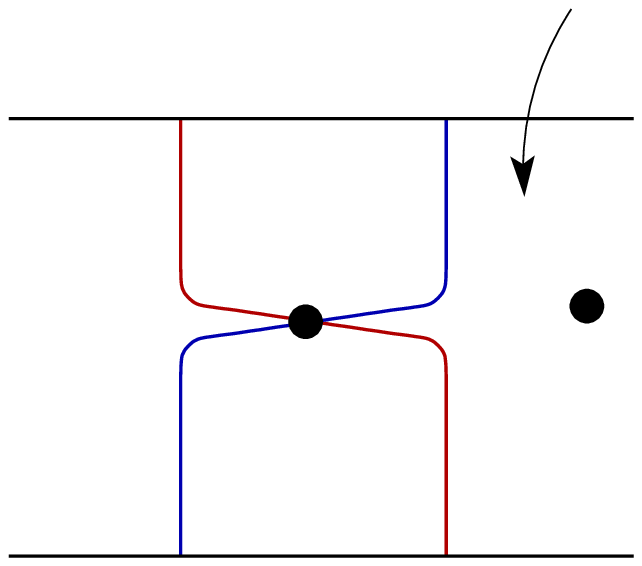}
\caption{Positioning of the point $z$ and choice of $b_i$.}
\label{Fig:figzpoint}
\end{figure}}

%
%
\newcommand{\pictureNINE}
{\begin{figure}[ht!]
\labellist\small\hair 2pt
\pinlabel {$\alpha$} [l] at 279 334
\pinlabel {$x$} [Bl] at 283 244
\pinlabel {$z$} [l] at 422 219
\pinlabel {$\dom_1$} [l] at 302 183
\pinlabel {$\dom_2$} [l] at 200 310
\pinlabel {$y$} [tl] at 284 121
\pinlabel {$\beta$} [r] at 146 44
\endlabellist
\centering
\includegraphics[width=6cm]{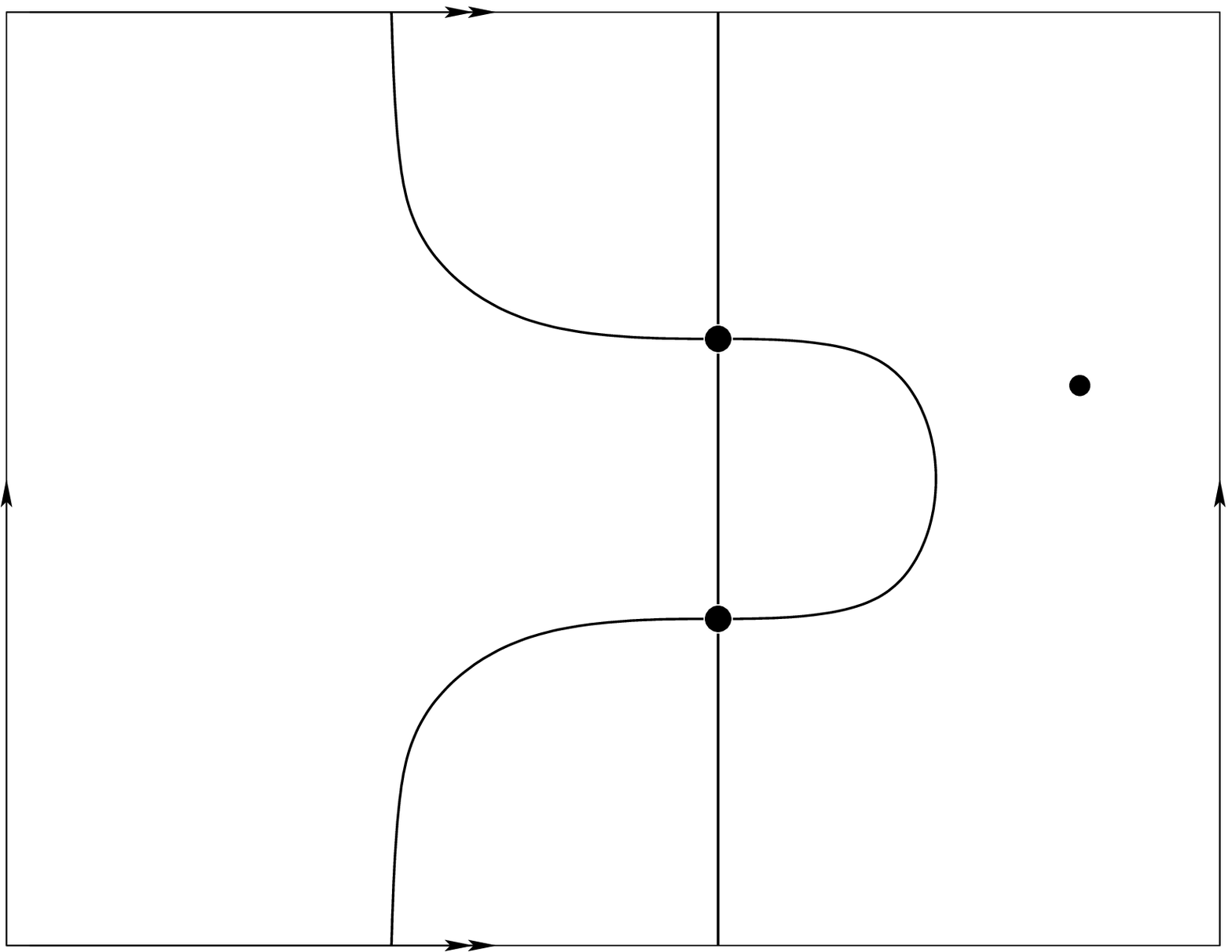}
\caption{An admissible Heegaard diagram for $\stwo\times\sone$.}
\label{Fig:figNine}
\end{figure}}

%
%
\newcommand{\pictureTEN}
{\begin{figure}[ht!]
\labellist\small\hair 2pt
\pinlabel {$z$} [l] at 240 141
\pinlabel {$x_1$} [l] at 94 116
\pinlabel {$y_1$} [l] at 385 120
\pinlabel {$x_2$} [l] at 89 74
\pinlabel {$y_2$} [l] at 380 74
\pinlabel {$\dom_1$} [t] at 161 26
\pinlabel {$\dom_2$} [l] at 105 10
\pinlabel {$\dom_3$} [tr] at 265 68
\pinlabel {$\dom_4$} [r] at 320 12
\endlabellist
\centering
\includegraphics[width=12cm]{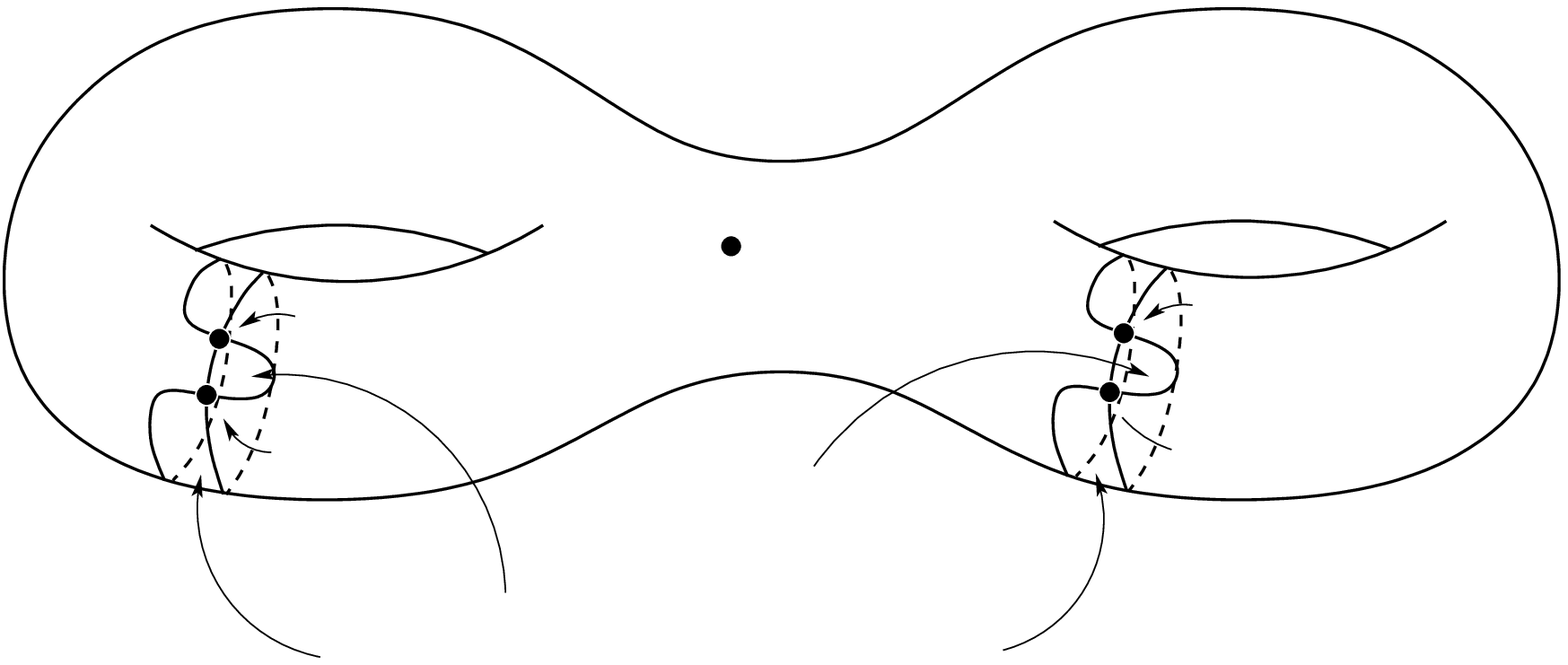}
\caption{An admissible Heegaard diagram for $\stwo\times\sone\#\stwo\times\sone$.}
\label{Fig:figTen}
\end{figure}}

%
%
\newcommand{\pictureELEVEN}
{\begin{figure}[ht!]
\labellist\small\hair 2pt
\pinlabel {$x'_1$} [B] at 149 230
\pinlabel {$z$} [r] at 201 239
\pinlabel {$\dom_z$} [l] at 255 242
\pinlabel {$\gamma_1$} [B] at 272 167
\pinlabel {$\beta_1$} [t] at 12 129
\pinlabel {$\hattheta_1$} [t] at 97 129
\pinlabel {$x_1$} [t] at 147 73
\pinlabel {$\alpha_1$} [t] at 283 97
\pinlabel {Domain of a holomorphic triangle} [t] at 133 23
\pinlabel {$_1$} [t] at 38 208
\pinlabel {$_2$} [l] at 70 236
\endlabellist
\centering
\includegraphics[height=4.5cm]{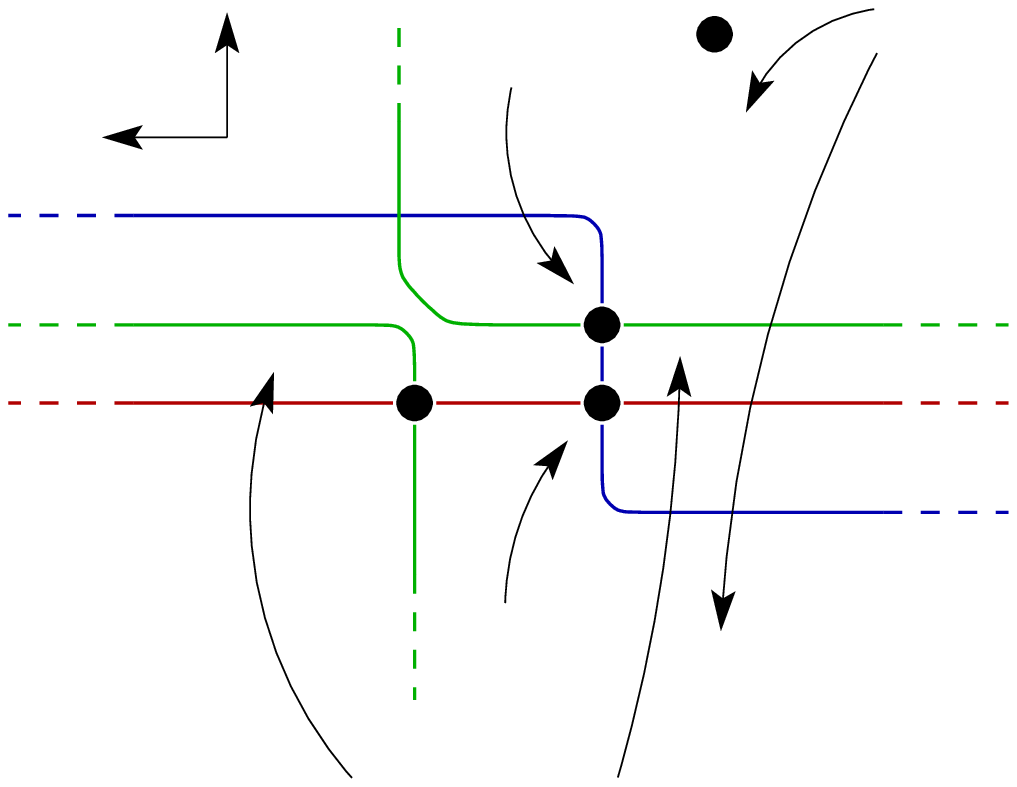}
\caption{Significant part of the Heegaard triple diagram.}
\label{Fig:figEleven}
\end{figure}}

%
%
\newcommand{\pictureFOURTEEN}
{\begin{figure}[ht!]
\labellist\small\hair 2pt
\pinlabel {$w$} [tr] at 7 34
\pinlabel {$x$} [B] at 126 235
\pinlabel {$y$} [tl] at 247 27
\pinlabel {$\talpha$} [Br] at 71 136
\pinlabel {$\tbeta$} [Bl] at 191 136
\pinlabel {$\tgamma$} [t] at 126 26
\endlabellist
\centering
\includegraphics[width=4cm]{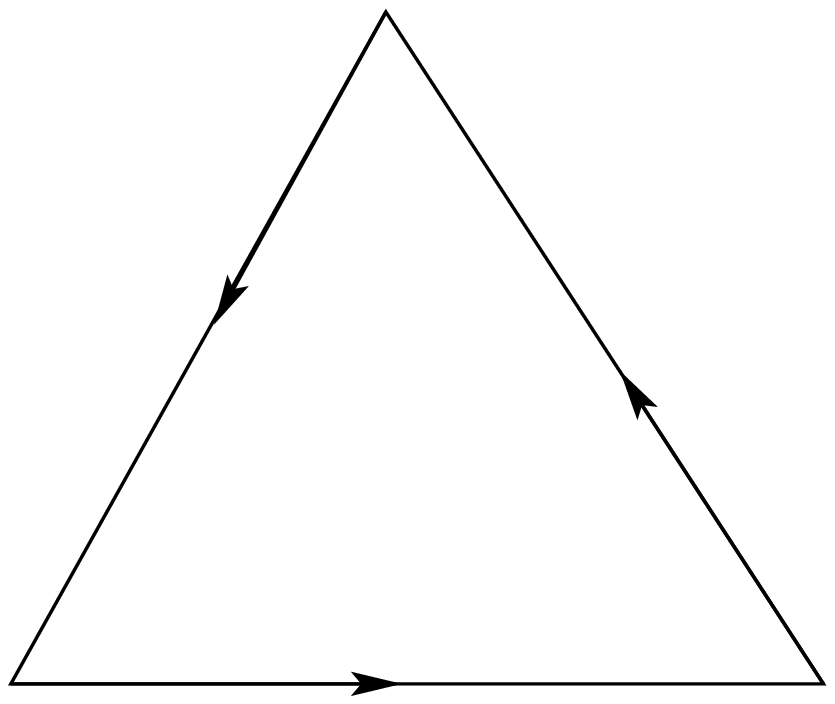}
\caption{A Whitney triangle and its boundary conditions.}
\label{Fig:figFourteen}
\end{figure}}

%
%
\newcommand{\pictureNINETEEN}
{\begin{figure}[ht!]
\labellist\small\hair 2pt
\pinlabel {$\hattheta_{\beta\gamma}$} [B] at 268 202
\pinlabel {$\hattheta_{\gamma\delta}$} [B] at 348 228
\pinlabel {$z$} [l] at 233 162
\pinlabel {$\delta_1$} [t] at 4 52
\pinlabel {$\beta_1$} [l] at 95 16
\pinlabel {$\hattheta_{\gamma\delta}$} [l] at 145 10
\pinlabel {$\gamma_1$} [t] at 183 45
\pinlabel {$\beta_2$} [t] at 254 37
\pinlabel {$\gamma_2$} [t] at 275 32
\pinlabel {$\delta_2$} [t] at 294 20
\endlabellist
\centering
\includegraphics[width=12cm]{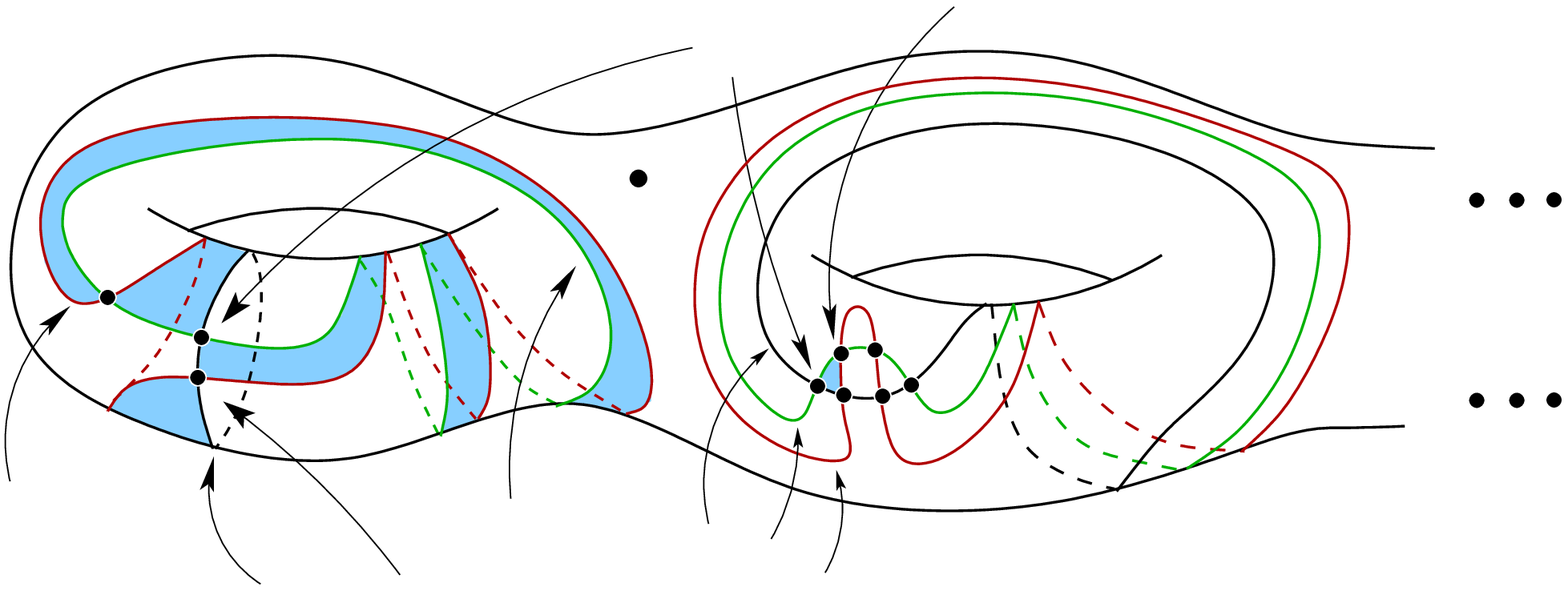}
\caption{Heegaard triple diagram for computation of 
$\fhat_{\beta\gamma\delta}(\hattheta_{\beta\gamma}\otimes
\hattheta_{\gamma\delta})$. }
\label{Fig:figNineteen}
\end{figure}
}